\pdfoutput=1
\documentclass[12pt,a4paper]{amsart}

\calclayout

\usepackage[utf8]{inputenc}
    \usepackage[T1]{fontenc}
    \usepackage{libertine}
    \usepackage[libertine,cmintegrals,cmbraces]{newtxmath}
    
    \usepackage{stackengine,amsmath,amsthm,amssymb,mathtools}
    \usepackage{mathdots}
    \usepackage{graphicx}
    \usepackage{tikz,tikz-cd}
            \tikzset{
                subseteq/.style={
                draw=none,
                edge node={node [sloped, allow upside down, auto=false]{$\subseteq$}}},
                Subseteq/.style={
                draw=none,
                every to/.append style={
                edge node={node [sloped, allow upside down, auto=false]{$\subseteq$}}}
                }
            }
            \tikzset{
                equal/.style={
                draw=none,
                edge node={node [sloped, allow upside down, auto=false]{$=$}}},
                Equal/.style={
                draw=none,
                every to/.append style={
                edge node={node [sloped, allow upside down, auto=false]{$=$}}}
                }
            }
            \tikzset{
            rotated_label/.style={anchor=south, rotate=90, inner sep=.5mm}
            }
    \usetikzlibrary{decorations}
    \usepackage{verbatim}
    \usepackage{booktabs}
    \usepackage{array}
    \usepackage{tabularx}
    \usepackage{booktabs}
    \usepackage{siunitx}
    \usepackage{rotating}
    \usepackage{adjustbox}
    \usepackage{longtable}
    \usepackage{caption} 
    \usepackage{stmaryrd}
    \usepackage{nicefrac}
    \usepackage{enumitem}
    \usepackage{xfrac}
    \usepackage[textwidth=0.8in,textsize=tiny]{todonotes}
\usepackage[backref=true,giveninits=true,doi=false,url=false,isbn=false,backend=biber]{biblatex}
    \newbibmacro{string+doiurlisbn}[1]{%
      \iffieldundef{doi}{%
        \iffieldundef{url}{%
          \iffieldundef{isbn}{%
            \iffieldundef{issn}{%
              #1%
            }{%
              \href{http://books.google.com/books?vid=ISSN\thefield{issn}}{#1}%
            }%
          }{%
            \href{http://books.google.com/books?vid=ISBN\thefield{isbn}}{#1}%
          }%
        }{%
          \href{\thefield{url}}{#1}%
        }%
      }{%
        \href{http://dx.doi.org/\thefield{doi}}{#1}%
      }%
    }
    \DeclareFieldFormat{title}{\usebibmacro{string+doiurlisbn}{\mkbibemph{#1}}}
    \DeclareFieldFormat[article,incollection,thesis,misc,inproceedings,online,inbook]{title}{\usebibmacro{string+doiurlisbn}{\mkbibquote{#1}}}
    \DeclareFieldFormat{year}{\usebibmacro{year-parenthesis}{\mkbibemph{#1}}}
\renewbibmacro*{date}{%
\iffieldundef{year}
{}
{\ifentrytype{misc}{\printtext[parens]{\printdate}}{\printdate}}}%
\usepackage{hyperref}
    \hypersetup{colorlinks=true,allcolors=blue}
    \usepackage[capitalize,nameinlink,noabbrev]{cleveref}

\usepackage{amsthm}
    \theoremstyle{plain}
        \newtheorem{theorem}{Theorem}[section]
        \newtheorem{lemma}[theorem]{Lemma}
        \newtheorem{proposition}[theorem]{Proposition}
        \newtheorem{corollary}[theorem]{Corollary}

    \theoremstyle{definition}
        \newtheorem{definition}[theorem]{Definition}

    \theoremstyle{remark}

        \newtheorem{remark}[theorem]{Remark}

 \DeclareMathOperator{\GL}{GL}

\DeclareMathOperator{\Gal}{Gal}

\DeclareMathOperator{\End}{End}
\DeclareMathOperator{\Aut}{Aut}

\newcommand{\F}{\mathbb{F}}

\newcommand{\N}{\mathbb{N}}
\newcommand{\Z}{\mathbb{Z}}
\newcommand{\Q}{\mathbb{Q}}

\newcommand{\C}{\mathbb{C}}

\newcommand{\Pgotic}{\mathfrak{P}}
\newcommand{\pgotic}{\mathfrak{p}}

\newcommand{\Ogotic}{\mathcal{O}}

\newcommand{\fgotic}{\mathfrak{f}}
\newcommand{\fgoticO}{\mathfrak{f}_\mathcal{O}}

\newcolumntype{C}[1]{>{\centering\arraybackslash}p{#1}}

\newcommand\restr[2]{{ \left.\kern-\nulldelimiterspace #1 \vphantom{\big|} \right|_{#2} }}

\newcommand\altxrightarrow[2][0pt]{\mathrel{\ensurestackMath{\stackengine%
  {\dimexpr#1-7.5pt}{\xrightarrow{\phantom{#2}}}{\scriptstyle\!#2\,}%
  {O}{c}{F}{F}{S}}}}

\makeatletter
\@namedef{subjclassname@2020}{\textup{2020} Mathematics Subject Classification}
\makeatother

\title[Entanglement in the family of division fields of CM elliptic curves]{Entanglement in the family of division fields of elliptic curves with complex multiplication}
    \author{Francesco Campagna}
    \author{Riccardo Pengo}
    \address{Francesco Campagna - Department of Mathematical Sciences, University of Copenhagen, Universitetsparken 5,
    2100 Copenhagen, Denmark}
    \email{\href{mailto:campagna@math.ku.dk}{campagna@math.ku.dk}}
    \address{Riccardo Pengo - École normale supérieure de Lyon, Unit\'e de Math\'ematiques Pures et Appliqu\'ees, 46 all\'ee d'Italie, 69007 Lyon, France}
    \email{\href{mailto:riccardo.pengo@ens-lyon.fr}{riccardo.pengo@ens-lyon.fr}}
    \date{\today}

\begin{document}
    \maketitle
    \begin{abstract}
        For every elliptic curve $E$ which has complex multiplication (CM) and is defined over a number field $F$ containing the CM field $K$, we prove that the family of
        $p^{\infty}$-division fields of $E$, with $p \in \mathbb{N}$ prime, becomes linearly disjoint over $F$ after removing an explicit finite subfamily of fields.
        We then give a necessary condition for this finite subfamily to be entangled over $F$, which is always met when $F = K$. 
        In this case, and under the further assumption that the elliptic curve $E$ is obtained as a base-change from $\mathbb{Q}$, we describe in detail the entanglement in the family of division fields of $E$. 
    \end{abstract}

%\vspace{3\baselineskip}
\section{Introduction}

Let $E$ be an elliptic curve defined over a number field $F$, and let $\overline{F} \supseteq F$ be a fixed algebraic closure. The absolute Galois group $\Gal(\overline{F}/F)$ acts on the group $E_{\text{tors}} := E(\overline{F})_{\text{tors}}$ of all torsion points of $E$, giving rise to a Galois representation
\[
    \rho_E: \Gal(F(E_{\text{tors}})/F) \hookrightarrow \Aut_{\Z}(E_{\text{tors}}) \cong \GL_2(\widehat{\Z})
\]
where $F(E_{\text{tors}})$ is the compositum of the family of fields $\{F(E[p^{\infty}])\}_{p}$ for $p \in \mathbb{N}$ prime.
Each extension $F \subseteq F(E[p^{\infty}])$ is in turn defined as the compositum of the family $\{F(E[p^n])\}_{n \in \mathbb{N}}$, where, for every $N \in \mathbb{N}$, we denote by $F(E[N])$ the \textit{division field} obtained by adjoining to $F$ the coordinates of all the points belonging to the $N$-torsion subgroup $E[N] := E[N](\overline{F})$.

For an elliptic curve $E$ without complex multiplication (CM), Serre's Open Image Theorem \cite[Théorème~3]{Serre_1971} asserts that the image of $\rho_E$ has finite index in $\GL_2(\widehat{\Z})$. 
However, explicitly describing this image is a non-trivial problem in general which is connected to the celebrated Uniformity Conjecture \cite[\S~4.3]{Serre_1971}. 
A first step in this direction is to study the \textit{entanglement} in the family $\{F(E[p^{\infty}])\}_p$ for $p$ prime, \textit{i.e.} to describe the image of the natural inclusion
\begin{equation} \label{eq:linear_disjointness_map}
    \Gal(F(E_{\text{tors}})/F) \hookrightarrow \prod_p \Gal(F(E[p^{\infty}])/F)
\end{equation}
where the product runs over all primes $p \in \N$.
For each non-CM elliptic curve $E_{/F}$ this has been addressed in \cite{Campagna_Stevenhagen} by Stevenhagen and the first named author. 
More precisely, they identify an explicit finite set $S$ of ``bad primes'' (depending on $E$ and $F$) such that the map \eqref{eq:linear_disjointness_map} induces an isomorphism
\[
\begin{tikzcd}[column sep = 0.8cm]
    \Gal(F(E_{\text{tors}})/F) \arrow[r,"\thicksim"] & \Gal(F(E[S^{\infty}])/F) \times \displaystyle\prod_{p \not\in S} \Gal(F(E[p^{\infty}])/F)
\end{tikzcd}
\]
where $F(E[S^{\infty}])$ denotes the compositum of the family of fields $\{F(E[p^{\infty}])\}_{p \in S}$.
In this case one says that the family $\{F(E[S^{\infty}])\} \cup \{F(E[p^{\infty}])\}_p$ is \textit{linearly disjoint} over $F$.
The first goal of this paper is to prove the following analogous statement for CM elliptic curves. 

\begin{theorem} \label{thm:disjointness_intro}
Let $F$ be a number field and $E_{/F}$ an elliptic curve with complex multiplication by an order $\mathcal{O}$ in an imaginary quadratic field $K \subseteq F$. Denote by $B_E := \fgotic_\mathcal{O} \, \Delta_F \, N_{F/\Q}(\fgotic_E) \in \mathbb{Z}$ the product of the conductor $\fgotic_\mathcal{O} := \lvert \mathcal{O}_K \colon \mathcal{O} \rvert \in \mathbb{N}$ of the order $\mathcal{O}$, the absolute discriminant $\Delta_F \in \mathbb{Z}$ of the number field $F$ and the absolute norm $N_{F/\Q}(\fgotic_E) := \lvert \mathcal{O}_F/\mathfrak{f}_E \rvert \in \mathbb{N}$ of the conductor ideal $\fgotic_E \subseteq \mathcal{O}_F$ of $E$.
Then the map \eqref{eq:linear_disjointness_map} induces an isomorphism
\begin{equation} \label{eq:iso_entanglement}
    \begin{tikzcd}[column sep = 0.8cm]
    \Gal(F(E_{\text{tors}})/F) \arrow[r,"\thicksim"] & \Gal(F(E[S^{\infty}])/F) \times \displaystyle\prod_{p \not\in S} \Gal(F(E[p^{\infty}])/F)
\end{tikzcd}
\end{equation}
for any finite set of primes $S \subseteq \N$ containing the primes dividing $B_E$.
\end{theorem}

The key ingredients involved in our proof of \cref{thm:disjointness_intro} are \cref{prop:unramified,prop:ramification}, which we establish by using some results concerning formal groups attached to CM elliptic curves, recalled in \cref{sec:preliminaries}.
Without further assumptions, the isomorphism \eqref{eq:iso_entanglement} does not hold if the set $S$ does not contain all the primes dividing $B_E$, as we point out in \cref{rmk:minimality_of_S_over_Q}. 
However, the following theorem provides a sufficient condition under which \eqref{eq:iso_entanglement} holds with $S = \emptyset$.

\begin{theorem} \label{thm:infinitely_many_linearly_disjoint}
    Let $\mathcal{O}$ be an order in an imaginary quadratic field $K$, and $E$ be an elliptic curve with complex multiplication by $\mathcal{O}$ defined over the ring class field $H_\mathcal{O} := K(j(E))$, such that $H_\mathcal{O}(E_\text{tors}) \neq K^\text{ab}$. Then the whole family of $p^\infty$-division fields $\{ H_\mathcal{O}(E[p^\infty]) \}_p$, where $p \in \mathbb{N}$ runs through the rational primes, is linearly disjoint over $H_\mathcal{O}$.
    Moreover, if $\operatorname{Pic}(\mathcal{O}) \neq \{1\}$, there exist infinitely many such elliptic curves $E_{/H_{\mathcal{O}}}$, which are non-isomorphic over $H_\Ogotic$.
\end{theorem}

\cref{thm:infinitely_many_linearly_disjoint} is the outcome of a study, carried out in \cref{sec:minimality}, concerning the minimality properties of division fields associated to elliptic curves $E$, defined over the ring class field $H_\mathcal{O} = K(j(E))$, which have complex multiplication by $\mathcal{O}$.
More precisely, as explained in \cref{sec:ray_class_fields_imaginary_quadratic_orders}, every division field $H_{\mathcal{O}}(E[N])$
is an extension, of degree at most $\lvert \mathcal{O}^\times \rvert$, of a specific \textit{ray class field} associated to the order $\mathcal{O}$ and the integer $N$.
The general aim of \cref{sec:minimality} is then to determine which division fields are precisely equal to the corresponding ray class fields.
To this end, \cref{thm:coates_wiles} provides an explicit infinite family of such division fields for any elliptic curve whose torsion points generate abelian extensions of $K$.
Moreover, \cref{thm:Shimura_condition_yes} determines for which orders $\mathcal{O}$ such an elliptic curve exists, and explicitly constructs infinitely many of them whenever possible.
Notice that these elliptic curves are exactly the ones not meeting the hypotheses of \cref{thm:infinitely_many_linearly_disjoint}.

In the final \cref{sec:entanglement} we focus on elliptic curves having complex multiplication by orders of class number one.
In particular, we use \cref{thm:disjointness_intro,thm:coates_wiles} to prove \cref{thm:classification_entanglement}, which provides a complete description of the image of \eqref{eq:linear_disjointness_map} when $F = K$ is an imaginary quadratic field and $E_{/K}$ is the base-change of an elliptic curve defined over $\Q$.
An immediate consequence of this classification is given by the following statement.

\begin{theorem} \label{thm:entanglement_over_Q_intro}
    Let $\mathcal{O}$ be an order of discriminant $\Delta_\mathcal{O} < -4$ inside an imaginary quadratic field $K$, and suppose that $\operatorname{Pic}(\mathcal{O}) = \{1\}$.
    Let $E_{/\mathbb{Q}}$ be an elliptic curve with complex multiplication by $\mathcal{O}$.
    Then the family of division fields $\{K(E[p^\infty])\}_{p}$, where $p$ runs over the rational primes $p \in \mathbb{N}$, is linearly disjoint over $K$ if and only if $E$ is isomorphic over $K$ to one of the thirty elliptic curves appearing either in \cref{tab:twist_minimal_CM_elliptic_curves} or in \eqref{eq:pristine_curves_2i}.
\end{theorem}

To conclude this introduction, we point out that \cref{sec:appendix_ray_class_fields} develops a general framework for the study of ray class fields associated to any order $\mathcal{O}$ inside a number field $F$, which can be of independent interest.
We also remark that our work, despite having different objectives, bears a connection with the work of Bourdon and Clark \cite{Bourdon_Clark_2020} and of Lozano-Robledo \cite{Lozano-Robledo_2019}.
We comment more punctually on this in \cref{rmk:lozano_robledo_arXiv,rmk:lozano-robledo_twist,rmk:bourdon_clark_1-5,rmk:lozano_robledo_1728,rmk:bourdon_clark_1-8}.

% \cref{rmk:lozano_robledo_arXiv, rmk:lozano-robledo_twist,rmk:bourdon_clark_1-5, rmk:lozano_robledo_1728, rmk:bourdon_clark_1-8}.

The results contained in this article are applied by the authors in two different ways. The first named author uses \cref{thm:disjointness_intro} in \cite{Campagna_Stevenhagen_CM} to study, jointly with Stevenhagen, cyclic reduction of CM elliptic curves. 
The second named author uses \cref{thm:disjointness_intro} in \cite{Pengo} to investigate the Mahler measure of certain explicit planar models of CM elliptic curves defined over $\Q$.

\section{Formal groups and elliptic curves}
\label{sec:preliminaries}

\subsection{Formal groups}
\label{sec:formal_groups}

The aim of this subsection is to recall, following \cite[Chapter~IV]{si09}, some of the main points of the theory of one dimensional, commutative formal group laws defined over a ring $R$, which we call \textit{formal groups} for short.
Roughly speaking, these are power series $\mathcal{F} \in R\llbracket z_1,z_2 \rrbracket$ for which the association $x +_{\mathcal{F}} y := \mathcal{F}(x,y)$ behaves like an abelian group law.

Given a formal group $\mathcal{F} \in R\llbracket z_1, z_2 \rrbracket$ we denote the set of endomorphisms of $\mathcal{F}$ by
\[
    \End_R(\mathcal{F}) := \{ f \in t R\llbracket t \rrbracket \, \mid \, f(x +_{\mathcal{F}} y) = f(x) +_{\mathcal{F}} f(y) \}
\]
which is a ring under the operations $(f +_{\mathcal{F}} g)(t) := \mathcal{F}(f(t),g(t))$ and $(g \circ f)(t) := g(f(t))$.
We write $\operatorname{Aut}_R(\mathcal{F})$ for the unit group $\End_R(\mathcal{F})^{\times}$ and we denote by $[\cdot]_{\mathcal{F}}$ the unique ring homomorphism $\mathbb{Z} \to \End_R(\mathcal{F})$.
For every $\phi \in \End_R(\mathcal{F})$ one has that $\phi \in \operatorname{Aut}_R(\mathcal{F})$ if and only if $\phi'(0) \in R^{\times}$ where $\phi'(t) \in R\llbracket t \rrbracket$ denotes the formal derivative (see \cite[Chapter~IV, Lemma~2.4]{si09}).
Moreover, every $\phi \in \End_R(\mathcal{F})$ is uniquely determined by $\phi'(0)$ whenever $R$ is torsion-free as an abelian group. 
More precisely, there exist two power series $\exp_{\mathcal{F}}, \log_{\mathcal{F}} \in (R \otimes_{\mathbb{Z}} \mathbb{Q})\llbracket t \rrbracket$ such that
\begin{equation} \label{eq:derivative_formal_group}
\phi(t) = \exp_{\mathcal{F}}(\phi'(0) \cdot \log_{\mathcal{F}}(t))   
\end{equation} 
as explained in \cite[Chapter~IV, \S~5]{si09}.

Let us now recall that if $(R,\mathfrak{m})$ is a complete local ring there is a well defined map
\[ \label{eq:group_of_points_of_formal_group}
    \begin{aligned}
    \mathfrak{m} \times \mathfrak{m} &\xrightarrow{+_{\mathcal{F}}} \mathfrak{m} \\
    (x,y) &\mapsto \mathcal{F}(x,y)
\end{aligned} 
\]
endowing the set $\mathfrak{m}$ with the structure of an abelian group, which will be denoted by $\mathcal{F}(\mathfrak{m})$.
We will sometimes refer to $\mathcal{F}(\mathfrak{m})$ as the \textit{group of $\mathfrak{m}$-points of $\mathcal{F}$}.
Every $\phi \in \End_R(\mathcal{F})$ induces an endomorphism $\phi_{\mathfrak{m}} \colon \mathcal{F}(\mathfrak{m}) \to \mathcal{F}(\mathfrak{m})$, and 
for every subset $\Phi \subseteq \End_R(\mathcal{F})$ we define the $\Phi$-torsion subgroup $\mathcal{F}(\mathfrak{m})[\Phi] \subseteq \mathcal{F}(\mathfrak{m})$ as
\[
    \mathcal{F}(\mathfrak{m})[\Phi] := \bigcap_{\phi \in \Phi} \ker(\phi_{\mathfrak{m}}).
\]
These $\Phi$-torsion subgroups generalise the usual $N$-torsion subgroups
$\mathcal{F}(\mathfrak{m})[N] \subseteq \mathcal{F}(\mathfrak{m})$ defined for every $N \in \mathbb{Z}$.
The following lemma provides some information about the behaviour of $\mathcal{F}(\mathfrak{m})[p^n]$ under finite extensions of local rings with residue characteristic $p$.

\begin{lemma}[see {\cite[Chapter~IV, Exercise~4.6]{si09}} and {\cite[Page~15]{silverman-errata}}] \label{lem:silverman_exercise_formal_groups}
    Let $R \subseteq S$ be a finite extension of complete discrete valuation rings of characteristic zero with maximal ideals $\mathfrak{m}_R \subseteq \mathfrak{m}_S$ and residue fields $\kappa_R \subseteq \kappa_S$. 
    Let $p := \operatorname{char}(\kappa_R) > 0$ be the residue characteristic of $R$ and $S$, and suppose that $\mathfrak{m}_R = p R$.
    Then for every formal group $\mathcal{F} \in R\llbracket z_1, z_2 \rrbracket$ and every $x \in \mathcal{F}(\mathfrak{m}_S)[p^{n}] \setminus \mathcal{F}(\mathfrak{m}_S)[p^{n - 1}]$ with $n \in \Z_{\geq 1}$ we have that 
    \[
        v_S(x) \leq \frac{v_S(p)}{p^{h (n - 1)} \cdot (p^h - 1)}
    \]
    where $v_S$ denotes the normalised valuation on $S$, and
    \[
        h = \operatorname{ht}(\overline{\mathcal{F}}) := \max\left\{ \, n \in \mathbb{N} \ \middle| \ [p]_{\overline{\mathcal{F}}} \in \kappa_R\llbracket t^{p^n} \rrbracket \, \right\}
    \]
    is the height of the reduced formal group $\overline{\mathcal{F}} \in \kappa_R\llbracket z_1, z_2 \rrbracket$.
\end{lemma}
\begin{proof}
    Using that $h = \operatorname{ht}(\overline{\mathcal{F}})$ and that $\mathfrak{m}_R = p \cdot R$ we see that there exist $f,g \in R\llbracket t \rrbracket$ such that $[p]_{\mathcal{F}} = f(t^{p^h}) + p \, g(t)$. We can assume that $f,g \in t \, R\llbracket t \rrbracket$ and $g'(0) = 1$ because $[p]_{\mathcal{F}} \in t \, R\llbracket t \rrbracket$ and $[p]_{\mathcal{F}}'(0) = p$. 
    Now fix $x \in \mathcal{F}(\mathfrak{m}_S)[p^{n}] \setminus \mathcal{F}(\mathfrak{m}_S)[p^{n - 1}]$ and proceed by induction on $n \in \Z_{\geq 1}$.
    
    If $n = 1$ then $f(x^{p^h}) + p \, g(x) = [p]_{\mathcal{F}}(x) = 0$, hence $v_S(p) + v_S(g(x)) = v_S(f(x^{p^h}))$. 
    Now $v_S(g(x)) = v_S(x)$ because $g(0) = 0$ and $g'(0) = 1$, and $v_S(f(x^{p^h})) \geq v_S(x^{p^h}) = p^h \, v_S(x)$ 
    because $f(0) = 0$. 
    Hence $v_S(p) \geq (p^h - 1) \cdot v_S(x)$, which is what we wanted to prove.
    
    If $n \geq 2$ we know by induction that
    \[
        \frac{v_S(p)}{p^{h (n - 2)} \cdot (p^h - 1)} \geq v_S([p]_{\mathcal{F}}(x)) = v_S(f(x^{p^h}) + p \, g(x)) \geq \min(v_S(x^{p^h}),v_S(p x))
    \]
    because $[p]_{\mathcal{F}}(x) \in \mathcal{F}(\mathfrak{m}_S)[p^{n-1}] \setminus \mathcal{F}(\mathfrak{m}_S)[p^{n-2}]$. 
    This implies that $\min(v_S(x^{p^h}),v_S(p x)) = v_S(x^{p^h})$. Otherwise we would get the contradiction $v_S(p) \geq p^{h (n - 2)} \cdot (p^h - 1) \cdot v_S(p x) > v_S(p)$ because $n \geq 2$, $v_S(x) > 0$ and $h \geq 1$. 
    Hence we have that
    \[
        v_S(x) = \frac{v_S(x^{p^h})}{p^h} \leq \frac{v_S(p)}{p^h \cdot (p^{h (n - 2)} \cdot (p^h - 1))} = \frac{v_S(p)}{p^{h (n - 1)} \cdot (p^h - 1)}
    \]
    which is what we wanted to prove.
\end{proof}

\subsection{Formal groups and elliptic curves}
\label{sec:CM_formal_groups}

Given an elliptic curve $E$ defined over a number field $F$ by an integral Weierstrass equation one can construct, following for example \cite[Chapter~IV]{si09}, a formal group $\widehat{E} \in \mathcal{O}_F\llbracket z_1, z_2 \rrbracket$ which can be thought of as the formal counterpart of the addition law on $E$.
The association $E \mapsto \widehat{E}$ is functorial and in particular induces a map 
\begin{equation} \label{eq:phi_to_phi_hat}
    \begin{aligned}
    \End_F(E) &\to \End_F(\widehat{E}) \\
    \phi &\mapsto \widehat{\phi}
\end{aligned}
\end{equation}
between the endomorphism rings of $E$ and $\widehat{E}$.
The power series lying in the image of \eqref{eq:phi_to_phi_hat} have integral coefficients, as proved in the following theorem, due to Streng.
\begin{theorem}[see {\cite[Theorem~2.9]{st08}}] \label{thm:streng}
    Let $E$ be an elliptic curve defined over a number field $F$ and let $\widehat{E} \in \mathcal{O}_F\llbracket z_1, z_2 \rrbracket$ be the formal group law associated to a Weierstrass model of $E$ whose coefficients lie in $\mathcal{O}_F$. 
    Then, for every $\phi \in \End_F(E)$ we have that $\widehat{\phi} \in \mathcal{O}_F\llbracket t \rrbracket$.
\end{theorem}

Let now $\mathfrak{P} \subseteq \mathcal{O}_F$ be a prime of $F$ with residue field $\kappa_{\mathfrak{P}}$ and corresponding maximal ideal $\mathfrak{m}_{\mathfrak{P}} \subseteq \mathcal{O}_{F_{\mathfrak{P}}}$, where $F_{\mathfrak{P}}$ denotes the completion of $F$ at $\mathfrak{P}$. 
Then \cite[\S~2]{st08} shows that there is a unique injective group homomorphism $\iota_{\mathfrak{P}} \colon \widehat{E}(\mathfrak{m}_{\mathfrak{P}}) \to E(F_{\mathfrak{P}})$ making the following diagram
\begin{equation} \label{eq:commutative_square_formal_groups}
        \begin{tikzcd}
            \widehat{E}(\mathfrak{m}_{\mathfrak{P}}) \arrow[d, "{\widehat{\phi}_{\mathfrak{P}}}"'] \arrow[r,"\iota_{\mathfrak{P}}"] & E(F_{\mathfrak{P}}) \arrow[d, "{\phi}"] \\
            \widehat{E}(\mathfrak{m}_{\mathfrak{P}}) \arrow[r,"\iota_{\mathfrak{P}}"] & E(F_{\mathfrak{P}})
        \end{tikzcd}
    \end{equation}
commute for every $\phi \in \End_{F_{\mathfrak{P}}}(E)$, where $\widehat{\phi}_{\mathfrak{P}} := ( \widehat{\phi} )_{\mathfrak{m}_{\mathfrak{P}}}$ (see \cref{sec:formal_groups}).

Suppose now that $E$ has good reduction at $\mathfrak{P}$.
Then \cite[Chapter~VII, Proposition~2.1 and Proposition~2.2]{si09} imply that $\iota_{\mathfrak{P}}$ fits in the following exact sequence 
\[ 
        0 \to \widehat{E}(\mathfrak{m}_{\mathfrak{P}}) \xrightarrow{\iota_{\mathfrak{P}}} E(F_{\mathfrak{P}}) \xrightarrow{\pi_{\mathfrak{P}}} \widetilde{E}(\kappa_{\mathfrak{P}}) \to 0
\]
in which $\widetilde{E}$ denotes the reduction of $E$ modulo $\mathfrak{P}$ and $\pi_{\mathfrak{P}} \colon E(F_{\mathfrak{P}}) \twoheadrightarrow \widetilde{E}(\kappa_{\mathfrak{P}})$ is the canonical projection. Taking torsion and using \eqref{eq:commutative_square_formal_groups} we get an exact sequence
\begin{equation} \label{eq:exact_torsion_sequence}
    0 \to \widehat{E}(\mathfrak{m}_{\mathfrak{P}})[\widehat{\Phi}] \xrightarrow{\iota_{\mathfrak{P}}} E(F_{\mathfrak{P}})[\Phi] \xrightarrow{\pi_{\mathfrak{P}}} \widetilde{E}(\kappa_{\mathfrak{P}})[\Phi]
\end{equation}
for every ideal $\Phi \subseteq \operatorname{End}_{F_{\mathfrak{P}}}(E)$. 
Here $E(F_{\mathfrak{P}})[\Phi] \subseteq E(F_{\mathfrak{P}})$ is the $\Phi$-torsion subgroup
\[
    E(F_{\mathfrak{P}})[\Phi] := \bigcap_{\phi \in \Phi} \ker(\phi)
\]
and $\widetilde{E}(\kappa_{\mathfrak{P}})[\Phi]$ is defined analogously, noting that the map $\End_{F_{\mathfrak{P}}}(E) \to \End_{\kappa_{\mathfrak{P}}}(\widetilde{E})$ is injective (see \cite[Chapter~II, Proposition~4.4]{si94}).
We remark that $\widehat{E}(\mathfrak{m}_{\mathfrak{P}})[\widehat{\Phi}]$ is well defined since $\widehat{\Phi} \subseteq \mathcal{O}_F\llbracket t \rrbracket$ by \cref{thm:streng}.
Sequence \eqref{eq:exact_torsion_sequence} will be extensively used in the next section.

\section{Division fields of CM elliptic curves: ramification and entanglement} \label{sec:main_results}

The goal of this section is to prove \cref{thm:disjointness_intro} by studying the ramification properties of primes in division field extensions associated to CM elliptic curves, as described in \cref{prop:unramified} and \cref{prop:ramification}.
The proof of these results is an application to the CM case of the theory of formal groups outlined in \cref{sec:preliminaries}. 
We work in a fixed algebraic closure $\overline{\Q}$ of $\Q$. 

Let $F \subseteq \overline{\Q}$ be a number field and let $E_{/F}$ be an elliptic curve with \textit{complex multiplication} by an order $\mathcal{O}$ in an imaginary quadratic field $K$, which means that $\End_{\overline{\Q}}(E) \cong \mathcal{O}$. 
One can always fix, combining \cite[Chapter~II, Proposition~1.1]{si94} and \cite[Chapter~IV, Corollary~4.3]{si09}, a unique isomorphism 
$[\cdot]_E \colon \mathcal{O} \altxrightarrow{\sim} \End_{\overline{\Q}}(E)$
normalised in such a way that $\widehat{[\alpha]}_E'(0) = \alpha$ for every $\alpha \in \mathcal{O}$, where $\widehat{[\alpha]}_E \in \End_{\overline{\Q}}(\widehat{E})$ denotes the endomorphism of the formal group $\widehat{E}$ associated to $[\alpha]_E$ by \eqref{eq:phi_to_phi_hat}.
We will assume throughout this section that the field of definition $F$ contains the CM field $K$. This assumption implies in particular that all the endomorphisms of $E$ are already defined over $F$, as proved in \cite[Chapter~II, Proposition~30]{Shimura_1998}.

For any field extension $F \subseteq L \subseteq \overline{\Q}$ and any ideal $I \subseteq \mathcal{O}$ we write
\[
    E(L)[I] := \{ P \in E(L): [\alpha]_E(P)=0 \text{ for all } \alpha \in I \}
\]
for the set of $I$-torsion points of $E$ defined over $L$, which is naturally a module over $\mathcal{O}/I$.
When $I = \alpha \cdot \mathcal{O}$ for some $\alpha \in \mathcal{O}$ we write $E(L)[\alpha] := E(L)[I]$ and $E[\alpha] := E(\overline{\Q})[\alpha]$.
For any nonzero ideal $I \subseteq \mathcal{O}$ the group $E[I] := E(\overline{\Q})[I]$ is finite, and gives rise to a finite extension $F \subseteq F(E[I])$ obtained by adjoining to $F$ the coordinates of every $I$-torsion point.
We refer to the number field $F(E[I])$ as the $I$-\textit{division field} of $E_{/F}$.
The next result summarises the main properties of the extension $F \subseteq F(E[I])$ when $I$ is invertible.

\begin{lemma} \label{lem:free_O-module}
     Let $F$ be a number field and $E_{/F}$ an elliptic curve with complex multiplication by an order $\mathcal{O}$ in an imaginary quadratic field $K \subseteq F$. 
     Then for every ideal $I \subseteq \mathcal{O}$ the extension $F \subseteq F(E[I])$ is Galois and
     there is a canonical inclusion $\Gal(F(E[I])/F) \hookrightarrow \Aut_{\mathcal{O}}(E[I])$. 
     Moreover, if $I$ is invertible, the group $E[I]$ has a natural structure of free $\mathcal{O}/I$-module of rank one and, after choosing a generator, one gets an injective group homomorphism
     \[ \rho_{E,I} \colon \Gal(F(E[I])/F) \hookrightarrow (\mathcal{O}/I)^{\times} \]
     which will be denoted by $\rho_{E,N}$ when $I = N \cdot \mathcal{O}$ for some $N \in \mathbb{Z}$. Under the further assumption that $I$ is coprime to the ideal $\mathfrak{f}_{\mathcal{O}} \cdot \mathcal{O}$ generated by the conductor $\mathfrak{f}_{\mathcal{O}} := \lvert \mathcal{O}_K \colon \mathcal{O} \rvert$ of the order $\mathcal{O}$, one has that $\mathcal{O}/I \cong \mathcal{O}_K/I \mathcal{O}_K$.
\end{lemma}
\begin{proof}
    Since $F$ contains the CM field $K$, the endomorphisms of $E$ are all defined over $F$ and this implies that $\Gal(\overline{\Q}/F)$ acts on $E[I]$ by $\mathcal{O}$-module automorphisms. 
    In particular $F \subseteq F(E[I])$ is Galois and there is a canonical inclusion $\Gal(F(E[I])/F) \hookrightarrow \Aut_{\mathcal{O}}(E[I])$.
    If $I$ is invertible, $E[I]$ has the structure of free $\mathcal{O}/I$-module of rank one by \cite[Lemma~2.4]{Bourdon_Clark_2020}, and the choice of a generator induces an isomorphism $\operatorname{Aut}_{\mathcal{O}}(E[I]) \cong (\mathcal{O}/I)^{\times}$ which gives the map $\rho_{E,I}$ appearing in the statement. The last assertion follows from \cite[Proposition~7.20]{Cox_2013}.
\end{proof}

With the next proposition we start our study concerning the ramification properties of the extensions $F \subseteq F(E[I])$ by finding an explicit finite set of primes outside which these are unramified.

\begin{proposition} \label{prop:unramified}
    Let $F$ be a number field and $E_{/F}$ an elliptic curve with complex multiplication by an order $\mathcal{O}$ in an imaginary quadratic field $K \subseteq F$. Denote by $\fgoticO := \lvert \mathcal{O}_K \colon \mathcal{O} \rvert$ the conductor of the order $\mathcal{O}$ and by $\fgotic_E \subseteq \mathcal{O}_F$ the conductor ideal of the elliptic curve $E$. Then for every ideal $I\subseteq \mathcal{O}$ coprime with $\fgoticO$ the extension $F \subseteq F(E[I])$ is unramified at all primes not dividing $(I \cdot \mathcal{O}_F) \cdot \fgotic_E$.
\end{proposition}

\begin{proof}
    Since $I$ is coprime with the conductor of the order $\mathcal{O}$, it can be uniquely factored into a product of invertible prime ideals of $\mathcal{O}$ (see \cite[Proposition~7.20]{Cox_2013}). The field $F(E[I])$ is then the compositum of all the division fields $F(E[\pgotic^n])$ with $\pgotic^n$ the prime power factors of $I$ in $\mathcal{O}$. Hence it suffices to prove that for every invertible prime ideal $\pgotic \subseteq \mathcal{O}$ and $n \in \N$, the field extension $F\subseteq F(E[\pgotic^n])$ is unramified at every prime of $F$ not dividing $(\mathfrak{p} \, \mathcal{O}_F) \cdot \fgotic_E$. 
    
    Fix an invertible prime $\mathfrak{p} \subseteq \mathcal{O}$ and write $L := F(E[\pgotic^n])$. Let $\mathfrak{q} \nmid (\mathfrak{p} \, \mathcal{O}_F) \cdot \fgotic_E$ be a prime of $F$ and fix a prime $\mathfrak{Q} \subseteq \mathcal{O}_L$ lying above $\mathfrak{q}$, with residue field $\kappa$. Since $\mathfrak{q}$ does not divide the conductor $\mathfrak{f}_E$ of the elliptic curve, $E$ has good reduction $\widetilde{E}$ modulo $\mathfrak{q}$ and we then denote by $\pi \colon E(L) \to \widetilde{E}(\kappa)$ the reduction modulo $\mathfrak{Q}$. 
    Take $\sigma \in I(\mathfrak{Q}/\mathfrak{q})$, where $I(\mathfrak{Q}/\mathfrak{q}) \subseteq \Gal(L/F)$ denotes the inertia subgroup of $\mathfrak{q} \subseteq \mathfrak{Q}$, and fix a torsion point $Q \in E[\mathfrak{p}^n] = E(L)[\mathfrak{p}^n]$. 
    By definition of inertia $\sigma$ acts trivially on the residue field $\kappa$, hence
    \begin{equation} \label{eq:pi_Q}
        \pi(Q^{\sigma} - Q)=\pi(Q^{\sigma})-\pi(Q)=\pi(Q)-\pi(Q)=0
    \end{equation}
    \textit{i.e.} the point $Q^{\sigma} - Q$ is in the kernel of the reduction map $\pi$. We are going to use the exact sequence \eqref{eq:exact_torsion_sequence} to show that the only $\pgotic^n$-torsion point contained in this kernel is $0$. To this aim, we embed $L$ in its $\mathfrak{Q}$-adic completion $L_{\mathfrak{Q}}$ with ring of integers $\mathcal{O}_{L_{\mathfrak{Q}}}$ and maximal ideal $\mathfrak{m}_{\mathfrak{Q}}$. Notice that the set $(\pgotic^n \cap \mathcal{O})\setminus (\mathfrak{Q} \cap \mathcal{O})$ is non-empty because $\mathfrak{p} \nmid \mathfrak{f}_{\mathcal{O}}$ and $\mathfrak{q} \nmid (\mathfrak{p} \, \mathcal{O}_F)$.
    Consider then the formal group $\widehat{E} \in \mathcal{O}_F\llbracket z_1, z_2 \rrbracket$ associated to an integral Weierstrass model of $E$, and let $\alpha \in (\pgotic^n \cap \mathcal{O})\setminus (\mathfrak{Q} \cap \mathcal{O})$. 
    The endomorphism $\widehat{[\alpha]}_{E} \in \End_F(\widehat{E})$ corresponding to $[\alpha]_E \in \End_{F}(E)$ via \eqref{eq:phi_to_phi_hat} becomes an automorphism over $L_{\mathfrak{Q}}$, because $\widehat{[\alpha]}_E'(0) = \alpha \in \mathcal{O}_{L_{\mathfrak{Q}}}^{\times}$. 
    Hence taking $\Phi = [\mathfrak{p}^n]_E$ in \eqref{eq:exact_torsion_sequence} shows that $E[\pgotic^n] \cap \ker(\pi) \subseteq E[\alpha] \cap \ker(\pi) = \{0\}$, where the last equality holds because $\widehat{E}(\mathfrak{m}_\mathfrak{Q})\widehat{[\alpha]}_E = 0$.
    Combining this with \eqref{eq:pi_Q} we see that 
    $Q^{\sigma} = Q$ for every $Q \in E[\mathfrak{p}^n]$ and $\sigma \in I(\mathfrak{Q}/\mathfrak{q})$.
    Since $L$ is generated over $F$ by the elements of $E[\mathfrak{p}^n]$, we deduce that the inertia group $I(\mathfrak{Q}/\mathfrak{q})$ is trivial. In particular, $F \subseteq L$ is unramified at every prime not dividing $(\mathfrak{p} \cdot \mathcal{O}_{F}) \, \mathfrak{f}_E$, as wanted.
\end{proof}

We now turn to the study of the primes which ramify in $F \subseteq F(E[I])$. To do this it suffices to restrict our attention to the case $I = \mathfrak{p}^n$ for some prime $\mathfrak{p} \subseteq \mathcal{O}$ and some $n \in \mathbb{N}$, as we do in the following proposition.

\begin{proposition} \label{prop:ramification}
    Let $F$ be a number field and $E_{/F}$ an elliptic curve with complex multiplication by an order $\mathcal{O}$ in an imaginary quadratic field $K \subseteq F$. Denote by $B_E := \fgotic_\mathcal{O} \, \Delta_F \, N_{F/\Q}(\fgotic_E)$ the product of the conductor $\fgotic_\mathcal{O} := \lvert \mathcal{O}_K \colon \mathcal{O} \rvert$ of the order $\mathcal{O}$, the absolute discriminant $\Delta_F \in \mathbb{Z}$ of the number field $F$ and the norm $N_{F/\Q}(\fgotic_E) := \lvert \mathcal{O}_F/\mathfrak{f}_E \rvert$ of the conductor ideal $\fgotic_E \subseteq \mathcal{O}_F$. 
    Then for any $n \in \mathbb{N}$ and any prime ideal $\mathfrak{p} \subseteq \mathcal{O}$ coprime with $B_E \, \mathcal{O}$ the extension $F \subseteq F(E[\mathfrak{p}^n])$ is totally ramified at each prime dividing $\mathfrak{p} \, \mathcal{O}_F$. Moreover, the Galois representation
    \[ 
        \rho_{E,\mathfrak{p}^n} \colon \Gal(F(E[\mathfrak{p}^n])/F) \hookrightarrow (\mathcal{O}/\mathfrak{p}^n)^{\times} \cong (\mathcal{O}_K/\mathfrak{p}^n \mathcal{O}_K)^{\times}
    \]
    defined in \cref{lem:free_O-module} is an isomorphism.
\end{proposition}
\begin{proof}
    The statement is trivially true if $n = 0$, hence we assume that $n \geq 1$.
    Fix $\widehat{E} \in \mathcal{O}_F\llbracket z_1, z_2 \rrbracket$ to be the formal group associated to an integral Weierstrass model of $E$, and let $\pgotic \subseteq \mathcal{O}$ be as in the statement. The hypothesis of coprimality with $B_E \, \mathcal{O}$ implies that $\pgotic$ is invertible in $\mathcal{O}$ and that it lies above a rational prime $p\in \N$ that is unramified in $K$. We divide the proof according to the splitting behaviour of $p$ in $\mathcal{O}$, which is the same as the splitting behaviour in $K$, since $p \nmid \mathfrak{f}_{\mathcal{O}}$. 
    
    First, assume that $p$ is inert in $K$, so that $\mathfrak{p} = p\mathcal{O}$. In this case, $L:=F(E[\pgotic^n])$ coincides with the $p^n$-division field $F(E[p^n])$. 
    The injectivity of the Galois representation
    \[ \rho_{E,p^n} \colon \Gal(L/F) \hookrightarrow (\mathcal{O}/p^n \mathcal{O})^{\times} \cong (\mathcal{O}_K/p^n \mathcal{O}_K)^{\times} \]
    shows that the degree of the extension $F\subseteq L$ is bounded as
    \[ [L:F] \leq \lvert \left(\mathcal{O}_K/p^n \mathcal{O}_K \right)^{\times} \rvert = p^{2(n-1)}(p^2-1).\]
    Let $\Pgotic \subseteq \mathcal{O}_L$ be a prime of $L$ lying above $p$ and denote by $L_{\Pgotic}$ the $\Pgotic$-adic completion of $L$ with ring of integers $\mathcal{O}_{L_\Pgotic}$, maximal ideal $\mathfrak{m}_\Pgotic$ and residue field $\kappa_\Pgotic$. We want to determine the ramification index $e(\Pgotic/(\Pgotic \cap \mathcal{O}_F))$.
    
    Since $p$ is inert in $K$, the reduced elliptic curve $\widetilde{E}$ is supersingular by \cite[\S~14, Theorem~12]{Lang_1987}, hence $\widetilde{E}(\kappa_\Pgotic)[p^n] = 0$. 
    Taking $\Phi = [p^n]_{E}$ in \eqref{eq:exact_torsion_sequence}, we see that the group $\widehat{E}(\mathfrak{m}_\Pgotic)$ contains a non-zero point of exact order $p^n$.
    We can now use \cref{lem:silverman_exercise_formal_groups} and the hypothesis $p \nmid \Delta_F$ to get
    \begin{equation} \label{eq:inequality_inert_case}
        p^{h(n-1)}(p^h - 1) \leq v_{L_\Pgotic}(p) = e(\Pgotic/p) = e(\Pgotic/(\Pgotic \cap \mathcal{O}_F)) \leq [L \colon F] \leq p^{2(n-1)}(p^2 - 1).
    \end{equation}
    where $h \in \mathbb{N}$ denotes the height of the reduction modulo $\mathfrak{P}$ of the formal group $\widehat{E}$. Since the latter is precisely the formal group associated to $\widetilde{E}$, we have that $h = 2$ by \cite[Chapter~V, Theorem~3.1]{si09}. Thus all the inequalities appearing in \eqref{eq:inequality_inert_case} are actually equalities, and we see at once that $
        e(\Pgotic/(\Pgotic \cap \mathcal{O}_F)) = [L \colon F] = p^{2(n-1)}(p^2 - 1)$,
    which implies that $\rho_{E,p^n}$ is an isomorphism and that $\Pgotic \cap \mathcal{O}_F$ is totally ramified in $L$. 
    This concludes the proof of the inert case.
    
    Suppose now that $p$ splits in $K$, so that $p\mathcal{O}=\pgotic \overline{\pgotic}$, where $\overline{\pgotic}$ is the image of $\pgotic$ under the unique non-trivial automorphism of $K$. If we put again $L:=F(E[\pgotic^n])$, the injectivity of $\rho_{E,\mathfrak{p}^n}$ gives
    \[
        [L:F] \leq \lvert \left(\mathcal{O}_K/p^n \mathcal{O}_K \right)^{\times} \rvert = p^{n-1}(p-1).
    \]
    It is convenient in this case to work inside the bigger division field $M:=F(E[p^n])$, which contains both $L$ and $L':=F(E[\, \overline{\pgotic}^n \,])$. We then fix $\Pgotic,\overline{\Pgotic}\subseteq \mathcal{O}_{M}$ two primes of $M$ lying respectively above $\pgotic \mathcal{O}_K$ and $\overline{\pgotic} \mathcal{O}_K$, and we denote by $\mathcal{P}:=\Pgotic \cap \mathcal{O}_L$ and $\overline{\mathcal{P}}:=\overline{\Pgotic} \cap \mathcal{O}_L$ the corresponding primes in $L$.
    For every prime ideal $\mathfrak{q} \in \{ \Pgotic,\overline{\Pgotic} \}$ we denote by $M_{\mathfrak{q}}$ the $\mathfrak{q}$-adic completion of $M$ with ring of integers $\mathcal{O}_{M_\mathfrak{q}}$ and residue field $\kappa_{\mathfrak{q}}$, and by 
    $\widetilde{E}_{\mathfrak{q}}$ the reduction of $E_{/M}$ modulo $\mathfrak{q}$. We use analogous notation for $\mathcal{P}$ and $\overline{\mathcal{P}}$. 
    The goal is to compute the ramification index $e(\mathcal{P}/\mathcal{P}\cap \mathcal{O}_F)$, and we divide our argument in three steps.
    
    \vspace{0.4\baselineskip}
    \fbox{\textbf{Step 1}} \ 
    First of all, we prove that $E(M)[\mathfrak{p}^n] \cap \ker(\pi_{\overline{\mathfrak{P}}}) = 0$, where $\pi_{\overline{\mathfrak{P}}} \colon E(M) \to \widetilde{E}_{\overline{\mathfrak{P}}}(\kappa_{\overline{\mathfrak{P}}})$ 
    denotes the reduction modulo $\overline{\mathfrak{P}}$. 
    Since $E(M)[\mathfrak{p}^n] \subseteq E(L) \subseteq E(L_{\overline{\mathcal{P}}})$, this is equivalent to say that $E(L_{\overline{\mathcal{P}}})[\mathfrak{p}^n] \cap \ker(\pi_{\overline{\mathcal{P}}}) = 0$, where 
    \[
        \pi_{\overline{\mathcal{P}}} \colon E(L_{\overline{\mathcal{P}}}) \twoheadrightarrow \widetilde{E}_{\overline{\mathcal{P}}}(\kappa_{\overline{\mathcal{P}}}) \subseteq \widetilde{E}_{\overline{\mathfrak{P}}}(\kappa_{\overline{\Pgotic}}) 
    \]
    denotes the reduction modulo $\overline{\mathcal{P}}$.
    Since $p$ is coprime with the conductor of the order $\mathcal{O}$ by assumption, it is possible to find~$\alpha \in \mathfrak{p}^n$ such that $\alpha \not\in \overline{\pgotic}$. 
    The endomorphism $\widehat{[\alpha]}_{E} \in \End_F(\widehat{E})$ corresponding to $[\alpha]_E \in \End_{F}(E)$ via \eqref{eq:phi_to_phi_hat} becomes an automorphism over $L_{\overline{\mathcal{P}}}$, because $\widehat{[\alpha]}_E'(0) = \alpha \in \mathcal{O}_{L_{\overline{\mathcal{P}}}}^{\times}$.
    Hence taking $\Phi = [\mathfrak{p}^n]_E$ in \eqref{eq:exact_torsion_sequence} shows that 
    \[
    \ker(\pi_{\overline{\mathcal{P}}}) \cap E(L_{\overline{\mathcal{P}}})[\mathfrak{p}^n] \subseteq \ker(\pi_{\overline{\mathcal{P}}}) \cap E(L_{\overline{\mathcal{P}}})[\alpha] = 0
    \]
    where the last equality holds because $\widehat{E}(\mathfrak{m}_{\overline{\mathcal{P}}})\widehat{[\alpha]}_E = 0$.
    In exactly the same way, using $L'$ in place of $L$, one shows that $E(M)[\overline{\pgotic}^n] \cap \ker(\pi_\mathfrak{P}) = 0$.
    
    \vspace{0.4\baselineskip}
    \fbox{\textbf{Step 2}} \ We now claim that $E(M)[p^n] \cap \operatorname{ker}(\pi_\Pgotic) = E(M)[\mathfrak{p}^n]$ where $\pi_\Pgotic: E(M) \to \widetilde{E}_\Pgotic(\kappa_\Pgotic)$ 
    denotes the reduction modulo $\Pgotic$. 
    Since $p^n \mathcal{O} = \pgotic^n \overline{\pgotic}^n$ with $\pgotic^n + \overline{\pgotic}^n = \mathcal{O}$, there is a decomposition of the group  $E(M)[p^n]$ into the direct sum of $E(M)[\, \pgotic^n \,]$ and $E(M)[\, \overline{\pgotic}^n \,]$, which are cyclic groups of order $p^n$ by \cref{lem:free_O-module}. 
    In particular, there exists $A \in E(M)[\pgotic^n]$ and $B \in E(M)[\overline{\pgotic}^n]$ such that every $p^n$-torsion point $Q \in E(M)[p^n]$ can be written as
    \[ Q = [a](A) + [b](B) \]
    for unique $a, b \in \{0,\dots,p^n-1\}$.
    If $\pi_\Pgotic(Q)=0$ then
    \[ \pi_\Pgotic([b](B)) = \pi_\Pgotic ([-a](A)) \in \widetilde{E}_\mathfrak{P}(\kappa_\Pgotic)[\, \mathfrak{p}^n \,] \cap \widetilde{E}_\mathfrak{P}(\kappa_\Pgotic)[\, \overline{\mathfrak{p}}^n \,] = \{0\}\]
    where the last equality follows from the fact that $\pgotic^n$ and $\overline{\pgotic}^n$ are coprime in $\mathcal{O}$. 
    In particular, $[b](B) \in \ker(\pi_\mathfrak{P}) \cap E(M)[\overline{\mathfrak{p}}^n]$, and the latter is trivial by \textbf{Step 1}. 
    Hence we have $Q = [a](A) \in E(M)[\pgotic^n]$, and this shows the inclusion $\ker(\pi_\Pgotic) \cap E(M)[p^n] \subseteq E(M)[\pgotic^n]$. 
    To prove the other inclusion first notice that the restriction of $\pi_\Pgotic$ to $E(M)[p^n]$ gives rise to a surjection 
    \[
    E(M)[p^n] \twoheadrightarrow \widetilde{E}_\mathfrak{P}(\kappa_{\Pgotic})[p^n]
    \] 
    because $E(M)[\overline{\pgotic}^n] \to \widetilde{E}_\mathfrak{P}(\kappa_{\Pgotic})[p^n]$ is injective and the elliptic curve $\widetilde{E}_{\Pgotic}$ is ordinary by Deuring's reduction criterion (see \cite[Chapter~13, Theorem~12]{Lang_1987}). 
    This gives
    \[
        \frac{E(M)[p^n]}{\ker(\pi_\Pgotic) \cap E(M)[p^n]} \cong \widetilde{E}_\mathfrak{P}(\kappa_{\Pgotic})[p^n]
    \]
    which in turn shows that
    \[
        \lvert \ker(\pi_\Pgotic) \cap E(M)[p^n] \rvert = \frac{\lvert E(M)[p^n] \rvert}{\lvert \widetilde{E}_\mathfrak{P}(\kappa_{\Pgotic})[p^n] \rvert} = \frac{p^{2n}}{p^n} = p^n = \lvert E(M)[\pgotic^n] \rvert.
    \]
    We conclude that $\ker(\pi_\Pgotic) \cap E(M)[p^n] = E(M)[\pgotic^n]$.
    
    \vspace{0.4\baselineskip}
    \fbox{\textbf{Step 3}} \ Using \eqref{eq:exact_torsion_sequence} with $\Phi = [p^n]_E$ and \textbf{Step 2}, after recalling that $\Pgotic$ lies over $\mathcal{P}$,
    one can see that the group $\widehat{E}(\mathfrak{m}_{\mathcal{P}})$ contains a point of exact order $p^n$.
    We now apply \cref{lem:silverman_exercise_formal_groups} and the hypothesis $p \nmid \Delta_F$ to get
    \begin{equation} \label{eq:inequality_split_case}
        p^{h(n-1)}(p^h - 1) \leq v_{L_\mathcal{P}}(p) = e(\mathcal{P}/p) = e(\mathcal{P}/(\mathcal{P} \cap \mathcal{O}_F)) \leq [L \colon F] \leq p^{n-1}(p - 1).
    \end{equation}
    where $h \in \mathbb{N}$ denotes the height of the reduction modulo $\mathcal{P}$ of the formal group $\widehat{E}$. Since the latter is precisely the formal group associated to the ordinary elliptic curve $\widetilde{E}_\mathcal{P}$, we have that $h = 1$ by \cite[Chapter~V, Theorem~3.1]{si09}. 
    Thus all the inequalities appearing in \eqref{eq:inequality_split_case} are actually equalities, and we see at once that $
        e(\mathcal{P}/(\mathcal{P} \cap \mathcal{O}_F)) = [L \colon F] = p^{n-1}(p - 1)$,
    which implies that $\rho_{E,\mathfrak{p}^n}$ is an isomorphism and that $\mathcal{P} \cap \mathcal{O}_F$ is totally ramified in $L$. 
    This concludes the proof.
\end{proof}

\begin{remark} \label{rmk:lozano_robledo}
    As we already stated in the introduction, \cref{prop:ramification} can be obtained by combining various results of Lozano-Robledo.
    More precisely, see \cite[Proposition~5.6]{Lozano-Robledo_2016} for the inert case and the proof of \cite[Theorem~6.10]{Lozano-Robledo_2018} for the split case.
    The arguments used by Lozano-Robledo for the inert case involve a formula for the valuation of the coefficient of $t^p$ in the power series $[p]_{\widehat{E}}(t) \in \mathcal{O}_F\llbracket t \rrbracket$ (see \cite[Theorem~3.9]{Lozano-Robledo_2013-formal}), and the study of the split case goes through a detailed investigation of Borel subgroups of $\operatorname{GL}_2(\mathbb{Z}/p^n \mathbb{Z})$ (see \cite[Section~4]{Lozano-Robledo_2018}).
    
    Our proof of \cref{prop:ramification}, which concerns only CM elliptic curves and prime ideals not dividing $B_E \, \mathcal{O}$, appears to be shorter because it uses the same techniques to deal with the split and inert case.
    Notice as well that our discussion is explicitly written for general imaginary quadratic orders, whereas \cite[Theorem~6.10]{Lozano-Robledo_2018} is stated and proved only for maximal orders.
    We observe however that \cite[Remark~6.12]{Lozano-Robledo_2018} points out that the proof of \cite[Theorem~6.10]{Lozano-Robledo_2018} carries over to the general case.
    
    We also remark that, if $\mathcal{O} = \mathcal{O}_K$ is a maximal order of class number $1$ and $F=K$, \cref{prop:ramification} is proved by Coates and Wiles in \cite[Lemma~5]{Coates_Wiles_1977} (see also \cite[Lemma~3]{Arthaud_1978} and \cite[Proposition~47]{Coates_2013}). The main tool used in their proof is Lubin-Tate theory.
\end{remark}

\begin{remark}
    Let $E_{/F}$ be any elliptic curve (not necessarily with complex multiplication) which has good supersingular reduction at a prime $\mathfrak{p} \subseteq \mathcal{O}_F$ lying above a prime $p \in \mathbb{N}$ which does not ramify in $\mathbb{Q} \subseteq F$. 
    Then one can use the same argument provided in the first part of the proof of \cref{prop:ramification} to show that the ramification index $e(\mathfrak{P}/\mathfrak{p})$ is bounded from below by $p^{2 (n-1)} (p^2 - 1)$, where $\mathfrak{P} \subseteq F(E[p^n])$ is any prime lying above $\mathfrak{p}$.
    This result has already been proved by Lozano-Robledo in \cite[Proposition~5.6]{Lozano-Robledo_2016} and by Smith in \cite[Theorem~2.1]{Smith_2018}.
\end{remark}

\begin{remark} \label{rmk:lozano_robledo_arXiv}
    Let $E$ be an elliptic curve having complex multiplication by an imaginary quadratic order $\mathcal{O}$, and suppose that $E$ is defined over the ring class field $H_\mathcal{O}$. 
    Then using the recent work \cite{Lozano-Robledo_2019} of Lozano-Robledo, and in particular  \cite[Theorem~1.2.(4)]{Lozano-Robledo_2019} and \cite[Theorem~7.11]{Lozano-Robledo_2019}, one can show that the Galois representation $\rho_{E,p^n}$ is an isomorphism for every $n \in \mathbb{N}$ and every rational prime $p \in \mathbb{N}$ such that $p \nmid 2 \mathfrak{f}_\mathcal{O} \Delta_K$.
    This strengthens, for elliptic curves defined over $H_\mathcal{O}$, the final assertion of \cref{prop:ramification}.
\end{remark}

We are now ready to prove \cref{thm:disjointness_intro}.
Recall that a family $\mathcal{F}=\{F_s\}_{s \in \mathcal{S}}$ of Galois extensions of a number field $F$, indexed over any set $\mathcal{S}$, is called \textit{linearly disjoint} over $F$ if the natural inclusion map
\[
    \Gal(L/F) \hookrightarrow \prod_{s \in \mathcal{S}} \Gal(F_s/F)
\]
is an isomorphism, where $L$ denotes the compositum of the fields $F_s$. Otherwise the family is called \textit{entangled} over $F$.

\begin{proof}[Proof of \cref{thm:disjointness_intro}]
    The family $\{F(E[p^\infty])\}_{q \not\in S} \cup \{ F(E[S^{\infty}])\}$ appearing in the statement of \cref{thm:disjointness_intro} is linearly disjoint over $F$ if and only if $F(E[p^n]) \cap F(E[m]) = F$ for every prime $p \not \in S$, every $n \in \N$ and every $m \in \mathbb{Z}$ coprime with $p$. To prove this latter statement, we first show that every non-trivial subextension of $M := F(E[p^n])$ is ramified at some prime dividing $p$. 
    
    When $p$ is inert in $K$, this follows immediately from \cref{prop:ramification}. 
    Suppose then that $p$ is split in $K$, with $p\mathcal{O}_K=\pgotic \overline{\pgotic}$. 
    The division field $M$ is the compositum over $F$ of the extensions $L := F(E[\pgotic^n])$ and $L' := F(E[\overline{\pgotic}^n])$. By \cref{prop:ramification} the extension $F \subseteq L$ (respectively $F \subseteq L'$) is totally ramified at every prime of $F$ lying over $\pgotic$ (resp. $\overline{\pgotic}$).
    Let $\Pgotic$ be a prime of $F$ lying above $\pgotic$, and denote by $I(\Pgotic) \subseteq \Gal(M/F)$ its inertia group and by $e(\Pgotic)$ its ramification index in the extension $F \subseteq M$. 
    If $F \subsetneq \widetilde{F}$ is a subextension of $F \subseteq M$ in which $\Pgotic$ does not ramify, then $\widetilde{F}$ must be contained in the inertia field $T=(M)^{I(\Pgotic)}$ relative to $\Pgotic$. Notice that the latter also contains $L'$, since by \cref{prop:unramified} the extension $F \subseteq L'$ is unramified at $\Pgotic$. On the other hand, the fact that $F \subseteq L$ is totally ramified at $\Pgotic$ gives the chain of inequalities
    \[
        [L' \colon F] \leq [T \colon F] = \frac{[M \colon F]}{\lvert I(\Pgotic) \rvert} = \frac{[M \colon F]}{e(\Pgotic)} \leq
        \frac{[L \colon F] \cdot [L' \colon F]}{e(\Pgotic)} \leq [L' \colon F]
    \]
    which shows that $T = L'$. 
    Hence \cref{prop:ramification} implies that any extension $F \subseteq \widetilde{F}$ which is unramified at every prime above $\pgotic$ is totally ramified at every prime above $\overline{\pgotic}$.
    
    Now it is easy to conclude that $M \cap F(E[m]) = F$, since otherwise $F \subseteq F(E[m])$ would ramify at some prime of $F$ dividing $p$, contradicting \cref{prop:unramified}.
\end{proof}
\begin{remark} \label{rmk:minimality_of_S}
    Let $F$ be a number field and $E_{/F}$ an elliptic curve with complex multiplication by an order $\mathcal{O}$ in an imaginary quadratic field $K \subseteq F$.
    Denote by $S \subseteq \mathbb{N}$ the set of primes dividing $B_E$, which is the minimal set of primes one can take in \cref{thm:disjointness_intro}.
    In this general setting it is an interesting question to study the entanglement in the finite family of ``bad'' division fields $\{ F(E[p^\infty]) \}_{p \in S}$, as we do in \cref{sec:entanglement} where we specify $F = K$ and $E$ to be the base-change of an elliptic curve defined over $\mathbb{Q}$.
    
    A first step towards a complete answer to the previous question in the general setting is to find the minimal set $S' \subseteq S$ such that the family of division fields \[\{ F(E[p^\infty]) \}_{p \not\in S'} \cup \{F(E[(S')^\infty])\}\] is linearly disjoint over $F$.
    We partially answer the latter question in \cref{thm:infinitely_many_linearly_disjoint}, where we prove that, if $\operatorname{Pic}(\mathcal{O}) \neq \{1\}$, one can take $S' = \emptyset$ for infinitely many elliptic curves defined over the ring class field $H_\mathcal{O}$.
    On the other hand, if $\operatorname{Pic}(\mathcal{O}) = \{1\}$, there are infinitely many examples of elliptic curves $E$ having complex multiplication by $\mathcal{O}$ for which $S' = S$ can be arbitrarily large (see \cref{rmk:minimality_of_S_over_Q}).
\end{remark}
\begin{remark} \label{rmk:division_fields_base_change}
    Let $F$ be a number field and $E$ be a CM elliptic curve defined over $F$. Then, even when $K \not\subseteq F$, we have that $K \subseteq F(E[N])$ for every $N > 2$. This has been shown in \cite[Lemma~6]{Murty_1983} for $F=\Q$ and in \cite[Lemma~3.15]{Bourdon_Clark_Stankewicz_2017} for arbitrary $F$.
    In particular, the statement of \cref{thm:disjointness_intro} does not hold when $K \not\subseteq F$. 
\end{remark}

The description of the set of primes $S$ in \cref{thm:disjointness_intro} is actually redundant, since all the primes $p$ dividing the conductor $\fgoticO$, with the possible exception of $p=2$, also divide the absolute discriminant $\Delta_F$ of the field of definition of $E$. This can be seen as follows: since $K\subseteq F$, the field $F$ always contains the field $K(j(E))$, obtained by adjoining to $K$ the $j$-invariant $j(E)$ of the elliptic curve $E$. Despite its definition, $H_{\mathcal{O}}:=K(j(E))$ does not depend on $E$ but only on its CM order $\mathcal{O}$, and is called the \textit{ring class field} of $K$ relative to the order $\mathcal{O}$. The extension $K\subseteq H_{\mathcal{O}}$ is always abelian and it can only be ramified at those primes of $K$ which divide the conductor $\fgoticO$ (see \cite[\S~9.A]{Cox_2013}). 
If $\mathcal{O} = \mathcal{O}_K$, the field $H_{\mathcal{O}_K}$ coincides with the \textit{Hilbert class field} of $K$, \textit{i.e.} the maximal abelian extension of $K$ which is unramified everywhere.
The initial assertion now follows from the following proposition, which is a weaker form of \cite[Exercise~9.20]{Cox_2013}.
\begin{proposition} \label{prop:ramification_conductor}
    Let $\mathcal{O}$ be an order of conductor $\fgoticO := \lvert \mathcal{O}_K \colon \mathcal{O} \rvert$ in an imaginary quadratic field $K$. Then the extension $\Q \subseteq H_{\mathcal{O}}$ is ramified at all the odd primes dividing $\fgoticO$. Moreover if $4 \mid \fgoticO$ the same extension is also ramified at $2$.
\end{proposition}
\begin{proof}
    If $\mathfrak{f}_{\mathcal{O}} = 1$ there is nothing to prove.
    Otherwise let $\mathfrak{f}_{\mathcal{O}} = p_1^{a_1} \cdots p_n^{a_n}$ be the prime factorisation of $\mathfrak{f}_{\mathcal{O}}$, and observe that, for every $i \in \{1,\dots,n\}$, one has the chain of inclusions
    \[ 
        K \subseteq H_{\mathcal{O}_K} \subseteq H_{\mathcal{O}_i} \subseteq H_{\Ogotic}
    \]
    given by the \textit{Anordnungsatz} for ring class fields (see \cref{rmk:anordnungssatz}), where $\mathcal{O}_i$ denotes the order of conductor $p_i^{a_i}$.
    Now, the class number formula \cite[Theorem~7.24]{Cox_2013} yields
    \begin{equation} \label{eq:order_class_number_formula}
        [H_{\mathcal{O}_i}:H_{\mathcal{O}_K}] = \frac{[H_{\mathcal{O}_i} \colon K]}{[H_{\mathcal{O}_K} \colon K]} = \frac{h_{\mathcal{O}_i}}{h_K}=\frac{p_i^{a_i}}{|\mathcal{O}_K^{\times}:\mathcal{O}_i^{\times}|} \left( 1-\left( \frac{\Delta_K}{p_i} \right) \frac{1}{p_i} \right).
    \end{equation}
    where $h_{\mathcal{O}_i} := [H_{\mathcal{O}_i} \colon K] = \lvert \operatorname{Pic}(\mathcal{O}_i) \rvert$ and analogously $h_K := [H_{\mathcal{O}_K} \colon K] = \lvert \operatorname{Pic}(\mathcal{O}_K) \rvert$.
    If $p_i \geq 3$ or $p_i = 2$ and $a_i \geq 2$, we see from \eqref{eq:order_class_number_formula} that $H_{\mathcal{O}_i} \neq H_{\mathcal{O}_K}$ except when $p_i=3$, $a_i = 1$ and $K = \Q(\sqrt{-3})$. 
    In this last case the extension $\mathbb{Q} \subseteq K$ is ramified at $p_i = 3$.
    Otherwise the extension $H_{\mathcal{O}_K} \subsetneq H_{\mathcal{O}_i}$ is ramified at some prime dividing $p_i$.
    Indeed, $H_{\mathcal{O}_K} \subsetneq H_{\mathcal{O}_i}$ is ramified at some prime because $K \subseteq H_{\mathcal{O}_i}$ is abelian and $H_{\mathcal{O}_K}$ is the Hilbert class field of $K$, and this suffices to conclude because $K \subseteq H_{\mathcal{O}_i}$ can ramify only at primes lying above $p_i$.
\end{proof}

\begin{remark}
    If $2 \mid \mathfrak{f}_{\mathcal{O}}$ but $4 \nmid \mathfrak{f}_{\mathcal{O}}$ the extension $\mathbb{Q} \subseteq H_{\mathcal{O}}$ could still be unramified at $2$.
    This happens, for instance, if $\mathfrak{f}_{\mathcal{O}} = 2$ and $2$ splits in $K$, because in this case the ring class field $H_{\mathcal{O}}$ is equal to the Hilbert class field $H_{\mathcal{O}_K}$.
\end{remark}

\cref{prop:ramification_conductor} shows that the set $S$ in \cref{thm:disjointness_intro} could be replaced by the set $S'$ of primes dividing $2 \cdot \Delta_F \cdot \operatorname{N}_{F/\Q}(\fgotic_E)$, even if this results in a slightly weaker statement. However, choosing the set $S'$ instead of the set $S$ allows to draw a comparison with a result of Lombardo on the image of $p$-adic Galois representations attached to CM elliptic curves, which is shown in \cite[Theorem~6.6]{Lombardo_2017}. 
In this paper Lombardo proves the isomorphism 
\[
    \Gal(F(E[p^{\infty}])/F) \cong (\mathcal{O} \otimes_{\mathbb{Z}} \mathbb{Z}_p)^{\times}
\]
for every prime $p \nmid \Delta_F \cdot \operatorname{N}_{F/\Q}(\fgotic_E)$. If moreover $p \geq 3$, \textit{i.e.} $p \not\in S'$, this isomorphism follows also from \cref{prop:ramification} by taking inverse limits.
The methods used in \cite{Lombardo_2017} are different from ours and generalise also to higher dimensional abelian varieties.

\section{Ray class fields for orders}
\label{sec:appendix_ray_class_fields}

In this section we define, for every ideal $I \subseteq \mathcal{O}$ contained in a general \textit{order} $\mathcal{O} \subseteq F$ of a number field $F$, an abelian extension $F \subseteq H_{I,\mathcal{O}}$. We call $H_{I,\mathcal{O}}$ the \textit{ray class field modulo $I$ for the order $\mathcal{O}$}.
Our definition generalises the one given by Söhngen in \cite{Sohngen_1935} and Stevenhagen in \cite[\S~4]{Stevenhagen_2001}, who restrict their attention to imaginary quadratic fields.
The content of this section is probably well known to the experts, but the authors have included it here since they have been unable to find a suitable reference.

Let $F$ be a number field. For a place $w\in M_F$ denote by $F_w$ the completion of the number field $F$ at $w$ and by $\mathcal{O}_{F_w}$ its ring of integers. Let $\mathbb{A}_F$ be the \textit{adèle ring} of $F$, defined by the restricted product
\[ \mathbb{A}_F := \sideset{}{'}\prod_{w \in M_{F}} F_w= \left \{s=(s_w)_{w \in M_F} \in \prod_{w \in M_{F}} F_w \, \middle| \, s_w \in \mathcal{O}_{F_w} \text{ for almost all }   w \in M_{F} \right\}. \]
The discussion on \cite[Page~371]{Neukirch_1999} shows that $\mathbb{A}_F$ can be obtained from the rational adèle ring by extending scalars, \textit{i.e.} there is a ring isomorphism $\mathbb{A}_F  \cong \mathbb{A}_{\mathbb{Q}} \otimes_{\mathbb{Q}} F$. This enables us to talk, for a place $p \in M_\Q$, of the \textit{$p$-component} $s_p \in F_p := \Q_p \otimes_\Q F$ of an adèle $s \in \mathbb{A}_F$; in particular if $p=\infty$ is the unique infinite place of $\Q$ we have the \textit{infinity component} $s_{\infty} \in \mathbb{R}\otimes_{\Q} F$. Hence $s\in \mathbb{A}_F$ can be alternatively written as
\begin{equation} \label{eq:decomposition_idèles_over_Q}
    s=(s_w)_{w\in M_F} \hspace{0.5cm} \text{or} \hspace{0.5cm} s=(s_p)_{p\in M_\Q}
\end{equation}
and of course the same is true if $s$ belongs to the \textit{idèle group} $\mathbb{A}_F^\times$. In what follows, we will often confuse finite places $p \in M^0_\Q$ and rational primes $p \in \N$.
    
Using the language introduced above, we are now able to define the ray class fields $H_{I,\mathcal{O}}$.

\begin{definition} \label{def:ray_class_field}
    Let $F$ be a number field, let $\mathcal{O} \subseteq \mathcal{O}_F$ be an order and let $I \subseteq \mathcal{O}$ be a non-zero ideal. Then we define the \textit{ray class field of $F$ modulo $I$ relative to the order $\mathcal{O}$} as
    \begin{equation} \label{eq:ray_class_field}
        H_{I,\mathcal{O}} := (F^{\text{ab}})^{[U_{I,\mathcal{O}},F]} \subseteq F^{\text{ab}}
    \end{equation}
    where $[\cdot,F] \colon \mathbb{A}_F^{\times} \twoheadrightarrow \Gal(F^{\text{ab}}/F)$ is the \textit{global Artin map} (see \cite[Chapter~VI, \S~5]{Neukirch_1999}) and $U_{I,\mathcal{O}} \subseteq \mathbb{A}_F^{\times}$ is the subgroup
    \begin{equation} \label{eq:U_n}
        U_{I,\mathcal{O}} := \left\{ s \in \mathbb{A}_F^{\times} \ \middle| \ s_p \in \left(\mathcal{O}_p^{\times} \cap (1 + I \cdot \mathcal{O}_p)\right) \ \text{for all rational primes} \ p \in \mathbb{N} \right\}
    \end{equation}
    defined using the decomposition \eqref{eq:decomposition_idèles_over_Q}, where 
    \begin{equation} \label{eq:definition_O_p}
        \mathcal{O}_p := \varprojlim_{n \in \mathbb{N}} \frac{\mathcal{O}}{p^n \mathcal{O}} \cong \mathcal{O} \otimes_{\mathbb{Z}} \mathbb{Z}_p
    \end{equation}
    denotes the completion of $\mathcal{O}$ with respect to the ideal $p \, \mathcal{O}$.
\end{definition}

When $I = N \cdot \mathcal{O}$ for some $N \in \mathbb{Z}_{\geq 1}$ we denote $U_{I,\mathcal{O}}$ by $U_{N,\mathcal{O}}$, and we write $U_{\mathcal{O}} := U_{1,\mathcal{O}}$. Analogously, we will write $H_{N,\mathcal{O}}$ in place of $H_{N \cdot \mathcal{O},\mathcal{O}}$, and we will denote by $H_{\mathcal{O}} := H_{1,\mathcal{O}}$ the \textit{ring class field} of $\mathcal{O}$.

\begin{remark} \label{rmk:ray_class_fields_attribution}
    When $\mathcal{O} = \mathcal{O}_F$ is the ring of integers, the ray class fields $H_{I,\mathcal{O}_F}$ coincide with the usual ray class fields of $F$ modulo $I$ (see \cite[Chapter~VI, Definition~6.2]{Neukirch_1999}). 
    Moreover, when $F = K$ is an imaginary quadratic field, the ray class fields $H_{I,\mathcal{O}}$ have been defined by Söhngen in \cite{Sohngen_1935}. 
    This work is exposed in great detail by Schertz in \cite[\S 3.3]{Schertz_2010}, and if $I = N \cdot \mathcal{O}$ for some $N \in \mathbb{N}$ the construction of $H_{I,\mathcal{O}} = H_{N,\mathcal{O}}$ has been reformulated using an adelic language by Stevenhagen in \cite[\S~4]{Stevenhagen_2001}. 
    Finally, the ring class fields $H_{\mathcal{O}}$ have been studied for general number fields $F$ by Lv and Deng in \cite{Lv_Deng_2015} and by Yi and Lv in \cite{Yi_Lv_2018}.
\end{remark}

\begin{remark} \label{rmk:anordnungssatz}
    It is clear from the definition that for every pair of ideals $I \subseteq J \subseteq \mathcal{O}$ we have that $U_{I,\mathcal{O}} \subseteq U_{J,\mathcal{O}}$, which implies that $H_{I,\mathcal{O}} \supseteq H_{J,\mathcal{O}}$. In particular, $H_{\mathcal{O}} \subseteq H_{I,\mathcal{O}}$ for every ideal $I \subseteq \mathcal{O}$.
    Similarly, for every pair of orders $\mathcal{O}_1 \subseteq \mathcal{O}_2 \subseteq F$ and every ideal $I \subseteq \mathcal{O}_1$ we have that $U_{I,\mathcal{O}_1} \subseteq U_{I \cdot \mathcal{O}_2,\mathcal{O}_2}$, which gives the containment $H_{I,\mathcal{O}_1} \supseteq H_{I\cdot\mathcal{O}_2,\mathcal{O}_2}$
    generalising the \textit{Anordnungssatz} explained in \cite[Page~169]{Stevenhagen_2001}.
    In particular for every order $\mathcal{O} \subseteq F$ and every ideal $I \subseteq \mathcal{O}$ we get the following inclusions
    \[
        \begin{tikzcd}
        & {H_{I\cdot\mathcal{O}_F,\mathcal{O}_F}} \arrow[Subseteq]{r}{} & {H_{I,\mathcal{O}}} \arrow[Subseteq]{r}{}       & {H_{I\cdot\mathfrak{f}_{\mathcal{O}}\cdot\mathcal{O}_F,\mathcal{O}_F}}            \\
        F \arrow[Subseteq]{r}{} & H_{\mathcal{O}_F} \arrow[Subseteq]{u}{} \arrow[Subseteq]{r}{}             & H_{\mathcal{O}} \arrow[Subseteq]{r}{} \arrow[Subseteq]{u}{} & {H_{\mathfrak{f}_{\mathcal{O}},\mathcal{O}_F}} \arrow[Subseteq]{u}{}
        \end{tikzcd}
    \]
where $\mathfrak{f}_{\mathcal{O}} := (\mathcal{O} \colon \mathcal{O}_F) = \{ \alpha \in F \mid \alpha \mathcal{O}_F \subseteq \mathcal{O} \} \subseteq \mathcal{O}$ is the \textit{conductor} of $\mathcal{O}$, which is the biggest ideal of $\mathcal{O}_F$ contained in $\mathcal{O}$.
This shows, applying \cite[Chapter~VI, Corollary~6.6]{Neukirch_1999}, that the extension $F \subseteq H_{I,\mathcal{O}}$ is unramified outside the set of primes dividing $I \cdot \mathfrak{f}_{\mathcal{O}} \cdot \mathcal{O}_F$.
\end{remark}

We now describe the Galois groups of the abelian extensions $F \subseteq H_{I,\mathcal{O}}$. 

\begin{lemma} \label{lem:U_I_open}
    Let $F$ be a number field, $\mathcal{O} \subseteq \mathcal{O}_F$ be an order and $I \subseteq \mathcal{O}$ be a non-zero ideal. Then $F^{\times} \cdot U_{I,\mathcal{O}} \subseteq \mathbb{A}_{F}^{\times}$ is a closed subgroup of finite index and, after identifying the group 
    \[F_{\infty}^{\times}:=(F \otimes_{\mathbb{Q}} \mathbb{R})^{\times} \cong \prod_{v \in M_F^{\infty}} F_v^\times
    \] 
    with its image under the natural inclusion $F_{\infty}^{\times} \hookrightarrow \mathbb{A}_F^{\times}$, one has
    \[
        F^{\times} \cdot F_{\infty}^{\times} \subseteq \operatorname{ker}([\cdot,F]) \subseteq F^{\times} \cdot U_{I,\mathcal{O}} = F^{\times} \cdot \operatorname{N}_{H_{I,\mathcal{O}}/F}(\mathbb{A}_{H_{I,\mathcal{O}}}^\times)    
    \]
    where $\operatorname{N}_{H_{I,\mathcal{O}}/F} \colon \mathbb{A}_{H_{I,\mathcal{O}}}^{\times} \to \mathbb{A}_F^{\times}$ denotes the idelic norm map.
    Moreover, there is an isomorphism
    \begin{equation} \label{eq:idelic_galois_ray_class_field}
        \operatorname{Gal}(H_{I,\mathcal{O}}/F) \cong \frac{\mathbb{A}_F^{\times}}{F^{\times} \cdot U_{I,\mathcal{O}}}
    \end{equation}
    induced by the global Artin map.
\end{lemma}
\begin{proof}
    The fact that $F^{\times} \cdot U_{I,\mathcal{O}}$ is closed of finite index follows from \cite[Chapter~VI, Proposition~1.8]{Neukirch_1999}, because $U_{I \cdot \mathfrak{f}_{\mathcal{O}} \cdot \mathcal{O}_F,\mathcal{O}_F} \subseteq U_{I,\mathcal{O}}$. Moreover, by definition $F_{\infty}^{\times} \subseteq U_{I,\mathcal{O}}$, so the inclusions $F^{\times} \cdot F_{\infty}^{\times} \subseteq \operatorname{ker}([\cdot,F]) \subseteq F^{\times} \cdot U_{I,\mathcal{O}}$ follow from the fact that $F^{\times} \cdot U_{I,\mathcal{O}}$ is closed in $\mathbb{A}_F^{\times}$ and $\ker([\cdot,F])$ is the closure of $F^{\times} \cdot F_{\infty}^{\times}$ inside $\mathbb{A}_F^{\times}$, as explained in \cite[Chapter~IX]{Artin_Tate_1968}. The global reciprocity law \cite[Chapter~VI, Theorem~6.1]{Neukirch_1999} now gives \eqref{eq:idelic_galois_ray_class_field} and shows that $F^{\times} \cdot \operatorname{N}_{H_{I,\mathcal{O}}/F}(\mathbb{A}_{H_{I,\mathcal{O}}}^\times) \subseteq \mathbb{A}_F^{\times}$ is also a closed subgroup of finite index containing the kernel of the Artin map and fixing precisely the field $H_{I,\mathcal{O}}$. Then by Galois theory we must have $F^{\times} \cdot U_{I,\mathcal{O}} = F^{\times} \cdot \operatorname{N}_{H_{I,\mathcal{O}}/F}(\mathbb{A}_{H_{I,\mathcal{O}}}^\times)$ and this concludes the proof.
\end{proof}

The previous description can be made more explicit by dividing the extension $F \subseteq H_{I,\mathcal{O}}$ in the two sub-extensions $F \subseteq H_{\mathcal{O}}$ and $H_{\mathcal{O}} \subseteq H_{I,\mathcal{O}}$.

\begin{proposition}
    Let $\mathcal{O}$ be an order inside a number field $F$. Then
    \[
        \operatorname{Gal}(H_{\mathcal{O}}/F) \cong \operatorname{Pic}(\mathcal{O})
    \]
    where $\operatorname{Pic}(\mathcal{O})$ denotes the class group of the order $\mathcal{O}$.
\end{proposition}
\begin{proof}
Combine \cite[Theorem and Definition~2.11]{Yi_Lv_2018} and \cite[Theorem~4.2]{Yi_Lv_2018}.
\end{proof}

\begin{theorem} \label{thm:galois_ray_class_fields}
    Let $F$ be a number field, $\mathcal{O} \subseteq \mathcal{O}_F$ be an order and $I \subseteq \mathcal{O}$ be a non-zero ideal. Then
    \[
        \operatorname{Gal}(H_{I,\mathcal{O}}/H_{\mathcal{O}}) \cong \frac{(\mathcal{O}/I)^{\times}}{\pi_I^{\times}(\mathcal{O}^{\times})}
    \]
    where $\pi_I^{\times} \colon \mathcal{O}^{\times} \to (\mathcal{O}/I)^{\times}$ is the map induced by the projection $\pi_I \colon \mathcal{O} \twoheadrightarrow \mathcal{O}/I$.
\end{theorem}
\begin{proof}
    First of all, we see that
    \[
    \begin{aligned}
        \operatorname{Gal}(H_{I,\mathcal{O}}/H_{\mathcal{O}}) &=
        \ker\left( \Gal(H_{I,\mathcal{O}}/F) \twoheadrightarrow \Gal(H_{\mathcal{O}}/F) \right)
        \overset{(a)}{\cong}
        \ker\left( \frac{\mathbb{A}_F^{\times}}{F^{\times} \cdot U_{I,\mathcal{O}}} \twoheadrightarrow \frac{\mathbb{A}_F^{\times}}{F^{\times} \cdot U_{\mathcal{O}}} \right) \cong \\[+4pt]
        &\cong \frac{F^{\times} \cdot U_{\mathcal{O}}}{F^{\times} \cdot U_{I,\mathcal{O}}} \cong \frac{F^{\times} \cdot U_{\mathcal{O}}/F^{\times}}{F^{\times} \cdot U_{I,\mathcal{O}}/F^{\times}} \overset{(b)}{\cong} \frac{U_{\mathcal{O}}/(F^{\times} \cap U_\mathcal{O})}{(U_{I,\mathcal{O}} \cdot (F^{\times} \cap U_\mathcal{O}))/(F^{\times} \cap U_\mathcal{O})} \cong \\[+4pt]
        &\cong \frac{U_{\mathcal{O}}}{U_{I,\mathcal{O}} \cdot (F^{\times} \cap U_\mathcal{O})} \overset{(c)}{=} \frac{U_{\mathcal{O}}}{U_{I,\mathcal{O}} \cdot \mathcal{O}^{\times}}
    \end{aligned}
    \]
    where $(a)$ comes from \cref{lem:U_I_open}, $(b)$ holds because $U_{I,\mathcal{O}} \subseteq U_{\mathcal{O}}$ and $(c)$ follows from the fact that $F^{\times} \cap U_{\mathcal{O}} = \mathcal{O}^{\times}$.
    
    Now, observe that $F_{\infty}^{\times} \subseteq U_\mathcal{O}$, where $F_{\infty} := F \otimes_{\mathbb{Q}} \mathbb{R} \cong \prod_{w \mid \infty} F_w \hookrightarrow \mathbb{A}_F$. Moreover, we have 
    \begin{equation} \label{eq:description_U_O}
        \frac{U_{\mathcal{O}}}{F_{\infty}^{\times}} \cong \prod_{p \in \mathbb{N}} \mathcal{O}_p^{\times} \cong \prod_{p \in \mathbb{N}} \varprojlim_{n \in \mathbb{N}} \left( \frac{\mathcal{O}}{p^n \mathcal{O}} \right)^\times \cong \varprojlim_{N \in \mathbb{Z}_{\geq 1}} \left( \frac{\mathcal{O}}{N \mathcal{O}} \right)^{\times} \cong \widehat{\mathcal{O}}^{\times}
    \end{equation}
    where the products run over the rational primes $p \in \mathbb{N}$, and $\mathcal{O}_p$ is the ring defined in \eqref{eq:definition_O_p}. 
    In the chain of isomorphisms \eqref{eq:description_U_O} the ring $\widehat{\mathcal{O}}$ is the profinite completion of $\mathcal{O}$, \textit{i.e.}
    \begin{equation} \label{eq:decomposition_O_hat}
        \widehat{\mathcal{O}} := \varprojlim_{N \in \mathbb{Z}_{\geq 1}} \frac{\mathcal{O}}{N \mathcal{O}} \cong \prod_{p \in \mathbb{N}} \mathcal{O}_p 
        \cong \prod_{\mathfrak{p} \subseteq \mathcal{O}} \mathcal{O}_{\mathfrak{p}}
    \end{equation}
    where the second product runs over all the non-zero prime ideals $\mathfrak{p} \subseteq \mathcal{O}$ and $\mathcal{O}_{\mathfrak{p}} := \varprojlim_{n \in \mathbb{N}} \mathcal{O}/\mathfrak{p}^n$ is the completion of $\mathcal{O}$ at the prime $\mathfrak{p}$. 
    The second isomorphism appearing in \eqref{eq:decomposition_O_hat} can be obtained by applying \cite[Corollary~7.6]{Eisenbud_1995} to $R = \mathbb{Z}_p$ and $A = \mathcal{O}_p$. This gives the decomposition
    \[
        \mathcal{O}_p \cong \prod_{\mathfrak{p} \supseteq p} \mathcal{O}_{\mathfrak{p}}
    \]
    where the product runs over all primes $\mathfrak{p} \subseteq \mathcal{O}$ lying above $p$.
    
    Under the isomorphism \eqref{eq:description_U_O} the subgroup $U_{I,\mathcal{O}}/F_{\infty}^{\times} \subseteq U_{\mathcal{O}}/F_{\infty}^{\times} \cong \widehat{\mathcal{O}}^{\times}$ is identified with the kernel of the map $\widehat{\pi_I}^{\times} \colon \widehat{\mathcal{O}}^{\times} \to (\widehat{\mathcal{O}}/I \widehat{\mathcal{O}})^{\times}$ induced by the projection $\widehat{\pi_I} \colon \widehat{\mathcal{O}} \twoheadrightarrow \widehat{\mathcal{O}}/I \widehat{\mathcal{O}}$.
    Hence
    \[
        \Gal(H_{I,\mathcal{O}}/H_{\mathcal{O}}) \cong \frac{U_{\mathcal{O}}}{U_{I,\mathcal{O}} \cdot \mathcal{O}^{\times}} \cong \frac{U_{\mathcal{O}}/F_{\infty}^{\times}}{(U_{I,\mathcal{O}} \cdot \mathcal{O}^{\times})/{F_{\infty}}^{\times}} \cong \frac{\widehat{\mathcal{O}}^{\times}}{\ker(\widehat{\pi_I}^{\times}) \cdot \mathcal{O}^{\times}} \cong
        \frac{(\widehat{\mathcal{O}}/I \widehat{\mathcal{O}})^{\times}}{\widehat{\pi_I}^{\times}(\mathcal{O}^{\times})} 
    \]
    because $\widehat{\pi_I}^{\times}$ is surjective. 
    This surjectivity is shown by the factorisation
    \[
    \begin{tikzcd}
        \widehat{\mathcal{O}}^{\times} \arrow[rr, "\widehat{\pi_I}^{\times}"] \arrow[rd, two heads] &  & \left( \widehat{\mathcal{O}}/I \widehat{\mathcal{O}} \right)^{\times} \\
        & \displaystyle\prod_{\mathfrak{p} \supseteq I} \mathcal{O}_{\mathfrak{p}}^{\times} \arrow[ru, two heads] & 
    \end{tikzcd}
    \]
    where the first map $\widehat{\mathcal{O}}^{\times} \twoheadrightarrow \prod_{\mathfrak{p} \supseteq I} \mathcal{O}_{\mathfrak{p}}^{\times}$ is surjective as follows from \eqref{eq:decomposition_O_hat}, and the second map
    \[
        \prod_{\mathfrak{p} \supseteq I} \mathcal{O}_{\mathfrak{p}}^{\times} \twoheadrightarrow \prod_{\mathfrak{p} \supseteq I} \left( \frac{\mathcal{O}_{\mathfrak{p}}}{I \mathcal{O}_{\mathfrak{p}}} \right)^{\times} \cong \left( \frac{\widehat{\mathcal{O}}}{I \widehat{\mathcal{O}}} \right)^{\times}
    \]
    is surjective by \cite[Corollary~2.3]{Chen_2020}, which can be applied since the ring $\prod_{\mathfrak{p} \supseteq I} \mathcal{O}_{\mathfrak{p}}$ has finitely many maximal ideals. 
    
    To finish our proof we need to show that
    \[
         \frac{(\widehat{\mathcal{O}}/I \widehat{\mathcal{O}})^{\times}}{\widehat{\pi_I}^{\times}(\mathcal{O}^{\times})} \cong  \frac{(\mathcal{O}/I)^{\times}}{\pi_I^{\times}(\mathcal{O}^{\times})}.
    \]
    To do this recall that $\pi_I$ and $\widehat{\pi_I}$ are related by the commutative diagram
    \[
    \begin{tikzcd}  
        \mathcal{O} \arrow[d,hook] \arrow[r,two heads,"\pi_I"] & \mathcal{O}/I \arrow[d] \arrow[r,"\gamma"] & \prod_{\mathfrak{p} \supseteq I} \frac{\mathcal{O}_{(\mathfrak{p})}}{I \mathcal{O}_{(\mathfrak{p})}} \arrow[d,"\beta"] \\
        \widehat{\mathcal{O}} \arrow[r,two heads,"\widehat{\pi_I}"] & \widehat{\mathcal{O}}/I \widehat{\mathcal{O}} \arrow[r,"\alpha","\sim" swap] & \prod_{\mathfrak{p} \supseteq I} \frac{\mathcal{O}_{\mathfrak{p}}}{I \mathcal{O}_{\mathfrak{p}}}
    \end{tikzcd}
    \]
    where $\alpha$ is the isomorphism coming from the decomposition \eqref{eq:decomposition_O_hat}, and $\beta$ and $\gamma$ are the maps induced by the natural inclusions $\mathcal{O} \subseteq \mathcal{O}_{(\mathfrak{p})} \subseteq \mathcal{O}_{\mathfrak{p}}$. Moreover the products run over all the prime ideals $\mathfrak{p} \subseteq \mathcal{O}$ such that $\mathfrak{p} \supseteq I$, and $\mathcal{O}_{(\mathfrak{p})}$ denotes the localisation of $\mathcal{O}$ at the prime $\mathfrak{p}$.
    
    Hence to conclude it is sufficient to observe that $\gamma$ is an isomorphism by \cite[Chapter~I, Proposition~12.3]{Neukirch_1999}, and $\beta$ is an isomorphism because 
    $\mathcal{O}$ is a one-dimensional Noetherian domain (see \cite[Chapter~I, Proposition~12.2]{Neukirch_1999}). 
    More explicitly, for any prime $\mathfrak{p} \subseteq \mathcal{O}$ such that $\mathfrak{p} \supseteq I$ we have that $\mathfrak{p} \cdot \mathcal{O}_{(\mathfrak{p})} = \sqrt{I \cdot \mathcal{O}_{(\mathfrak{p})}}$ because $\mathcal{O}_{(\mathfrak{p})}$ is a one-dimensional local ring.
    Hence \cite[Chapter~II, \S~2.6, Proposition~15]{Bourbaki_Commutative_Algebra} shows that $\mathcal{O}_{(\mathfrak{p})}/I\mathcal{O}_{(\mathfrak{p})}$ is complete with respect to $\mathfrak{p} \mathcal{O}_{(\mathfrak{p})}$. 
    Thus we can conclude that $\mathcal{O}_{(\mathfrak{p})}/I\mathcal{O}_{(\mathfrak{p})}$ is isomorphic to $\mathcal{O}_{\mathfrak{p}}/I \mathcal{O}_{\mathfrak{p}}$ using the exactness of completion, which holds because $\mathcal{O}_{(\mathfrak{p})}$ is Noetherian (see \cite[Lemma~7.15]{Eisenbud_1995}).
\end{proof}

\subsection{Ray class fields for imaginary quadratic orders} \label{sec:ray_class_fields_imaginary_quadratic_orders}
Since the definition of the ray class fields $H_{I,\mathcal{O}}$ is somehow implicit, a natural question would be to provide an explicit set of generators for the extension $F \subseteq H_{I,\mathcal{O}}$.
This can be done when $F = K$ is an imaginary quadratic field, and $I \subseteq \mathcal{O}$ is invertible, as we will see in \cref{prop:hilbert_12}.
In order to show this, we now introduce some notation concerning lattices in number fields, following \cite[Chapter~8]{Lang_1987}.

Let $F$ be a number field. A \textit{lattice} $\Lambda \subseteq F$ is an additive subgroup of $F$ which is free of rank $[F:\Q]$ over $\Z$. Given a pair of lattices $\Lambda_1, \Lambda_2 \subseteq F$ we can form their sum $\Lambda_1 + \Lambda_2 \subseteq F$, their product $\Lambda_1 \cdot \Lambda_2 \subseteq F$ and their quotient $(\Lambda_1 \colon \Lambda_2) := \{ x \in F \ \mid \ x \cdot \Lambda_2 \subseteq \Lambda_1 \} \subseteq F$. Moreover, it is possible to define an action of the idèle group of $F$ on the set $\{\Lambda \subseteq F: \Lambda \text{ lattice} \}$, as we are going to describe.

For a lattice $\Lambda \subseteq F$ and a prime $p\in\N$, denote by $\Lambda_p:=\Lambda \otimes_{\Z} \Z_p$ the completion of the lattice $\Lambda$ at $p$. Given an idèle $s = (s_p)_{p \in M_{\mathbb{Q}}} \in \mathbb{A}_F^{\times}$ there exists a unique lattice $s \cdot \Lambda \subseteq F$ with the property that $(s \cdot \Lambda)_p = s_p \cdot \Lambda_p$ for every prime $p \in \mathbb{N}$. This defines an action of the idèle group $\mathbb{A}^\times _F$ on the set of lattices in $F$, given by $(s, \Lambda) \mapsto s\cdot \Lambda$. We remark that the notation $s \cdot \Lambda$, although evocative of a multiplication between an idèle and a lattice, is purely formal and should not be confused with the notation $\Lambda_1 \cdot \Lambda_2$ for the usual product of lattices. Nevertheless, it is easy to see from the definitions that $(s \cdot \Lambda_1) \cdot \Lambda_2 = s \cdot (\Lambda_1 \cdot \Lambda_2)$ for every pair of lattices $\Lambda_1, \Lambda_2 \subseteq F$. Using the action just described, it is also possible to define a \textit{multiplication by $s$} map $F/\Lambda \xrightarrow{s \cdot} F/(s \cdot \Lambda)$ by means of the following commutative diagram
\[
    \begin{tikzcd}
        \displaystyle\frac{F}{\Lambda} \arrow[r,"s \cdot"] \arrow[d,"\sim" rotated_label]                              & \displaystyle\frac{F}{s \cdot \Lambda} \arrow[d,"\sim" rotated_label]                              \\
        \displaystyle\bigoplus_{p \in M_{\mathbb{Q}}^0} \frac{F_p}{\Lambda_p} \arrow[r,"(s_p \cdot \,)_p"] & \displaystyle\bigoplus_{p \in M_{\mathbb{Q}}^0} \frac{F_p}{s_p \Lambda_p}
    \end{tikzcd}
\]
where the vertical maps are the obvious isomorphisms induced by the inclusions $F \hookrightarrow F_p$ and the bottom map is given by $(x_p)_p \mapsto (s_p \, x_p)_p$.

The case of lattices inside an imaginary quadratic field $K$ is of particular interest for us. 
Indeed, if $\mathcal{O} \subseteq K$ is an order, any finitely generated $\mathcal{O}$-module $\Lambda \subseteq K$ is a lattice inside $K$.
Moreover, if we fix an embedding $K \hookrightarrow \mathbb{C}$, the quotient $\mathbb{C}/\Lambda$ can be canonically identified with the complex points $E(\mathbb{C})$ of an elliptic curve $E_{/\mathbb{C}}$ having complex multiplication by $\mathcal{O}$.
For any invertible ideal $I \subseteq \mathcal{O}$, the following \cref{prop:hilbert_12} 
shows that the extension $H_\mathcal{O} \subseteq H_{I,\mathcal{O}}$ can be obtained by adjoining to the ring class field $H_\mathcal{O}$ the values of the \textit{Weber function} $\mathfrak{h}_E \colon E \twoheadrightarrow E/\operatorname{Aut}(E) \cong \mathbb{P}^1$ (see \cite[Page~134]{si94}) at torsion points $z \in E[I] := E(\mathbb{C})[I]$. 

\begin{theorem} \label{prop:hilbert_12}
    Let $\mathcal{O}$ be an order inside an imaginary quadratic field $K \subseteq \mathbb{C}$, and let $I \subseteq \mathcal{O}$ be an invertible ideal. Then we have that
    \[
        H_{I,\mathcal{O}} = H_{\mathcal{O}}(\mathfrak{h}_E(E[I])) = K(j(E),\mathfrak{h}_{E}(E[I]))
    \]
    for any elliptic curve $E_{/ \mathbb{C}}$ such that $\End(E) \cong \mathcal{O}$. In particular, if $E$ is an elliptic curve defined over a number field $F$ such that $\End_F(E) \cong \mathcal{O}$ then $H_{I,\mathcal{O}} \subseteq F(E[I])$.
\end{theorem}
\begin{proof}
    By the previous discussion, we can assume that $j(E) \not\in \{0,1728\}$, because in this case $\mathcal{O} = \mathcal{O}_K$. 
    Recall that, since $I \subseteq \mathcal{O}$ is an invertible ideal, $E[I]$ is a free $\mathcal{O}/I$-module of rank one (see \cite[Lemma~2.4]{Bourdon_Clark_2020} or \cref{lem:free_O-module}).
    Fix a generator $P$ of $E[I]$ as a module over $\mathcal{O}/I$. 
    Then $H_{\mathcal{O}}(\mathfrak{h}_E(E[I])) = H_{\mathcal{O}}(\mathfrak{h}_E(P))$, as one can see by writing every endomorphism of $E$ in the standard form described in \cite[\S~2.9]{Washington_2008} and applying \cite[Chapter~I, Theorem~7]{Lang_1987}.
    
    Let now $\xi \colon \mathbb{C}/\mathfrak{a} \altxrightarrow{\sim} E(\mathbb{C})$ be a complex parametrisation, where $\mathfrak{a} \subseteq \mathcal{O}$ is an invertible ideal (see \cite[Proposition~4.8]{Shimura_1994}). 
    Fix moreover $z \in (\mathfrak{a} \colon I) \subseteq K \subseteq \mathbb{C}$ such that $\xi\left( \tilde{z} \right) = P$, where $\tilde{z} := (z + \mathfrak{a})/\mathfrak{a}$ denotes the image of $z$ in the quotient $K/\mathfrak{a} \subseteq \mathbb{C}/\mathfrak{a}$.
    Then \cite[Theorem~5.5]{Shimura_1994} shows that 
    \[
        H_{\mathcal{O}}(\mathfrak{h}_E(P)) = (K^{\text{ab}})^{[W_P,K]}
    \]
    where $W_P \subseteq \mathbb{A}_K^{\times}$ is the subgroup defined by $W_P := \left\{ s \in \mathbb{A}_K^{\times} \ \middle| \ s \cdot \mathfrak{a} = \mathfrak{a}, \ s \cdot \tilde{z} = \tilde{z} \right\}$.
    In particular, we recall that for any $s \in \mathbb{A}_K^{\times}$ such that $s \cdot \mathfrak{a} = \mathfrak{a}$ the notation $s \cdot \tilde{z}$ stands for the image of $\tilde{z} \in K/\mathfrak{a}$ under the map $K/\mathfrak{a} \altxrightarrow{s \cdot} K/\mathfrak{a}$. This map is defined by the commutative diagram
    \[
        \begin{tikzcd}
        \displaystyle\frac{K}{\mathfrak{a}} \arrow[r,"s \cdot"] \ar[d, "\sim" rotated_label]                              & \displaystyle\frac{K}{s \cdot \mathfrak{a}} \ar[d, "\sim" rotated_label] \arrow[Equal]{r}{} &[-0.7cm]  \displaystyle\frac{K}{\mathfrak{a}} \ar[d, "\sim" rotated_label] \\
        \displaystyle\bigoplus_{p \in M_{\mathbb{Q}}^0} \frac{K_p}{\mathfrak{a}_p} \arrow[r,"(s_p \cdot \,)_p"] & \displaystyle\bigoplus_{p \in M_{\mathbb{Q}}^0} \frac{K_p}{s_p \mathfrak{a}_p} \arrow[Equal]{r}{} & \displaystyle\bigoplus_{p \in M_{\mathbb{Q}}^0} \frac{K_p}{\mathfrak{a}_p}
        \end{tikzcd}
    \]
    where $\mathfrak{a}_p := \mathfrak{a} \otimes_{\mathbb{Z}} \mathbb{Z}_p = \mathfrak{a} \, \mathcal{O}_p$ for any rational prime $p \in \mathbb{N}$.
    Since $H_{\mathcal{O}} = K(j(E))$ the theorem will follow from the equality $W_P = U_{I,\mathcal{O}}$, where $U_{I,\mathcal{O}} \subseteq \mathbb{A}_K^{\times}$ is the subgroup defined in \eqref{eq:U_n}.
    
    To prove the inclusion $U_{I,\mathcal{O}} \subseteq W_P$ take any $s \in U_{I,\mathcal{O}}$. 
    Then $s \cdot \mathfrak{a} = \mathfrak{a}$ because $s_p \mathfrak{a}_p = \mathfrak{a}_p$ for every rational prime $p \in \mathbb{N}$, since by definition $s_p \in \mathcal{O}_p^{\times}$.
    Moreover, $s \cdot \tilde{z} = \tilde{z}$ because $z \in (\mathfrak{a} \colon I)$ and $s_p \in 1 + I \mathcal{O}_p$ for every rational prime $p \in \mathbb{N}$, which implies that $(s_p - 1) z \in \mathfrak{a}_p$.
    This shows that $U_{I,\mathcal{O}} \subseteq W_z$
    
    To prove the opposite inclusion $W_P \subseteq U_{I,\mathcal{O}}$ fix any rational prime $p \in \mathbb{N}$ and take $s \in W_P$, so that $s \cdot \mathfrak{a} = \mathfrak{a}$ and $s \cdot \tilde{z} = \tilde{z}$. 
    Since $\mathfrak{a} \subseteq \mathcal{O}$ is invertible we have that $\mathfrak{a} \cdot (\mathcal{O} \colon \mathfrak{a}) = \mathcal{O}$ and
    \[
        s \cdot \mathcal{O} = s \cdot (\mathfrak{a} \cdot (\mathcal{O} \colon \mathfrak{a})) = (s \cdot \mathfrak{a}) \cdot (\mathcal{O} \colon \mathfrak{a}) = \mathfrak{a} \cdot (\mathcal{O} \colon \mathfrak{a}) = \mathcal{O}
    \]
    which shows that $s_p \in \mathcal{O}_p^{\times}$. 
    Let us now prove that $s_p \in 1 + I \cdot \mathcal{O}_p$.
    Since $I \subseteq \mathcal{O}$ and $\mathfrak{a} \subseteq \mathcal{O}$ are both invertible we have that
    $I \cdot (\mathcal{O} \colon \mathfrak{a}) \cdot (\mathfrak{a} \colon I) = \mathcal{O}$, so that we can write $1 = \sum_{j = 1}^J \alpha_j \beta_j \tau_j$ with $\alpha_j \in I$, $\beta_j \in (\mathcal{O} \colon \mathfrak{a})$ and $\tau_j \in (\mathfrak{a} \colon I)$.
    Notice that $s \cdot \overline{\tau_j} = \overline{\tau_j}$ for every $j \in \{1,\dots,J\}$ because $s \cdot \tilde{z} = \tilde{z}$ and $P = \xi(\tilde{z})$ generates $E[I]$ as a module over $\mathcal{O}/I$.
    Hence $s_p - 1 \in I \cdot \mathcal{O}_p$ because we can write
    \[
        s_p - 1 = \sum_{j = 1}^J \alpha_j \, \beta_j \, (s_p \, \tau_j - \tau_j)
    \]
    where $s_p \, \tau_j - \tau_j \in \mathfrak{a}_p = \mathfrak{a} \, \mathcal{O}_p$ and $\beta_j (s_p \, \tau_j - \tau_j) \in \mathcal{O}_p$ since $\beta_j \in (\mathcal{O} \colon \mathfrak{a})$ for every $j \in \{1,\dots,J\}$.
    Thus we have shown that $s_p \in \mathcal{O}_p^{\times}$ and $s_p \in 1 + I \cdot \mathcal{O}_p$ for every prime $p \in \mathbb{N}$, which gives $W_P \subseteq U_{I,\mathcal{O}}$ as we wanted to prove.
\end{proof}

\begin{remark}
    The explicit description of the ray class fields given in \cref{prop:hilbert_12} shows that $H_{I,\mathcal{O}}$ coincides with the ray class field defined by Söhngen in \cite{Sohngen_1935} using the classical language of class field theory (see \cite[Chapter~IV, \S 7]{Neukirch_1999}).
    A more recent exposition of the work of Söhngen can be found in \cite[Theorem~6.2.3]{Schertz_2010}.
\end{remark}

\section{Minimality of division fields}
\label{sec:minimality}

We have seen in \cref{prop:ramification} that for every CM elliptic curve $E$ defined over a number field $F$ with $\End_F(E) \cong \mathcal{O}$ for some order $\mathcal{O}$ in an imaginary quadratic field $K \subseteq F$, the division fields $F(E[N])$ are \textit{maximal} for all integers $N$ coprime with a fixed integer $B_E \in \mathbb{Z}$. This is to say that the associated Galois representation $\rho_{E,N}$ given by \cref{lem:free_O-module} is surjective. 
When $E$ is defined over the ring class field $H_{\mathcal{O}}$ of $K$ relative to $\mathcal{O}$, the division fields $H_{\mathcal{O}}(E[N])$ always contain the ray class field $H_{N,\mathcal{O}} \subseteq K^{\text{ab}}$, as we have shown in \cref{prop:hilbert_12}. 
If the division field $H_{\mathcal{O}}(E[N])$ is maximal and $N > 2$ then the containment $H_{N,\mathcal{O}} \subseteq H_{\mathcal{O}}(E[N])$ is strict.
In this section we want to study for which integers $N$ the division fields are \textit{minimal}, in the sense that $H_{\mathcal{O}}(E[N])=H_{N,\mathcal{O}}$.
This investigation will yield to the proof of \cref{thm:infinitely_many_linearly_disjoint}, and will also be crucially used in \cref{sec:entanglement}.

We begin by studying how the maximality of division fields changes upon twisting.
Given an elliptic curve $E$ defined over a number field $F$ and an element $\alpha \in F^{\times}$, we denote by $E^{(\alpha)}$ the \textit{quadratic  twist} of $E$ by $\alpha$, as described in \cite[Chapter~X, \S~5]{si09}.
We recall that two twists $E^{(\alpha)}$ and $E^{(\alpha')}$ are isomorphic over $F$ if and only if $\alpha$ and $\alpha'$ represent the same class in $F^{\times}/(F^{\times})^2$, \textit{i.e.} if and only if $F(\sqrt{\alpha})=F(\sqrt{\alpha'})$.
\begin{proposition} \label{prop:proposition_twist}
    Let $\mathcal{O}$ be an order of discriminant $\Delta_{\mathcal{O}} < -4$ in an imaginary quadratic field $K$, and let $H_{\mathcal{O}}$ be the ring class field of $K$ relative to the order $\mathcal{O}$.
    Consider an elliptic curve $E_{/H_{\mathcal{O}}}$ with complex multiplication by $\mathcal{O}$ and fix $\alpha \in H_{\mathcal{O}}^{\times}$. Then for every invertible ideal $I \subseteq \mathcal{O}$ such that $I \cap \mathbb{Z} = N \mathbb{Z}$ with $N > 2$, the surjectivity of the Galois representation $\rho_{E,I}$ defined in \cref{lem:free_O-module} determines the surjectivity of $\rho_{E^{(\alpha)},I}$ as follows:
    \begin{enumerate}[label*=\protect\fbox{\arabic{enumi}}]
        \item if $\rho_{E,I}$ is surjective, then $\rho_{E^{(\alpha)},I}$ is surjective if and only if
        \[ H_{\mathcal{O}}(E[I]) \neq H_{I, \mathcal{O}}(\sqrt{\alpha}) \]
        where $H_{I, \mathcal{O}}$ is the ray class field of $K$ modulo $I$ relative to $\mathcal{O}$, defined in \cref{def:ray_class_field}; \label{it:twist_general_1}
        \item if $\rho_{E,I}$ is not surjective, then $\rho_{E^{(\alpha)},I}$ is surjective if and only if $H_\mathcal{O}(\sqrt{\alpha}) \neq H_\mathcal{O}$ and
        \[ H_{\mathcal{O}}(E[I]) \cap H_{\mathcal{O}}(\sqrt{\alpha})=H_{\mathcal{O}}. \] \label{it:twist_general_2}
    \end{enumerate}
\end{proposition}
\begin{proof}
    First of all, we claim that $\rho_{E,I}$ (respectively $\rho_{E^{(\alpha)},I}$) has maximal image if and only if there exists $\sigma \in \Gal(\overline{\Q}/H_{\mathcal{O}})$ such that $\rho_{E,I}(\sigma) = -1 \in (\mathcal{O}/I)^{\times}$ (respectively $\rho_{E^{(\alpha)},I}(\sigma) = -1$).
    Indeed, $H_{\mathcal{O}}(E[I])$ contains the ray class field $H_{I,\mathcal{O}}$, which is generated over $H_{\mathcal{O}}$ by the values of the Weber function $\mathfrak{h}_E \colon E \twoheadrightarrow E/\operatorname{Aut}(E) \cong \mathbb{P}^1$ at $I$-torsion points (see \cref{prop:hilbert_12}).
    Since $\mathfrak{h}_E([\varepsilon](P)) = \mathfrak{h}_E(P)$ for every $P \in E[I]$ and $\varepsilon \in \{\pm 1\} = \mathcal{O}^{\times} \cong \operatorname{Aut}(E)$, we see that $\rho_{E,I}$ induces the identification
    \begin{equation} \label{eq:galois_division_over_ray}
        \operatorname{Gal}(H_{\mathcal{O}}(E[I])/H_{I,\mathcal{O}}) \cong \operatorname{Im}(\pi_I^{\times}) \cap \operatorname{Im}(\rho_{E,I}) = \{\pm 1\} \cap \operatorname{Im}(\rho_{E,I}) \subseteq (\mathcal{O}/I)^{\times}
    \end{equation}
    where $\pi_I^{\times} \colon \mathcal{O}^{\times} \to (\mathcal{O}/I)^{\times}$ denotes the map induced by the quotient $\pi_I \colon \mathcal{O} \twoheadrightarrow \mathcal{O}/I$. Hence $\rho_{E,I}$ is surjective if and only if $-1 \in \operatorname{Im}(\rho_{E,I})$, and the same holds for $\rho_{E^{(\alpha)},I}$.
    Moreover $\rho_{E^{(\alpha)},I}$ is linked to $\rho_{E,I}$, after choosing compatible generators of $E[I]$ and $E^{(\alpha)}[I]$ as $\mathcal{O}/I$-modules, by the formula
    \begin{equation} \label{eq:twisted_character}
        \rho_{E^{(\alpha)},I}=\rho_{E,I} \cdot \chi_{\alpha}
    \end{equation}
    where $\chi_{\alpha} \colon \Gal(\overline{\Q}/H_{\mathcal{O}}) \to \{ \pm 1 \} \subseteq (\mathcal{O}/I)^{\times} $ is the quadratic character associated to $H_{\mathcal{O}}(\sqrt{\alpha})$.
    
    To prove \labelcref{it:twist_general_1} suppose that $\rho_{E,I}$ has maximal image.
    First, assume that $H_{\mathcal{O}}(E[I]) \neq H_{I, \mathcal{O}}(\sqrt{\alpha})$. 
    Then, either $H_{\mathcal{O}}(\sqrt{\alpha}) \cap H_{\mathcal{O}}(E[I])=H_{\mathcal{O}}$ or we have $H_{\mathcal{O}}(\sqrt{\alpha}) \subseteq H_{I,\mathcal{O}}$. 
    In the first case, we can certainly find $\sigma \in \Gal(\overline{\Q}/H_{\mathcal{O}})$ acting trivially on $H_{\mathcal{O}}(\sqrt{\alpha})$ and such that $\rho_{E,I}(\sigma)=-1$.
    Hence we can use \eqref{eq:twisted_character} to see that $\rho_{E^{(\alpha)},I}(\sigma)=\rho_{E,I}(\sigma) \cdot \chi_{\alpha}(\sigma)=-1$.
    This implies, by the initial discussion, that $\rho_{E^{(\alpha)},I}$ has maximal image. 
    In the second case, any $\sigma \in \Gal(\overline{\Q}/H_{\mathcal{O}})$ with $\rho_{E,I}(\sigma)=-1$ will act trivially on $H_{I,\mathcal{O}} \supseteq H_{\mathcal{O}}(\sqrt{\alpha})$ by \eqref{eq:galois_division_over_ray}.
    As before, we can use \eqref{eq:twisted_character} to conclude that $\rho_{E^{(\alpha)},I}$ has maximal image. 
    
    Assume now that $H_{\mathcal{O}}(E[I]) = H_{I, \mathcal{O}}(\sqrt{\alpha})$.
    This implies that the extensions $H_{\mathcal{O}} \subseteq H_{\mathcal{O}}(\sqrt{\alpha})$ and $H_{\mathcal{O}} \subseteq H_{I,\mathcal{O}}$ are linearly disjoint over $H_{\mathcal{O}}$, because
    $\rho_{E,I}$ has maximal image.
    In particular
    \[
        \Gal(H_{\mathcal{O}}(E[I])/H_{\mathcal{O}}) \cong \Gal(H_{I,\mathcal{O}}/H_{\mathcal{O}}) \times \Gal(H_{\mathcal{O}}(\sqrt{\alpha})/H_{\mathcal{O}}).
    \]
    We deduce that any $\sigma \in \Gal(\overline{\Q}/H_{\mathcal{O}})$ with $\rho_{E,I}(\sigma)=-1$, being the identity on $H_{I,\mathcal{O}}$ by \eqref{eq:galois_division_over_ray}, must act non-trivially on $H_{\mathcal{O}}(\sqrt{\alpha})$.
    Then \eqref{eq:twisted_character} gives
    \[
        \rho_{E^{(\alpha)},I}(\sigma)=\rho_{E,I}(\sigma) \cdot \chi_{\alpha}(\sigma)=1
    \]
    and this suffices to see that $\rho_{E^{(\alpha)},I}$ is non-maximal.
    This concludes the proof of \labelcref{it:twist_general_1}.
    
    The proof of \labelcref{it:twist_general_2} can be carried out in a similar fashion.
    First of all, notice that the non-maximality of $\rho_{E,I}$ and \eqref{eq:galois_division_over_ray} imply that $H_{I,\mathcal{O}} = H_{\mathcal{O}}(E[I])$.
    Now, by \eqref{eq:twisted_character} the only possibility for $\rho_{E^{(\alpha)},I}$ to be surjective in this case is to find an automorphism $\sigma \in \Gal(\overline{\Q}/H_{\mathcal{O}})$ with $\rho_{E,I}(\sigma)=1$ and $\chi_{\alpha}(\sigma)=-1$, which is clearly impossible if $H_{\mathcal{O}}(\sqrt{\alpha}) \subseteq H_{\mathcal{O}}(E[I])=H_{I,\mathcal{O}}$.
    On the other hand, if $H_{\mathcal{O}}(E[I]) \cap H_{\mathcal{O}}(\sqrt{\alpha}) = H_{\mathcal{O}}$ one can certainly find $\sigma \in \Gal(\overline{\Q}/H_{\mathcal{O}})$ such that $\chi_{\alpha}(\sigma) = -1$ and $\rho_{E,I}(\sigma)= 1$, which shows by \eqref{eq:twisted_character} that $\rho_{E^{(\alpha)},I}$ has maximal image.
\end{proof}
\begin{remark} \label{rmk:lozano-robledo_twist}
    Let $E$ be an elliptic curve with complex multiplication by an imaginary quadratic order $\mathcal{O}$ of discriminant $\Delta_\mathcal{O}$, and suppose that $E$ is defined over the ring class field $H_\mathcal{O}$.
    Fix a rational prime $p \in \mathbb{N}$ such that $p \nmid 2 \Delta_\mathcal{O}$ and $p \equiv \pm 1 \ \text{mod} \ 9$ if $\Delta_\mathcal{O} = -3$.
    Then the recent work \cite{Lozano-Robledo_2019} of Lozano-Robledo, and in particular \cite[Theorem~4.4.(5)]{Lozano-Robledo_2019} and \cite[Theorem~7.11]{Lozano-Robledo_2019}, show that for every $\alpha \in H_\mathcal{O}^\times$ and every $n \in \mathbb{N}$, the Galois representation $\rho_{E,p^n}$ is surjective if and only if $\rho_{E^{(\alpha)},p^n}$ is surjective.
    If moreover $\Delta_\mathcal{O} < -4$ then one can combine \labelcref{it:twist_general_1} of \cref{prop:proposition_twist} with \cref{rmk:lozano_robledo_arXiv} to show that $H_\mathcal{O}(E[p^n]) \neq H_{p^n,\mathcal{O}}(\sqrt{\alpha})$ for every $\alpha \in H_\mathcal{O}$ and every $n \in \mathbb{Z}_{\geq 1}$.
\end{remark}

In order to apply \cref{prop:proposition_twist} to entanglement questions, it is essential to identify elliptic curves having an infinite family of minimal division fields.
A first step in this direction is given by \cref{thm:coates_wiles}, which provides a sufficient condition on an elliptic curve $E$, ensuring the existence of an explicit set of invertible ideals $I \subseteq \mathcal{O}$ for which the corresponding division fields $H_{\mathcal{O}}(E[I])$ are minimal. 
The proof of this result crucially relies on \cref{thm:main_theorem_of_CM}, which describes the action of complex automorphisms on torsion points of a CM elliptic curve in terms of its analytic parametrisation.
The statement of the result involves the \textit{global Artin map} $[\cdot,F] \colon \mathbb{A}_F^{\times} \twoheadrightarrow \operatorname{Gal}(F^{\text{ab}}/F)$, which was already used in \cref{sec:appendix_ray_class_fields}, and the notion of \textit{Hecke character}. 
We recall that an \textit{Hecke character} on a number field $F$ is a continuous group homomorphism
\[
    \psi: \mathbb{A}_F^{\times} \to \C^{\times}
\]
such that $\psi(F^{\times})=1$. Given a Hecke character $\psi$ we denote by $\mathfrak{f}_\psi \subseteq \mathcal{O}_F$ its conductor, as defined in \cite[Chapter~16, Definition~5.7]{Husemoller_2004}. For every place $w \in M_F$ we denote by $\psi_w \colon F_w^{\times} \to \mathbb{C}^{\times}$ the group homomorphism $\psi_w := \psi \circ \iota_w$, where $\iota_w \colon F_w^{\times} \hookrightarrow \mathbb{A}_F^{\times}$ is the natural inclusion. 
Similarly, for every rational prime $p \in \N$ we denote by $\psi_p \colon F_p^{\times} \to \mathbb{C}^{\times}$ the group homomorphism $\psi_p := \psi \circ \iota_p$ where $\iota_p \colon F_p^{\times} \hookrightarrow \mathbb{A}_F^{\times}$ is the analogous inclusion induced by the decomposition \eqref{eq:decomposition_idèles_over_Q}.

\begin{theorem} \label{thm:main_theorem_of_CM}
    Let $F \subseteq \mathbb{C}$ be a number field, $E_{/F}$ be an elliptic curve such that $\End_F(E) \cong \mathcal{O}$ for some order $\mathcal{O}$ inside an imaginary quadratic field $K \subseteq F$. Let $K \subseteq M \subseteq F$ be a subfield such that $F(E_{\text{tors}}) \subseteq M^{\text{ab}} \cdot F$. Then there exist $[M^{\text{ab}} \cap F \colon M]$ group homomorphisms $\alpha \colon \mathbb{A}_M^{\times} \to K^{\times} \subseteq \mathbb{C}^{\times}$ such that: 
    \begin{itemize}
        \item the map $\varphi \colon \mathbb{A}_M^{\times} \to \mathbb{C}^{\times}$ defined as $\varphi(s) := \alpha(s) \cdot \operatorname{N}_{M/K}(s)_{\infty}^{-1}$ is a Hecke character, where $\operatorname{N}_{M/K} \colon \mathbb{A}_M^{\times} \to \mathbb{A}_{K}^{\times}$ is the idelic norm map described for example in \cite[Chapter VI, \S~2]{Neukirch_1999};
        \item for every lattice $\Lambda \subseteq K \subseteq \mathbb{C}$, every analytic isomorphism $\xi \colon \mathbb{C}/\Lambda \altxrightarrow{\sim} E(\mathbb{C})$ and every $s \in M^{\times} \cdot \operatorname{N}_{F/M}(\mathbb{A}_F^{\times}) \subseteq \mathbb{A}_M^{\times}$ we have that $(\alpha(s) \cdot \operatorname{N}_{M/K}(s)^{-1}) \cdot \Lambda = \Lambda$ and the following diagram
        \[
        \begin{tikzcd}[column sep = large]
            K/\Lambda \arrow[d, "\xi"'] \arrow[r, "\left( \alpha(s) \cdot \operatorname{N}_{M/K}(s)^{-1} \right) \cdot"] & K/\Lambda \arrow[d, "\xi"] \\
            E(M^{\text{ab}} \cdot F) \arrow[r, "\tau"'] & E(M^{\text{ab}} \cdot F)
        \end{tikzcd}
        \]
        commutes, where $\tau \in \Gal(M^{\text{ab}} \cdot F/F)$ is the unique automorphism such that $\restr{\tau}{M^{\text{ab}}} = [s,M]$.
    \end{itemize}
\end{theorem}
\begin{proof}
    Combine \cite[Proposition~7.40]{Shimura_1994} and \cite[Proposition~7.41]{Shimura_1994} when $M = F$ and use \cite[Theorem~7.44]{Shimura_1994} for the general case. Notice that, by class field theory, for every $s \in M^{\times} \cdot \operatorname{N}_{F/M}(\mathbb{A}_F^{\times})$ the restriction $\restr{[s,M]}{M^{\text{ab}} \cap F}$ is trivial. This gives a unique $\tau \in \Gal(M^{\text{ab}} \cdot F/F)$ such that $\restr{\tau}{M^{\text{ab}}} = [s,M]$. Moreover, fixing an embedding $F \subseteq \mathbb{C}$ automatically fixes an embedding $M^{\text{ab}} \cdot F \subseteq \mathbb{C}$, hence $E(M^{\text{ab}} \cdot F) \subseteq E(\mathbb{C})$, which gives a meaning to the vertical arrows in the diagram.
\end{proof}

\begin{remark}
    If $K \subseteq M \subseteq M' \subseteq F$ and $F(E_{\text{tors}}) \subseteq M^{\text{ab}}$ then $M \subseteq F$ is abelian and \cref{thm:main_theorem_of_CM} gives us $[M^{\text{ab}} \cap F \colon M] = [F \colon M]$ Hecke characters $\varphi \colon \mathbb{A}_M^{\times} \to \mathbb{C}^{\times}$ and $[(M')^{\text{ab}} \cap F \colon M'] = [F \colon M']$ Hecke characters $\widetilde{\varphi} \colon \mathbb{A}_{M'}^{\times} \to \mathbb{C}^{\times}$.
    We can observe that
    \[
        \frac{[M^{\text{ab}} \cap F \colon M]}{[(M')^{\text{ab}} \cap F \colon M']} = \frac{[F \colon M]}{[F \colon M']} = [M' \colon M] \in \mathbb{N}
    \]
    and that for every Hecke character $\widetilde{\varphi} \colon \mathbb{A}_{M'}^{\times} \to \mathbb{C}^{\times}$ given by \cref{thm:main_theorem_of_CM} there are exactly $[M' \colon M]$ Hecke characters $\varphi \colon \mathbb{A}_M^{\times} \to \mathbb{C}^{\times}$ such that $\widetilde{\varphi} = \varphi \circ \operatorname{N}_{M'/M}$. 
    If $K = M$ and  $F = M'$ then we have a unique Hecke character $\widetilde{\varphi} \colon \mathbb{A}_{F}^{\times} \to \mathbb{C}^{\times}$ which coincides with the usual Hecke character associated to elliptic curves with complex multiplication, defined for example in \cite[Chapter~II, \S~9]{si94} and \cite[Chapter~10, Theorem~9]{Lang_1987}.
\end{remark}

We can now state \cref{thm:coates_wiles}, recalling that for every order $\mathcal{O}$ contained in an imaginary quadratic field $K$ and every ideal $I \subseteq \mathcal{O}$ we denote by $H_{I,\mathcal{O}}$ the ray class field of $K$ modulo $I$ relative to the order $\mathcal{O}$, as defined in \cref{sec:appendix_ray_class_fields}.

\begin{theorem}\label{thm:coates_wiles}
    Let $F \subseteq \mathbb{C}$ be a number field and let $E_{/F}$ be an elliptic curve such that $\End_F(E) \cong \mathcal{O}$ for some order $\mathcal{O}$ inside an imaginary quadratic field $K \subseteq F$.
    Suppose that $F(E_{\text{tors}}) \subseteq K^{\text{ab}}$.
    Let $H := H_{\mathcal{O}}$ the ring class field of $\mathcal{O}$, and fix $\alpha \colon \mathbb{A}_K^{\times} \to \mathbb{C}^{\times}$ as in \cref{thm:main_theorem_of_CM}, with $M = K$.
    Then we have that $F(E[I]) = F \cdot H_{I,\mathcal{O}}$ for every invertible ideal $I \subseteq \mathcal{O}$ such that $I \subseteq \mathfrak{f}_{\varphi} \cap \mathcal{O}$, where $\mathfrak{f}_{\varphi} \subseteq \mathcal{O}_K$ is the conductor of the Hecke character $\varphi \colon \mathbb{A}_K^{\times} \to \mathbb{C}^{\times}$ defined by $\varphi(s) := \alpha(s) \cdot s_{\infty}^{-1}$.
\end{theorem}
\begin{proof}
    The containment $H_{I,\mathcal{O}} \subseteq F(E[I])$ is given by \cref{prop:hilbert_12}. Observe moreover that $K \subseteq F$ is an abelian extension, since $F \subseteq F(E_{\text{tors}}) \subseteq K^{\text{ab}}$ by assumption. 
    Hence to prove that $F(E[I]) \subseteq F \cdot H_{I,\mathcal{O}}$ it is sufficient to show that every $I$-torsion point of $E$ is fixed by $[s,K]$, for any $s \in \mathbb{A}_K^{\times}$ such that $\restr{[s,K]}{H_{I,\mathcal{O}}} = \operatorname{Id}$.
    Moreover, it suffices to consider only those $s \in \mathbb{A}_K^{\times}$ such that $s_{\infty} = 1$ and $s \in U_{I,\mathcal{O}}$, where $U_{I,\mathcal{O}} \leq \mathbb{A}_K^{\times}$ is the subgroup defined in \eqref{eq:U_n}. 
    This follows from the fact that $[U_{I,\mathcal{O}},K] = \Gal(K^{\text{ab}}/H_{I,\mathcal{O}})$ and 
    $K_{\infty}^{\times} \subseteq \ker([\cdot,K]) \cap U_{I,\mathcal{O}}$ by \cref{def:ray_class_field} and \cref{lem:U_I_open}.
    
    Fix then $s \in U_{I,\mathcal{O}}$ with $s_{\infty} = 1$.
    To study the action of $[s,K]$ on $E[I]$, we fix an invertible ideal $\mathfrak{a} \subseteq \mathcal{O} \subseteq \mathbb{C}$ and a complex uniformisation $\xi \colon \mathbb{C}/\mathfrak{a} \altxrightarrow{\sim} E(\mathbb{C})$, which exists by \cite[Proposition~4.8]{Shimura_1994}.
    Take a torsion point $P \in E[I]$, and let $z \in (\mathfrak{a} \colon I)$ be any element such that $\xi(\tilde{z}) = P$, where $\tilde{z} \in (\mathfrak{a} \colon I)/\mathfrak{a}$ denotes the image of $z$ in the quotient.
    Since $s \in K^{\times} \cdot \operatorname{N}_{H/K}(\mathbb{A}_H^{\times})$, we have that
    \[
        P^{[s,K]} = \xi(\tilde{z})^{[s,K]} = \xi\left( ( \alpha(s) \, s^{-1} ) \cdot \tilde{z} \right)
    \]
    which follows from applying \cref{thm:main_theorem_of_CM} with $M = K$.
    This can be applied because
    \[
        s \in U_{I,\mathcal{O}} \subseteq U_{\mathcal{O}} \subseteq K^{\times} \cdot U_{\mathcal{O}} = K^{\times} \cdot \operatorname{N}_{H/K}(\mathbb{A}_H^{\times})
    \]
    where the last equality is given by \cref{lem:U_I_open}.
    
    To conclude, it suffices to show that $s^{-1} \cdot \tilde{z} = \tilde{z}$ and $\alpha(s) = 1$. Notice that $s^{-1} \cdot \mathfrak{a} = \mathfrak{a}$ because $\mathfrak{a} \subseteq \mathcal{O}$ is invertible and $s_p \in \mathcal{O}_p^{\times}$ for every rational prime $p \in \mathbb{N}$. The equality $s^{-1} \cdot \tilde{z} = \tilde{z}$ then follows from the fact that, for every prime $p \in \mathbb{N}$, we have $s_p^{-1} \, z - z \in \mathfrak{a}_p$ because $z \in (\mathfrak{a} \colon I)$ and $s_p^{-1} \in 1 + I \, \mathcal{O}_p$.
    To prove the equality $\alpha(s) = 1$, notice that for every prime $p \in \mathbb{N}$ we have
    \[
        1 + I \, \mathcal{O}_p \subseteq \prod_{\substack{w \mid p \\ w \in M_K^{0}}} (1 + \mathfrak{f}_{\varphi} \, \mathcal{O}_{K_w})
    \]
    since $I \subseteq \mathfrak{f}_{\varphi} \cap \mathcal{O}$ by assumption. This implies that $\varphi_p(s_p) = 1$ for every prime $p \in \mathbb{N}$. Indeed $s_p \in 1 + I \, \mathcal{O}_p$ by the definition of $U_{I,\mathcal{O}}$ and for every $w \in M_K^0$ we have that $\varphi_w(1 + \mathfrak{f}_{\varphi} \, \mathcal{O}_{K_w}) = 1$ because $\mathfrak{f}_{\varphi}$ is the conductor of $\varphi$.
    Since $s_{\infty} = 1$ we get that $\alpha(s) = \varphi(s) = 1$, as was to be shown. 
\end{proof}

\begin{remark} \label{rmk:coates_wiles_attribution}
    \cref{thm:coates_wiles} has been proved by Coates and Wiles (see \cite[Lemma~3]{Coates_Wiles_1977}) if $\mathcal{O} = \mathcal{O}_K$ is a maximal order of class number one.
    Their result has been generalised in the PhD thesis of Kuhman (see \cite[Chapter~II, Lemma~3]{KUHMAN_1978}) to maximal orders $\mathcal{O} = \mathcal{O}_K$, under the hypothesis that $F \subseteq H_{I,\mathcal{O}_K}$.
\end{remark}

\cref{thm:coates_wiles} has a partial converse, as we show in the following proposition.

\begin{proposition} \label{prop: converse Coates-Wiles}
    Let $\mathcal{O}$ be an order in an imaginary quadratic field $K$ and $F\supseteq K$ be an abelian extension. Let $E_{/F}$ be an elliptic curve with complex multiplication by the order $\mathcal{O}$. Suppose that there exists an invertible ideal $I\subseteq \mathcal{O}$ such that $F(E[I])=F\cdot H_{I,\mathcal{O}}$ and $I\cap \Z = N\Z$ with $N>2$ if $j(E) \neq 0$ or $N>3$ if $j(E)=0$. Then $F(E_{\text{tors}})=K^{\text{ab}}$.
\end{proposition}

\begin{proof}
    It is sufficient to prove that $F(E_{\text{tors}}) \subseteq K^{\text{ab}}$, since the other inclusion follows from the class field theory of imaginary quadratic fields and the fact that $K \subseteq F$ is abelian.
    
    Fix an embedding $K\hookrightarrow \C$ and let $\xi: \C/\Lambda \altxrightarrow{\sim} E(\C)$ be a complex parametrization for $E$, where $\Lambda \subseteq K$ is a lattice. Take $\sigma \in \Aut(\C/K^{\text{ab}})$. By \cite[Theorem 5.4]{Shimura_1994} with $s=1$, there exists a complex parametrization $\xi': \C/\Lambda \altxrightarrow{\sim} E(\C)$ such that the following diagram 
    \[
    \begin{tikzcd}
    E(\C) \arrow[rr, "\sigma"] &                                                 & E(\C) \\
                          & K/\Lambda \arrow[lu, "\xi"] \arrow[ru, "\xi'"'] &     
    \end{tikzcd}
    \]
    commutes. This means that $\sigma$ acts on $E_{\text{tors}}$ as an automorphism $\gamma=\xi' \circ \xi^{-1} \in \Aut(E) \cong \mathcal{O}^{\times}$. In particular, for any point $P \in E[I]$ we have
    \begin{equation}\label{who's gamma}
        \gamma (P) = \sigma (P)=P
    \end{equation}
    since by assumption $F(E[I])=F\cdot H_{I,\mathcal{O}} \subseteq K^{\text{ab}}$. Notice now that if $j(E)\neq 0,1728$ we have $\Aut(E)=\{\pm1\}$ and equality \eqref{who's gamma} can occur for $\gamma=-1$ only when $I \cap \mathbb{Z} = 2 \mathbb{Z}$. Similarly, if $j(E)=1728$ or $j(E)=0$ one sees that a non-trivial element of $\Aut(E)$ can possibly fix only points of $E[2]$ or points of $E[2] \cup E[3]$, respectively. Our assumptions on $I$ allow then to conclude that $\gamma$ must be the identity on $E$.
    
    We have shown that every complex automorphism which fixes the maximal abelian extension of $K$ fixes also the torsion points of $E$. We conclude that $F(E_{\text{tors}}) \subseteq K^{\text{ab}}$ and this finishes the proof.
\end{proof}

As a consequence of \cref{prop: converse Coates-Wiles} we deduce that, for an elliptic curve $E$ with complex multiplication by an order $\mathcal{O}$ in an imaginary quadratic field $K$ which is defined over the ring class field $H_\mathcal{O}$, the whole family of division fields $\{H_\mathcal{O}(E[p^\infty])\}_p$ is linearly disjoint over $H_\mathcal{O}$ as soon as the extension $K \subseteq H_\mathcal{O}(E_\text{tors})$ is not abelian.

\begin{corollary} \label{cor:index_of_Galois}
    Let $\mathcal{O}$ be an order inside an imaginary quadratic field $K$, and let $E_{/H_{\mathcal{O}}}$ be an elliptic curve with complex multiplication by $\mathcal{O}$.
    Then we have that 
    \begin{equation} \label{eq:cm_index_formula}
        \lvert \operatorname{Aut}_{\mathcal{O}}(E_{\text{tors}}) \colon \operatorname{Im}(\rho_E) \rvert = \begin{cases}
            \lvert \mathcal{O}^\times \rvert,  &\text{if} \ K \subseteq H_{\mathcal{O}}(E_{\text{tors}}) \ \text{is abelian,} \\
            1,  &\text{otherwise.}
        \end{cases}
    \end{equation}
    
    In particular, if $H_{\mathcal{O}}(E_{\text{tors}}) \not\subseteq K^\text{ab}$ then all the Galois representations $\rho_{E,p^n}$ defined in \cref{lem:free_O-module} are isomorphisms, and the family of division fields $\{H_\mathcal{O}(E[p^\infty])\}_p$ is linearly disjoint over $H_\mathcal{O}$.
\end{corollary}
\begin{proof}
    Suppose that $K \subseteq H_{\mathcal{O}}(E_{\text{tors}})$ is not abelian. Since $H_{\mathcal{O}}(E_{\text{tors}}) \subseteq H_{\mathcal{O}}^{\text{ab}}$ this shows in particular that $K \neq H_{\mathcal{O}}$ and hence that $j(E) \not\in \{0,1728\}$.
    Then \cref{prop: converse Coates-Wiles} shows that 
    \[
    H_{\mathcal{O}}(E[N]) \neq H_{N,\mathcal{O}}
    \] 
    for every $N \in \mathbb{N}$ with $N \geq 2$.
    Since $j(E) \not\in \{0,1728\}$ this implies that the Galois representation 
    \[
    \rho_{E,N} \colon \operatorname{Gal}(H_\mathcal{O}(E[N])/H_\mathcal{O}) \to (\mathcal{O}/N \mathcal{O})^\times
    \] 
    introduced in \cref{lem:free_O-module} is an isomorphism
    for every $N \in \mathbb{Z}_{\geq 1}$. 
    Hence the family of division fields $\{H_\mathcal{O}(E[p^\infty])\}_p$ is linearly disjoint over $H_\mathcal{O}$ and $\operatorname{Im}(\rho_E) = \operatorname{Aut}_{\mathcal{O}}(E_{\text{tors}})$.
    
    Suppose now that $K \subseteq H_{\mathcal{O}}(E_{\text{tors}})$ is abelian. Then \cref{thm:coates_wiles} shows that there exists $N \in \mathbb{N}$ such that for every $M \in \mathbb{N}$ with $N \mid M$ we have that 
    $H_{\mathcal{O}}(E[M]) = H_{M,\mathcal{O}}$.
    Combining this with \cref{thm:galois_ray_class_fields}
    we get that $[\operatorname{Aut}_{\mathcal{O}}(E_{\text{tors}}) \colon \operatorname{Im}(\rho_E)] \geq \lvert \mathcal{O}^\times \rvert$.
    However, \cref{thm:galois_ray_class_fields} and \cref{prop:hilbert_12} imply that $[\operatorname{Aut}_{\mathcal{O}}(E_{\text{tors}}) \colon \operatorname{Im}(\rho_E)] \leq \lvert \mathcal{O}^\times \rvert$, which allows us to conclude.
\end{proof}
\begin{remark} \label{rmk:bourdon_clark_1-5}
    We point out that \cite[Corollary~1.5]{Bourdon_Clark_2020} proves, for every elliptic curve $E$ with CM by $\mathcal{O}$ and defined over a number field $F \supseteq H_\mathcal{O}$, that the index $\lvert \operatorname{Aut}_{\mathcal{O}}(E_{\text{tors}}) \colon \operatorname{Im}(\rho_E) \rvert$ always divides $\lvert \mathcal{O}^\times \rvert \cdot [F \colon H_\mathcal{O}]$.
    In the case $F = H_\mathcal{O}$, this is a consequence of \cref{cor:index_of_Galois}.
    Moreover, \eqref{eq:cm_index_formula} admits a generalisation to CM elliptic curves defined over any number field. 
    This generalisation, contained in the forthcoming work \cite{Campagna_Pengo_II}, fully recovers  \cite[Corollary~1.5]{Bourdon_Clark_2020}. 
\end{remark}

\begin{remark} \label{rmk:lozano_robledo_1728}
    The previous \cref{cor:index_of_Galois} generalises \cite[Theorem~1.3]{Lozano-Robledo_2019}, whose proof will appear in the forthcoming work \cite{Lozano_Robledo_prep}.
    Indeed, if $E_{/\mathbb{Q}}$ is an elliptic curve with complex multiplication by an order $\mathcal{O}$ in an imaginary quadratic field $K$ then we clearly have that $K(E_\text{tors}) \subseteq K^\text{ab}$, hence \cref{cor:index_of_Galois} shows that the Galois representation $\rho_E \colon \operatorname{Gal}(K(E_\text{tors})/K) \hookrightarrow \widehat{\mathcal{O}}^\times$ is not surjective.
    Let now $\widetilde{\rho}_E \colon \operatorname{Gal}(\overline{\mathbb{Q}}/\mathbb{Q}) \to \mathcal{N}_{\delta,\phi}$ be the Galois representation associated to the elliptic curve $E$ over $\mathbb{Q}$, where $\mathcal{N}_{\delta,\phi} \subseteq \operatorname{GL}_2(\widehat{\mathbb{Z}})$ is the subgroup defined by Lozano-Robledo in \cite[Theorem~1.1]{Lozano-Robledo_2019}.
    Then \cite[Theorem~1.1.(2)]{Lozano-Robledo_2019} and \cref{cor:index_of_Galois} show that
    \[
        [\mathcal{N}_{\delta,\phi} \colon \operatorname{Im}(\widetilde{\rho}_E)] = [\widehat{\mathcal{O}}^\times \colon \operatorname{Im}(\rho_E)] = \lvert \mathcal{O}^\times \rvert
    \]
    hence we get that $\widetilde{\rho}_E$ is not surjective.
    In particular, if $j(E) = 1728$ as in \cite[Theorem~1.3]{Lozano-Robledo_2019} we get that $[\mathcal{N}_{\delta,\phi} \colon \operatorname{Im}(\widetilde{\rho}_E)] = 4$.
\end{remark}

We have seen that, for a CM elliptic curve $E$ defined over an abelian extension $F$ of the CM field $K$, having infinitely many minimal division fields is equivalent to the property that torsion points of $E$ generate abelian extensions of $K$ (and not only of $F$). 
It seems then natural to ask whether, for a fixed order $\mathcal{O}$ in an imaginary quadratic field $K$, there exists any elliptic curve $E$ with complex multiplication by $\mathcal{O}$ and defined over the ring class field $H_{\mathcal{O}}$ (the smallest possible field of definition for $E$) with the property that $H_{\mathcal{O}}(E_{\text{tors}})=K^{\text{ab}}$. 
To the best of the authors' knowledge, this question was first discussed by Shimura in \cite[Page~217]{Shimura_1994} and subsequently studied by various authors, including Shimura himself \cite[\S~5]{Shimura_1971}, Robert \cite{Robert_1983} and more recently Gurney \cite[\S~4]{Gurney_2019}. 
One of the main outcomes of these investigations is the following: for every order $\mathcal{O}$ in an imaginary quadratic field $K \neq \mathbb{Q}(i)$, there exists an elliptic curve $E_{/H_\mathcal{O}}$ satisfying $H_{\mathcal{O}}(E_{\text{tors}})=K^{\text{ab}}$. Moreover, the same is true for orders $\mathcal{O} \subseteq \mathbb{Q}(i)$ if and only if either $\mathcal{O}=\mathbb{Z}[i]$ or the conductor $\mathfrak{f}_\mathcal{O} := \lvert \mathbb{Z}[i] \colon \mathcal{O} \rvert$ is divisible by at least one prime $p \not \equiv 1 \text{ mod } 4$. 
Using the aforementioned references, one can deduce this result in different ways, for instance combining \cite[Corollaire~3, Page~8-08]{Robert_1983}, \cite[Corollaire~1, Bottom~of~page~8-12]{Robert_1983} and \cite[Corollaire~2, Page~8-14]{Robert_1983}. 
However, none of these arguments seems to provide a way of finding, when possible, an explicit elliptic curve $E_{/H_\mathcal{O}}$ satisfying the property $H_{\mathcal{O}}(E_{\text{tors}})=K^{\text{ab}}$. We therefore decided to give a different proof of the above result, which yields an explicit construction of infinitely many such elliptic curves.

\begin{theorem} \label{thm:Shimura_condition_yes}
    Let $\mathcal{O}$ be an order of discriminant $\Delta_\mathcal{O} \in \mathbb{Z}$ inside an imaginary quadratic field $K$, and let $j \in H_{\mathcal{O}}$ be the $j$-invariant of any elliptic curve with complex multiplication by $\mathcal{O}$. 
    Then:
    \begin{enumerate}[label=(\alph*)]
        \item \label{it:shimura_yes_A} if $\Delta_\mathcal{O} \neq -4 f^2$ for every $f \in \mathbb{Z}_{\geq 2}$ which is only divisible by primes $p \equiv 1 \text{ mod } 4$, there exist infinitely many elliptic curves $E_{/H_{\mathcal{O}}}$, pairwise non-isomorphic over $H_\Ogotic$, with $j(E)=j$ and such that $H_{\mathcal{O}}(E_{\text{tors}})=K^{\text{ab}}$;
        \item \label{it:shimura_yes_B} if $\Delta_\mathcal{O} = -4 f^2$ for some $f \in \mathbb{Z}_{\geq 2}$ which is only divisible by primes $p \equiv 1 \text{ mod } 4$, then $H_{\mathcal{O}}(E_{\text{tors}}) \neq K^{\text{ab}}$ for every elliptic curve $E_{/H_\mathcal{O}}$ with $j(E) = j$.
    \end{enumerate}
\end{theorem}
\begin{proof}
    We begin by proving \labelcref{it:shimura_yes_A}.
    When $\mathcal{O}$ has class number $1$ the statement is trivially true. We may then assume that $\text{Pic}(\mathcal{O}) \neq \{1\}$, and in particular that $\Delta_\mathcal{O} < -4$.
    We fix moreover ${E_0}_{/H_{\mathcal{O}}}$ to be any elliptic curve with $j(E_0)=j$.
    
    Suppose first of all that $K \neq \mathbb{Q}(i)$, where $i^2 = -1$.
    Let $p \in \N$ be a prime satisfying
    \begin{enumerate}[label*=\protect\fbox{\arabic{enumi}}]
        \item $p \equiv 3 \text{ mod } 4$, \textit{i.e.} $p$ is inert in $\mathbb{Q}(i)$; \label{it:prime_condition_1}
        \item $p$ does not divide $\mathfrak{f}_{\mathcal{O}} \cdot N_{H_{\mathcal{O}}/\Q}(\fgotic_{E_0})$, where $\mathfrak{f}_{\mathcal{O}} := \lvert \mathcal{O}_K \colon \mathcal{O} \rvert$ denotes the conductor of the order $\mathcal{O}$ and $\fgotic_{E_0} \subseteq \mathcal{O}_{H_{\mathcal{O}}}$ is the conductor ideal of the elliptic curve $E_0$; \label{it:prime_condition_2}
        \item $p$ splits completely in $K$. \label{it:prime_condition_3}
    \end{enumerate}
    Since we are assuming that $K \neq \mathbb{Q}(i)$, there are infinitely many such primes, as follows by Dirichlet's theorem on primes in arithmetic progressions (see \cite[Chapter~VII, Theorem~5.14]{Neukirch_1999}). 
    
    Let $\pgotic \subseteq \mathcal{O}$ be a prime ideal lying over $p$ and note that $\pgotic$ is invertible by condition \labelcref{it:prime_condition_2}.
    We define a new elliptic curve $E_{\pgotic}$ over $H_{\Ogotic}$, as follows.
    By \cref{prop:ramification} there is an isomorphism
    \[
        \Gal(H_{\mathcal{O}}(E_0[\pgotic])/H_{\mathcal{O}}) \cong \left( \mathcal{O}/\pgotic \mathcal{O}\right)^{\times} \cong \F_p^{\times}
    \]
    where the last isomorphism follows from the fact that $p$ splits in $K$.
    In particular, the group $\Gal(H_{\mathcal{O}}(E_0[\pgotic])/H_{\mathcal{O}})$ is cyclic of order $p-1$, so $H_{\mathcal{O}} \subseteq H_{\mathcal{O}}(E_0[\pgotic])$ contains unique sub-extensions of degree $(p-1)/2$ and of degree 2 over $H_{\mathcal{O}}$.
    The first one is necessarily the ray class field $H_{\pgotic,\mathcal{O}}$ (see \cref{prop:hilbert_12}), the second one is of the form $H_{\mathcal{O}}(\sqrt{\alpha})$ for some element $\alpha = \alpha_\pgotic \in H_{\mathcal{O}}^{\times}$. By condition \labelcref{it:prime_condition_1}, the integer $p-1$ is not divisible by $4$, hence these two extensions must be linearly disjoint over $H_{\mathcal{O}}$. We deduce that $H_{\mathcal{O}}(E_0[\pgotic])=H_{\pgotic,\mathcal{O}}(\sqrt{\alpha})$. We set $E_\pgotic:=E_0^{(\alpha)}$, where $E_0^{(\alpha)}$ denotes the twist of $E_0$ by $\alpha \in H_{\mathcal{O}}^{\times}$.
    
    By \cref{prop:proposition_twist}, the Galois representation 
    \[ \rho_{E_\pgotic, \pgotic}:  \Gal(H_{\mathcal{O}}(E_\pgotic[\pgotic])/H_{\mathcal{O}}) \hookrightarrow \left( \mathcal{O}/\pgotic \mathcal{O}\right)^{\times} \]
    is not surjective. 
    This in particular implies that $H_{\mathcal{O}}(E_\pgotic[\pgotic])=H_{\pgotic,\mathcal{O}}$. It follows then from \cref{prop: converse Coates-Wiles} that $H_{\mathcal{O}}((E_{\pgotic})_{ \text{ tors }})= K^{\text{ab}}$.
    
    We claim that the infinitely many elliptic curves $E_\pgotic$ with $\pgotic \subseteq \mathcal{O}$ chosen as above, are pairwise non-isomorphic over $H_{\mathcal{O}}$. 
    To show this, it suffices to prove that the fields $H_{\mathcal{O}}(\sqrt{\alpha_\pgotic})$ associated to the quadratic twists are pairwise distinct. But this follows from \cref{prop:unramified} and \cref{prop:ramification}, which show that the extension $H_{\mathcal{O}} \subseteq H_{\mathcal{O}}(\sqrt{\alpha_\pgotic})$ is ramified at all primes of $H_{\mathcal{O}}$ lying above $\pgotic$ and unramified at all primes of $H_{\mathcal{O}}$ which do not divide $\pgotic \cdot \mathfrak{f}_{E_\pgotic} \cdot \mathcal{O}_{H_\Ogotic}$, because $H_{\mathcal{O}}(\sqrt{\alpha_\pgotic}) \subseteq H_{\mathcal{O}}(E_0[\mathfrak{p}])$.
    
    Suppose now that $K = \mathbb{Q}(i)$. 
    We show first of all how to obtain from $E_0$ an elliptic curve ${E_1}_{/H_{\mathcal{O}}}$ such that $H_\mathcal{O}((E_1)_\text{tors}) = \mathbb{Q}(i)^\text{ab}$.
    If there exists an integer $N \in \mathbb{N}$ such that $N > 2$ and $H_\mathcal{O}(E_0[N]) = H_{N,\mathcal{O}}$, then \cref{prop: converse Coates-Wiles} shows that we can take $E_1 = E_0$. Suppose on the contrary that $H_\mathcal{O}(E_0[N]) \neq H_{N,\mathcal{O}}$ for every $N \in \mathbb{Z}_{\geq 3}$, which implies by \cref{lem:free_O-module} and \cref{prop:hilbert_12} that 
    \[
        G_N := \operatorname{Gal}(H_\mathcal{O}(E_0[N])/H_\mathcal{O}) \cong (\mathcal{O}/N \mathcal{O})^\times
    \]
    for every $N \in \mathbb{Z}_{\geq 3}$.
    Then we distinguish two cases:
    \begin{itemize}
        \item if the conductor $\mathfrak{f}_\mathcal{O} := \lvert \mathbb{Z}[i] \colon \mathcal{O} \rvert$ is even, the isomorphism 
    \[
        \frac{\mathcal{O}}{4 \mathcal{O}} \cong \frac{\mathbb{Z}[x]}{(x^2+\mathfrak{f}_\mathcal{O}^2,4)} \cong \frac{(\mathbb{Z}/4\mathbb{Z})[x]}{(x^2)} 
    \]
    holds. Hence the group $G_4$ contains a subgroup $Q \subseteq G_4$ of index two, corresponding via the following isomorphism
    \[
        G_4 \cong \left(\frac{\mathcal{O}}{4 \mathcal{O}}\right)^\times \cong \left( \frac{(\mathbb{Z}/4 \mathbb{Z})[x]}{(x^2)} \right)^\times \cong \left\{ \begin{pmatrix} a & b \\ 0 & a \end{pmatrix} \ \middle| \ \begin{aligned} a &\in (\mathbb{Z}/4 \mathbb{Z})^\times \\ b &\in \mathbb{Z}/4 \mathbb{Z} \end{aligned} \right\} \subseteq \operatorname{GL}_2(\mathbb{Z}/4 \mathbb{Z})
    \]
    to the group of matrices of the form $\begin{psmallmatrix} 1 & b \\ 0 & 1 \end{psmallmatrix}$ with $b \in \mathbb{Z}/4 \mathbb{Z}$.
    Therefore the sub-extension of $H_\mathcal{O} \subseteq H_\mathcal{O}(E_0[4])$ fixed by $Q$ is given by $H_\mathcal{O}(\sqrt{\alpha})$ for some $\alpha \in H_\mathcal{O}$.
    Moreover $H_\mathcal{O}(\sqrt{\alpha}) \cap H_{4,\mathcal{O}} = H_\mathcal{O}$, because $Q$ does not contain the subgroup $\operatorname{Gal}(H_\mathcal{O}(E_0[4])/H_{4,\mathcal{O}})$, since the latter corresponds via the previous isomorphism to the group of matrices $\left\{ \pm \begin{psmallmatrix} 1 & 0 \\ 0 & 1 \end{psmallmatrix} \right\}$.
    Hence $H_\mathcal{O}(E_0[4]) = H_{4,\mathcal{O}}(\sqrt{\alpha})$, and \cref{prop:proposition_twist} shows that the twisted elliptic curve $E_1 := E_0^{(\alpha)}$ has the property that $H_\mathcal{O}(E_1[4]) = H_{4,\mathcal{O}}$.
    Therefore, \cref{prop: converse Coates-Wiles} shows that $H_\mathcal{O}((E_1)_\text{tors}) = \mathbb{Q}(i)^\text{ab}$;
    \item if $\mathfrak{f}_\mathcal{O}$ is odd, our assumptions on $\mathcal{O}$ imply that there exists a prime $p \mid \mathfrak{f}_\mathcal{O}$ such that $p \equiv 3 \text{ mod } 4$.
    Then the group
    \[
        G_p \cong \left(\frac{\mathcal{O}}{p \mathcal{O}}\right)^\times \cong \left( \frac{\mathbb{F}_p[x]}{(x^2)} \right)^\times \cong \left\{ \begin{pmatrix} a & b \\ 0 & a \end{pmatrix} \ \middle| \ 
        \begin{aligned}
            a &\in \mathbb{F}_p^\times \\
            b &\in \mathbb{F}_p
        \end{aligned} \right\} \subseteq \operatorname{GL}_2(\mathbb{F}_p)
    \]
    contains a subgroup of index two, corresponding to the group of matrices of the form $\begin{psmallmatrix} a^2 & b \\ 0 & a^2 \end{psmallmatrix}$ with $a \in \mathbb{F}_p^\times$ and $b \in \mathbb{F}_p$. The sub-extension of $H_\mathcal{O} \subseteq H_\mathcal{O}(E_0[p])$ fixed by this subgroup is given by $H_\mathcal{O}(\sqrt{\alpha})$ for some $\alpha \in H_\mathcal{O}$. 
    Moreover $H_\mathcal{O}(\sqrt{\alpha}) \cap H_{p,\mathcal{O}} = H_\mathcal{O}$, since the degree $[H_{p,\mathcal{O}} \colon H_\mathcal{O}] = p (p - 1)/2$ is odd. 
    Hence $H_\mathcal{O}(E_0[p]) = H_{p,\mathcal{O}}(\sqrt{\alpha})$, and again \cref{prop:proposition_twist} shows that the twisted elliptic curve $E_1 := E_0^{(\alpha)}$ has the property that $H_\mathcal{O}(E_1[p]) = H_{p,\mathcal{O}}$.
    Therefore, \cref{prop: converse Coates-Wiles} shows that $H_\mathcal{O}((E_1)_\text{tors}) = \mathbb{Q}(i)^\text{ab}$.
    \end{itemize}
    
    Finally, we construct, by suitably twisting $E_1$, infinitely many elliptic curves $E_{/H_\mathcal{O}}$ which are pairwise non-isomorphic over $H_\mathcal{O}$ and share the property that $H_\mathcal{O}(E_\text{tors}) = \mathbb{Q}(i)^\text{ab}$.
    To do this, fix an integer $m \in \mathbb{Z}_{\geq 3}$ such that $m \mid \Delta_\mathcal{O}$ and $H_\mathcal{O}(E_1[m]) = H_{m,\mathcal{O}}$, which exists by the previous discussion. 
    Now, observe that for every prime ideal $\mathfrak{p} \subseteq \mathcal{O}$ which is coprime with $\operatorname{N}_{H_\mathcal{O}/\mathbb{Q}}(\mathfrak{f}_{E_1}) \cdot \Delta_\mathcal{O}$, the Galois group
    \[
        \operatorname{Gal}(H_{\mathfrak{p},\mathcal{O}}/H_\mathcal{O}) \cong \frac{(\mathcal{O}/\mathfrak{p})^\times}{\mathcal{O}^\times} \cong \frac{(\mathbb{Z}[i]/\mathfrak{p} \mathbb{Z}[i])^\times}{\{\pm 1 \}}
    \]
    is cyclic, and its order is even.
    Thus the extension $H_\mathcal{O} \subseteq H_{\mathfrak{p},\mathcal{O}}$ contains a unique quadratic sub-extension, of the form $H_\mathcal{O}(\sqrt{\alpha_\mathfrak{p}})$ for some $\alpha_\mathfrak{p} \in H_\mathcal{O}$.
    Since $\mathfrak{p}$ is invertible in $\mathcal{O}$, \cref{prop:hilbert_12} shows that $H_{\mathfrak{p},\mathcal{O}} \subseteq H_\mathcal{O}(E[\mathfrak{p}])$, and \cref{prop:proposition_twist} shows that the twisted elliptic curve $E_\mathfrak{p} := E_1^{(\alpha_\mathfrak{p})}$ has the property that $H_\mathcal{O}(E_\mathfrak{p}[m]) \cap H_\mathcal{O}(E_\mathfrak{p}[\mathfrak{p}]) = H_\mathcal{O}(\sqrt{\alpha_\mathfrak{p}})$.
    Thus $H_\mathcal{O}(E_\mathfrak{p}[m \mathfrak{p}]) = H_{m \mathfrak{p}, \mathcal{O}}$, and \cref{prop: converse Coates-Wiles} shows that $H_\mathcal{O}((E_\mathfrak{p})_\text{tors}) = \mathbb{Q}(i)^\text{ab}$.
    To conclude our proof of \labelcref{it:shimura_yes_A}, we observe that the elliptic curves $E_\mathfrak{p}$ are pairwise non-isomorphic over $H_\mathcal{O}$, by the same argument used in the case $K \neq \mathbb{Q}(i)$.
    
    We now prove \labelcref{it:shimura_yes_B}.
    Fix a non-maximal order $\mathcal{O} \subseteq \mathbb{Z}[i]$ whose conductor $\mathfrak{f}_\mathcal{O} \in \mathbb{Z}_{\geq 2}$ is divided only by primes $p \equiv 1 \ \text{mod} \ 4$. 
    Then $\widehat{\mathcal{O}} = \prod_p (\mathcal{O} \otimes_\mathbb{Z} \mathbb{Z}_p) \cong \widehat{\mathbb{Z}}[i]$, because for each prime $p \nmid \mathfrak{f}_\mathcal{O}$ one evidently has that $\mathcal{O} \otimes_\mathbb{Z} \mathbb{Z}_p \cong \mathbb{Z}_p[i]$, and for each prime $p \mid \mathfrak{f}_\mathcal{O}$, since $p \equiv 1  \ \text{mod} \ 4$ by our assumptions, one has that $\mathbb{Z}[i] \subseteq \mathbb{Z}_p$, which shows that $\mathcal{O} \otimes_\mathbb{Z} \mathbb{Z}_p \cong \mathbb{Z}_p[i]$ also in this case.
    In particular, for every $N \in \mathbb{N}$ we have that $-1 \in (\mathcal{O}/N \mathcal{O})^\times$ is a square.
    
    Suppose now by contradiction that there exists an elliptic curve $E_{/H_\mathcal{O}}$ such that $H_\mathcal{O}(E_\text{tors}) = \mathbb{Q}(i)^\text{ab}$. 
    Then \cref{thm:coates_wiles} shows that $H_\mathcal{O}(E[N]) = H_{N,\mathcal{O}}$ for some integer $N \in \mathbb{Z}_{\geq 3}$.
    Using the Galois representation $\rho_{E,N}$ defined in \cref{lem:free_O-module}, one gets an embedding
    \[
        \iota \colon \frac{(\mathcal{O}/N \mathcal{O})^\times}{\mathcal{O}^\times} \overset{\left( \dagger \right)}{\cong} \operatorname{Gal}(H_{N,\mathcal{O}}/H_\mathcal{O}) = \operatorname{Gal}(H_\mathcal{O}(E[N])/H_\mathcal{O}) \hookrightarrow (\mathcal{O}/N \mathcal{O})^\times
    \]
    where $\left( \dagger \right)$ is the isomorphism given by \cref{thm:galois_ray_class_fields}.
    Hence $\iota \colon (\mathcal{O}/N \mathcal{O})^\times/\mathcal{O}^\times \hookrightarrow (\mathcal{O}/N \mathcal{O})^\times$ is a section of the quotient map $(\mathcal{O}/N \mathcal{O})^\times \twoheadrightarrow (\mathcal{O}/N \mathcal{O})^\times/\mathcal{O}^\times$, and the short exact sequence
    \[
        1 \to \mathcal{O}^\times \to (\mathcal{O}/N \mathcal{O})^\times \to (\mathcal{O}/N \mathcal{O})^\times/\mathcal{O}^\times \to 1
    \]
    splits. Thus there exists a map $h \colon (\mathcal{O}/N \mathcal{O})^\times \twoheadrightarrow \mathcal{O}^\times$ which is a retraction of the inclusion $\mathcal{O}^\times \hookrightarrow (\mathcal{O}/N \mathcal{O})^\times$. In particular, one has that $h(-1) = -1$, which yields a contradiction because $-1 \in \mathcal{O}^\times = \{\pm 1\}$ is not a square. This concludes the proof of \labelcref{it:shimura_yes_B}.
\end{proof}

We now prove, detailing upon and generalising a remark of Shimura (see \cite[Pages~217-218]{Shimura_1994}), that, under the assumption $\operatorname{Pic}(\mathcal{O}) \neq \{1\}$, not all CM elliptic curves $E_{/H_{\mathcal{O}}}$ satisfy $H_{\mathcal{O}}(E_{\text{tors}})=K^{\text{ab}}$.

\begin{theorem}
\label{thm:Shimura_condition_no}
    Let $\mathcal{O}$ be an order in an imaginary quadratic field $K$ such that $\operatorname{Pic}(\mathcal{O}) \neq \{1\}$, and fix $j \in H_{\mathcal{O}}$ to be the $j$-invariant of any elliptic curve with complex multiplication by $\mathcal{O}$. Then there exist infinitely many elliptic curves $E_{/H_{\mathcal{O}}}$ with $j(E)=j$ but non-isomorphic over $H_\Ogotic$, and such that $H_{\mathcal{O}}(E_{\text{tors}}) \neq K^{\text{ab}}$.
\end{theorem}
\begin{proof}
    If $\mathcal{O}$ is an order of discriminant $\Delta_\mathcal{O}=-4f_\mathcal{O}^2$ whose conductor $f_\mathcal{O} \in \mathbb{Z}_{\geq 2}$ is only divided by primes $p \equiv 1 \text{ mod } 4$, then by \cref{thm:Shimura_condition_yes}\labelcref{it:shimura_yes_B} all elliptic curves $E_{/H_\mathcal{O}}$ with complex multiplication by $\mathcal{O}$ satisfy $H_{\mathcal{O}}(E_{\text{tors}}) \neq K^{\text{ab}}$. Hence the statement is trivially true in this case.
    
    Suppose now that $\mathcal{O}$ is not as above. Fix an elliptic curve $E_0$ defined over $H_\mathcal{O}$ such that $j(E_0) = j$ and $H_\mathcal{O}((E_0)_\text{tors}) = K^\text{ab}$. We know that infinitely many such elliptic curves $E_0$ exist by \cref{thm:Shimura_condition_yes}.
    We observe now that for every $\alpha \in H_\mathcal{O}^\times$ such that the extension $K \subseteq H_\mathcal{O}(\sqrt{\alpha})$ is not abelian,
    we have that 
    \[
    H_\mathcal{O}((E_0^{(\alpha)})_\text{tors}) \neq K^\text{ab}
    \] where $E_0^{(\alpha)}$ denotes the quadratic twist of $E_0$ by $\alpha \in H_\mathcal{O}^\times$.
    Indeed, \cref{thm:coates_wiles} shows that $H_\mathcal{O}(E_0[N]) = H_{N,\mathcal{O}}$ for some $N \in \mathbb{N}$, and this combined with \cref{prop:proposition_twist}, implies that $H_\mathcal{O}(E_0^{(\alpha)}[N]) = H_{N,\mathcal{O}}(\sqrt{\alpha}) \not\subseteq K^\text{ab}$.
    
    In order to conclude the proof it is thus sufficient to show that there exist infinitely many $\alpha \in H_\mathcal{O}^\times$ such that $\sqrt{\alpha} \not\in K^\text{ab}$ and the elliptic curves $E_0^{(\alpha)}$ are pairwise not isomorphic over $H_\mathcal{O}$.
    This is equivalent to say that there exist infinitely many distinct quadratic extensions of $H_\mathcal{O}$ which are not abelian over $K$. This can be shown, for instance, as follows. 
    
    Since $\operatorname{Pic}(\mathcal{O}) \neq \{1\}$ we have that $K \neq H_\mathcal{O}$. 
    Hence one can show, using for example Chebotarëv's density theorem \cite[Chapter~VII, Theorem~13.4]{Neukirch_1999}, that there exists an infinite set of prime ideals $\Lambda_0=\{\mathfrak{p}_j \subseteq \mathcal{O}_K \}_{j \in \mathbb{N}}$ such that for every index $j \in \mathbb{N}$ we have that $2 \not\in \mathfrak{p}_j$ and the ideal $\mathfrak{p}_j \cdot \mathcal{O}_{H_\mathcal{O}}$ is divisible by at least two distinct primes $\mathfrak{P}_{1,j}, \mathfrak{P}_{2,j} \subseteq \mathcal{O}_{H_\mathcal{O}}$. Fix now an index $j_0 \in \mathbb{N}$ (\textit{e.g.} $j_0 = 0$), and take any element $\alpha_0 \in \mathfrak{P}_{1,j_0} \setminus (\mathfrak{P}_{1,j_0}^2 \cup \mathfrak{P}_{2,j_0})$. 
    Now, elementary ramification theory of quadratic extensions (see for instance \cite[Chapter~I, Theorem~6.3]{Gras_2003}) shows that the extension $H_\mathcal{O} \subseteq H_\mathcal{O}(\sqrt{\alpha_0})$ ramifies at $\mathfrak{P}_{1,j_0}$ but not at $\mathfrak{P}_{2,j_0}$. This implies that the extension $K \subseteq H_\mathcal{O}(\sqrt{\alpha_0})$ is not Galois, hence in particular not abelian.
    Now, let $\Gamma_0$ be the finite set of prime ideals of $\mathcal{O}_K$ dividing $\operatorname{N}_{H_\mathcal{O}/K}(\alpha_0)$ and put $\Lambda_1 := \Lambda_0 \setminus \Gamma_0$, which is still an infinite set. 
    Fix an index $j_1 \in \mathbb{N}$ such that $\mathfrak{p}_{j_1} \in \Lambda_1$ and take any element $\alpha_1 \in \mathfrak{P}_{1,j_1} \setminus (\mathfrak{P}_{1,j_1}^2 \cup \mathfrak{P}_{2,j_1})$. 
    Again $K \subseteq H_\mathcal{O}(\sqrt{\alpha_1})$ is a non-abelian extension. 
    Moreover we have that $H_\mathcal{O}(\sqrt{\alpha_0}) \neq H_\mathcal{O}(\sqrt{\alpha_1})$ since the prime $\mathfrak{P}_{1,j_1}$ ramifies in the extension $H_\mathcal{O} \subseteq H_\mathcal{O}(\sqrt{\alpha_1})$, but the same prime does not ramify in $H_\mathcal{O} \subseteq H_\mathcal{O}(\sqrt{\alpha_0})$. 
    Repeating this process, we construct an infinite set of pairwise distinct quadratic extensions $\{H_\mathcal{O} \subseteq H_\mathcal{O}(\sqrt{\alpha_j}): j \in \mathbb{N} \}$ that are non-abelian over $K$. 
    This concludes the proof.
\end{proof}

We conclude this section by proving \cref{thm:infinitely_many_linearly_disjoint}.

\begin{proof}[Proof of \cref{thm:infinitely_many_linearly_disjoint}]
    By \cref{thm:Shimura_condition_no} there exist infinitely many pairwise non-isomorphic elliptic curves $E_{/H_\mathcal{O}}$ such that $j(E) = j$ and $H_\mathcal{O}(E_\text{tors}) \neq K^\text{ab}$.
    Applying \cref{cor:index_of_Galois} to any such elliptic curve allows us to conclude.
\end{proof}
\begin{remark} \label{rmk:bourdon_clark_1-8}
    Fix an imaginary quadratic order $\mathcal{O}$.
    Then \cite[Corollary~1.8]{Bourdon_Clark_2020} shows that, for every $N \in \mathbb{Z}_{\geq 1}$, there exists an elliptic curve $E_{/H_\mathcal{O}}$
    such that the Galois representation $\rho_{E,N}$ introduced in \cref{lem:free_O-module} is surjective.
    \cref{thm:infinitely_many_linearly_disjoint} strengthens this result in the case $\operatorname{Pic}(\mathcal{O}) \neq \{1\}$, by showing that there exist infinitely many pairwise non-isomorphic elliptic curves $E_{/H_\mathcal{O}}$ such that $\rho_{E,N}$ is surjective for every $N \in \mathbb{Z}_{\geq 1}$.
\end{remark}

\section{Entanglement in the family of division fields of CM elliptic curves over \texorpdfstring{$\mathbb{Q}$}{Q}}
\label{sec:entanglement}

Let $E_{/\Q}$ be an elliptic curve with complex multiplication by an order in an imaginary quadratic field $K$. 
The aim of this section is to explicitly determine the image of the natural map
\begin{equation} \label{entanglement}
    \Gal(K(E_{\text{tors}})/K)\hookrightarrow \prod_{q} \Gal(K(E[q^{\infty}])/K)
\end{equation}
where the product runs over all rational primes $q \in \mathbb{N}$ and $K(E[q^{\infty}])$ denotes the compositum of the $q$-power division fields of $E_{/K}$.
In other words, we want to analyse the entanglement in the family of Galois extensions $\{ K(E[q^{\infty}]) \}_q$ over $K$.
The conclusion of this study will be \cref{thm:classification_entanglement}, which provides a complete description of the image of \eqref{entanglement} for all CM elliptic curves $E_{/\mathbb{Q}}$ such that $j(E) \not\in \{0,1728\}$.
Observe that there is essentially no difference in considering the division fields of the elliptic curve $E_{/\mathbb{Q}}$ and of its base change $E_{/K}$, because $\mathbb{Q}(E[n]) = K(E[n])$ for every $n > 2$ as explained in \cref{rmk:division_fields_base_change}.
In particular, the family of division fields $\{\mathbb{Q}(E[q^{\infty}]) \}_{q}$ is always entangled over $\mathbb{Q}$, but there are elliptic curves for which it is linearly disjoint over $K$, as we will see in \cref{thm:classification_entanglement}.

We briefly outline the strategy of our proof: since $E$ is defined over $\mathbb{Q}$ we have that 
$\lvert \operatorname{Pic}(\mathcal{O}) \rvert = [\mathbb{Q}(j(E)) \colon \mathbb{Q}] = 1$ (see \cite[Proposition~13.2]{Cox_2013}) which implies that the elliptic curve $E$ has complex multiplication by one of the thirteen imaginary quadratic orders $\mathcal{O}$ of class number $1$, listed in \cite[Theorem~7.30]{Cox_2013}.
For each of these orders $\mathcal{O}$, we first find an elliptic curve ${E_0}_{/\Q}$ with complex multiplication by $\mathcal{O}$ such that $\lvert \mathfrak{f}_{E_0} \rvert \in \mathbb{N}$ is minimal among all the conductors\footnote{The symbol $\lvert \mathfrak{f}_{A} \rvert \in \mathbb{N}$ denotes the positive generator of the conductor ideal $\mathfrak{f}_{A} \subseteq \mathbb{Z}$ of an elliptic curve $A_{/\Q}$} of elliptic curves defined over $\Q$ which have complex multiplication by $\mathcal{O}$.
We then proceed to compute the full entanglement in the family of division fields of ${E_0}_{/K}$, using \cref{thm:disjointness_intro}, \cref{thm:coates_wiles}, and \cref{prop:Deuring_formula}. 
Since $\mathcal{O}$ is an order of class number $1$ and $j(E) \not\in \{0,1728\}$, we have that $E$ is a quadratic twist of $E_0$. 
We then use \cref{prop:proposition_twist}, which describes how Galois representations attached to CM elliptic curves behave under quadratic twisting, to determine the complete entanglement in the family of division fields of $E_{/K}$. 

We begin by deriving some consequences of \cref{prop:proposition_twist} when $\operatorname{Pic}(\mathcal{O}) = 1$ and the elliptic curve $E_{/K}$ is the base change to the imaginary quadratic field $K = H_{\mathcal{O}}$ of an elliptic curve defined over $\mathbb{Q}$.
To do this, we need a formula originally due to Deuring that relates the conductor of a CM elliptic curve defined over $\mathbb{Q}$ to the conductor of the unique Hecke character $\varphi \colon \mathbb{A}_{K}^{\times} \to \mathbb{C}^{\times}$ associated to its base change over $K$ by \cref{thm:main_theorem_of_CM}.
\begin{proposition}[Deuring] \label{prop:Deuring_formula}
    Let $\mathcal{O} \subseteq K$ be an order inside an imaginary quadratic field $K$. Let $E$ be an elliptic curve defined over $\mathbb{Q}(j(E))$ with complex multiplication by $\mathcal{O}$. Denote by $\varphi \colon \mathbb{A}_{H_{\mathcal{O}}}^{\times} \to \mathbb{C}^{\times}$ the unique Hecke character associated by \cref{thm:main_theorem_of_CM} to the base change of $E$ over $K(j(E)) = H_{\mathcal{O}}$.
    Then, letting $j = j(E)$, one can write the conductor $\mathfrak{f}_E \subseteq \mathcal{O}_{\mathbb{Q}(j)}$ of $E$ as
    \[ 
        \mathfrak{f}_E = \operatorname{N}_{K(j)/\mathbb{Q}(j)}(\mathfrak{f}_{\varphi}) \cdot \delta_{K(j)/\mathbb{Q}(j)}
    \]
    where $\operatorname{N}_{K(j)/\mathbb{Q}(j)}(\mathfrak{f}_{\varphi}) \subseteq \mathcal{O}_{\mathbb{Q}(j)}$ denotes the relative norm of the conductor $\mathfrak{f}_{\varphi} \subseteq \mathcal{O}_{K(j)}$ of the Hecke character $\varphi$ and $\delta_{K(j)/\mathbb{Q}(j)} \subseteq \mathcal{O}_{\mathbb{Q}(j)}$ denotes the relative discriminant ideal associated to the quadratic extension $\mathbb{Q}(j) \subseteq K(j)$.
\end{proposition}
\begin{proof}
    A modern proof of this formula can be obtained using \cite[Theorem~3]{Milne_1972} and \cite[Theorem~12]{Serre_Tate_1968}. This is detailed in \cite[Appendix~A]{Pengo}.
\end{proof}

We go back to study the consequences of \cref{prop:proposition_twist}. 
Let $E_{/K}$ be the base change to an imaginary quadratic field $K = H_{\mathcal{O}}$ of an elliptic curve $E_{/\mathbb{Q}}$ of conductor $\mathfrak{f}_E \subseteq \mathbb{Z}$ and with complex multiplication by an order $\mathcal{O}$ of class number one and discriminant $\Delta_{\mathcal{O}} < -4$. 
Fix also $\alpha \in \mathbb{Q}^{\times}$.
Under these assumptions we may assume that $\alpha = \Delta$ where $\Delta = \Delta_F \in \mathbb{Z}$ is the fundamental discriminant associated to some quadratic extension $\mathbb{Q} \subseteq F$. Since $E^{(\alpha \beta)} = (E^{(\alpha)})^{(\beta)}$ for any $\alpha, \beta \in \mathbb{Q}^{\times}$, we reduce the study of the Galois representation $\rho_{E^{(\Delta)},p^n}$ for any prime $p \in \mathbb{Z}_{\geq 1}$ and any $n \in \mathbb{N}$ to the following cases:
\begin{enumerate}[label*=\protect\fbox{T.\arabic{enumi}}]
        \item $\Delta = (-1)^{(q-1)/2} \, q$ for some prime $q \in \mathbb{Z}_{\geq 3}$ with $q \nmid p \, \mathfrak{f}_E$.
        In this case $K(\sqrt{\Delta}) \cap K(E[p^n]) = K$. Indeed any prime $\mathfrak{q} \subseteq \mathcal{O}_K$ such that $\mathfrak{q} \mid q \mathcal{O}_K$ does not ramify in $K \subseteq K(E[p^n])$, as follows from \cref{prop:unramified} because $q \nmid p \, \mathfrak{f}_E$. On the other hand, any prime $\mathfrak{q} \mid q \mathcal{O}_K$ ramifies in $K \subseteq K(\sqrt{\Delta})$ since \cref{prop:Deuring_formula} shows that $q \nmid \Delta_K$, where $\Delta_K \in \mathbb{Z}_{< 0}$ denotes the absolute discriminant of the imaginary quadratic field $K$. 
        Thus \cref{prop:proposition_twist} implies that $\rho_{E^{(\Delta)},p^n}$ will have maximal image independently from the behaviour of $\rho_{E,p^n}$; \label{it:twist_q_not_p}
        \item $p \geq 3$ and $\Delta = (-1)^{(p-1)/2} \, p$. In this case class field theory shows that
        \[ \mathbb{Q}(\sqrt{\Delta}) \subseteq \mathbb{Q}(\mu_p) \subseteq H_{p^n,\mathcal{O}} \]
        where for every $m \in \mathbb{N}$ we let $\mu_m \subseteq \overline{\mathbb{Q}}$ denote the group of $m$-th roots of unity.
        Hence \cref{prop:proposition_twist} implies that $\rho_{E^{(\Delta)},p^n}$ has maximal image if and only if $\rho_{E,p^n}$ does; \label{it:twist_p}
        \item $\Delta \in \{-4,-8,8\}$ and $2 \nmid p \, \mathfrak{f}_E$. In this case $K(\sqrt{\Delta}) \cap K(E[p^n]) = K$, as in \labelcref{it:twist_q_not_p}, hence \cref{prop:proposition_twist} shows that $\rho_{E^{(\Delta)},p^n}$ will have maximal image independently from the behaviour of $\rho_{E,p^n}$;
        \label{it:twist_2_not_p}
        \item $\Delta \in \{-4,-8,8\}$ and $p = 2$. In this case $\mathbb{Q}(\sqrt{\Delta}) \subseteq \mathbb{Q}(\mu_{\lvert \Delta \rvert}) \subseteq H_{\lvert \Delta \rvert,\mathcal{O}}$ by class field theory. Hence \cref{prop:proposition_twist} implies that for every $n \in \N$ such that $2^n \geq \lvert \Delta \rvert$ the representation $\rho_{E^{(\Delta)},2^n}$ has maximal image if and only if $\rho_{E,2^n}$ does, similarly to what we proved in \labelcref{it:twist_p}.
        \label{it:twist_2}
\end{enumerate}
\begin{remark}
    The previous discussion shows in particular that, under suitable hypotheses on $\Delta$, if the Galois representation $\rho_{E,p^n}$ is surjective then $\rho_{E^{(\Delta)},p^n}$ is surjective.
    This might not be the case if these assumptions on $\Delta$ are not satisfied, as it follows from
    \cref{thm:classification_entanglement}.
\end{remark}

We are now ready to study the entanglement of division fields of CM elliptic curves $E$ defined over $\mathbb{Q}$ such that $j(E) \not\in\{0,1728\}$. 

First of all, assume that $E$ has complex multiplication by an order $\mathcal{O}$ with $\operatorname{gcd}(\Delta_{\mathcal{O}},6) = 1$.
Here $\Delta_{\mathcal{O}} := \mathfrak{f}_{\mathcal{O}}^2 \, \Delta_K$ denotes the discriminant of $\mathcal{O}$, where $\Delta_K \in \mathbb{Z}$ denotes the absolute discriminant of $K$ and $\mathfrak{f}_{\mathcal{O}} := [\mathcal{O}_K \colon \mathcal{O}]$ denotes the conductor of $\mathcal{O}$.
Since $\operatorname{Pic}(\mathcal{O}) = \{ 1 \}$ we have that $\mathcal{O} = \mathcal{O}_K$ and $\Delta_{\mathcal{O}} = \Delta_K = - p$ where $p \in \mathbb{N}$ is a prime number such that $p \geq 7$ and $p \equiv 3 \ \text{mod} \ 4$ (see \cite[Theorem~7.30]{Cox_2013}). Moreover $E = E_0^{(\Delta)}$ for some fundamental discriminant $\Delta \in \mathbb{Z}$, where $E_0$ is one of the two elliptic curves with $j(E_0) = j(E)$ appearing in \cref{tab:twist_minimal_CM_elliptic_curves}, which lists the CM elliptic curves defined over $\mathbb{Q}$ whose conductor $\lvert \mathfrak{f}_E \rvert \in \mathbb{N}$ is minimal among their twists.
\begin{table}[t]
    \centering
    \begin{tabular}{cccccc}
    \toprule
     $\Delta_K$ & $\mathfrak{f}_{\mathcal{O}}$ & $j(E)$ & $\lvert \mathfrak{f}_E \rvert$ & Equations  \\
     \midrule
        $-3$  & $1$ & $0$ & $3^3$ & $\begin{aligned} y^2 + y &= x^3 - 7 \\[-7pt] y^2 + y &= x^3 \end{aligned}$ \\\cmidrule{2-5}
            & $2$ & $2^4 \, 3^3 \, 5^3$ & $2^2 3^2$ & $\begin{aligned} y^2 &= x^{3} - 15 x + 22 \\[-7pt] y^2 &= x^{3} - 135 x - 594  \end{aligned}$ \\\cmidrule{2-5}
            & $3$ & $-2^{15} \, 3 \, 5^3$ & $3^3$ & $\begin{aligned} y^2 + y &= x^{3} - 30 x + 63 \\[-7pt] y^2 + y &= x^{3} - 270 x - 1708  \end{aligned}$ \\
    \midrule
        -4    & $1$ & $2^6 \, 3^3$ & $2^5$ & $\begin{aligned} y^2 &= x^{3} - x \\[-7pt] y^2 &= x^{3} + 4 x  \end{aligned}$ \\\cmidrule{2-5}
           & $2$ & $2^3 \, 3^3 \, 11^3$ & $2^5$ & $\begin{aligned} y^2 &= x^{3} - 11 x - 14 \\[-7pt] y^2 &= x^{3} - 11 x + 14  \end{aligned}$ \\
    \midrule
        -7    & $1$ & $-3^3 \, 5^3$ & $7^2$ & $\begin{aligned} y^2 + x y &= x^{3} -  x^{2} - 2 x - 1 \\[-7pt] y^2 + x y &= x^{3} -  x^{2} - 107 x + 552  \end{aligned}$ \\\cmidrule{2-5}
           & $2$ & $3^3 \, 5^3 \, 17^3$ & $7^2$ & $\begin{aligned} y^2 + x y &= x^{3} -  x^{2} - 37 x - 78 \\[-7pt] y^2 + x y &= x^{3} -  x^{2} - 1822 x + 30393  \end{aligned}$ \\
    \midrule
        -8    & $1$ & $2^6 \, 5^3$ & $2^8$ & $\begin{aligned} y^2 &= x^{3} -  x^{2} - 3 x - 1 \\[-7pt] y^2 &= x^{3} + x^{2} - 3 x + 1 \\[-7pt] y^2 &= x^{3} -  x^{2} - 13 x + 21 \\[-7pt] y^2 &= x^{3} + x^{2} - 13 x - 21  \end{aligned}$ \\
        \midrule
        -11    & $1$ & $-2^{15}$ & $11^2$ & $\begin{aligned} y^2 + y &= x^{3} -  x^{2} - 7 x + 10 \\[-7pt] y^2 + y &= x^{3} -  x^{2} - 887 x - 10143  \end{aligned}$ \\
        \midrule
        -19    & $1$ & $-2^{15} \, 3^3$ & $19^2$ & $\begin{aligned} y^2 + y &= x^{3} - 38 x + 90 \\[-7pt] y^2 + y &= x^{3} - 13718 x - 619025  \end{aligned}$ \\
        \midrule
        -43    & $1$ & $-2^{18} \, 3^3 \, 5^3$ & $43^2$ & $\begin{aligned} y^2 + y &= x^{3} - 860 x + 9707 \\[-7pt] y^2 + y &= x^{3} - 1590140 x - 771794326  \end{aligned}$ \\
        \midrule
        -67    & $1$ & $-2^{15} \, 3^3 \, 5^3 \, 11^3$ & $67^2$ & $\begin{aligned} y^2 + y &= x^{3} - 7370 x + 243528 \\[-7pt] y^2 + y &= x^{3} - 33083930 x - 73244287055  \end{aligned}$ \\
        \midrule
        -163    & $1$ & $-2^{18} \, 3^3 \, 5^3 \, 23^3 \, 29^3$ & $163^2$ & $\begin{aligned} y^2 + y &= x^{3} - 2174420 x + 1234136692 \\[-7pt] y^2 + y &= x^{3} - 57772164980 x - 5344733777551611  \end{aligned}$ \\
    \bottomrule \\
    \end{tabular}
    \caption{Minimal Weierstrass equations of CM elliptic curves defined over $\mathbb{Q}$ having the smallest conductor $\lvert \mathfrak{f}_E \rvert$ amongst all their twists, where $\lvert \mathfrak{f}_E \rvert \in \mathbb{N}$ denotes the unique positive generator of the conductor ideal $\mathfrak{f}_E \subseteq \mathbb{Z}$.}
    \label{tab:twist_minimal_CM_elliptic_curves}
\end{table}

Let us study the division fields of $E_0$, as a first step towards the analysis of the division fields of $E$.
\cref{thm:disjointness_intro} provides a decomposition
\begin{equation} \label{eq:division_fields_E_0}
    \Gal(K((E_0)_{\text{tors}})/K) \cong \prod_{q} \Gal(K(E_0[q^{\infty}])/K)
\end{equation}
where the product runs over all the rational primes $q \in \mathbb{N}$.
Indeed in this case the set $S_{E_0}$ appearing in \cref{thm:disjointness_intro} consists of the single prime $p$ because $\lvert \mathfrak{f}_{E_0} \rvert = p^2$ as follows from an inspection of \cref{tab:twist_minimal_CM_elliptic_curves}.
The isomorphism \eqref{eq:division_fields_E_0} shows that the family of division fields $\{ K(E_0[q^{\infty}]) \}_{q}$ is linearly disjoint over $K$, where $q \in \mathbb{N}$ runs over all the rational primes.
\cref{prop:ramification} implies also that $\Gal(K(E_0[q^m])/K) \cong (\mathcal{O}/q^m \mathcal{O})^{\times}$ for every prime $q \neq p$ and every $m \in \mathbb{N}$. 
On the other hand we have that $\Gal(K(E_0[p^m])/K) \cong (\mathcal{O}/p^m \mathcal{O})^{\times}/\{\pm 1\}$ for every $m \in \mathbb{N}$.
Indeed, it follows from \cref{prop:Deuring_formula} that $\mathfrak{f}_{\varphi_0} = \mathfrak{p}$, where $\mathfrak{p} \subseteq \mathcal{O}$ is the unique prime lying above $p$ and $\varphi_0 \colon \mathbb{A}_K^{\times} \to \mathbb{C}^{\times}$ is the unique Hecke character associated to $E_0$ by \cref{thm:main_theorem_of_CM}. 
Hence \cref{thm:coates_wiles} shows that $K(E_0[p^m]) = H_{p^m,\mathcal{O}}$ for every $m \in \mathbb{N}$, where $H_{p^m,\mathcal{O}}$ is the ray class field of $K$ modulo $p^m$ because $\mathcal{O} = \mathcal{O}_K$. 
Hence we can conclude that $\Gal(K(E_0[p^m])/K) \cong (\mathcal{O}/p^m \mathcal{O})^{\times}/\{\pm 1\}$ using \cref{thm:galois_ray_class_fields}.

Let us now go back to the division fields of $E = E_0^{(\Delta)}$. We can assume that $p \nmid \Delta$ because otherwise $\Delta = -p \, \Delta'$ for some fundamental discriminant $\Delta' \in \mathbb{Z},$ hence $E \cong_K E_0^{(\Delta')}$ since $\sqrt{-p} \in K$. Here the symbol $\cong_K$ means that the two elliptic curves $E$ and $E_0^{(\Delta')}$, which are defined over $\mathbb{Q}$, become isomorphic when base-changed to $K$. 
Observe that $\lvert \mathfrak{f}_E \rvert = (p \, \Delta)^2$, which follows from \eqref{eq:twisted_character} and \cite[\S~10, Proposition~1]{Ulmer_2016} because $\lvert \mathfrak{f}_{E_0} \rvert$ is coprime with $\Delta$. 
Now, \cref{thm:disjointness_intro} gives
\[
    \Gal(K(E_{\text{tors}})/K) \cong \left( \prod_{q \not\in S} \Gal(K(E[q^{\infty}])/K) \right) \times \Gal(K(E[S^{\infty}])/K)
\]
with the product running over the rational primes $q \in \mathbb{N}$ such that $q \not\in S$, where in this case the finite set $S = S_E \subseteq \mathbb{N}$ appearing in \cref{thm:disjointness_intro} consists uniquely of the primes dividing
$\lvert \mathfrak{f}_E \rvert = (p \, \Delta)^2$.
Moreover, $\Gal(K(E[\ell^m])/K) \cong (\mathcal{O}/\ell^m \mathcal{O})^{\times}$ for every prime $\ell \in \mathbb{N}$ and every $m \in \mathbb{N}$, since \labelcref{it:twist_q_not_p} and \labelcref{it:twist_2_not_p} show that for every $m \in \mathbb{N}$ the Galois representation $\rho_{E,\ell^m}$ has maximal image. 
On the other hand, \cref{prop:proposition_twist} shows that $K(E[p^m]) = H_{p^m,\mathcal{O}}(\sqrt{\Delta})$ and
\[
     K(E[p^m]) \cap K(E[\Delta]) = K(\sqrt{\Delta})
\]
for every $m \in \mathbb{Z}_{\geq 1}$. Hence the family of division fields $\{K(E[q^{\infty}])\}_{q \in S}$ is entangled over $K$, and for every collection of integers 
$\{a_q\}_{q \in S} \subseteq \mathbb{Z}_{\geq 1}$ such that $a_2 \geq 3$ we get
\[
    \Gal(L/K) \cong \frac{\prod_{q \in S} \left( \mathcal{O}/q^{a_q}\mathcal{O} \right)^{\times}}{\{\pm 1\}}    
\]
where $L$ is the compositum of all the division fields $K(E[q^{a_q}])$ for $q \in S$.

Let us now consider orders $\mathcal{O}$ such that $\operatorname{gcd}(\Delta_{\mathcal{O}},6) \neq 1$.
The analysis of the division fields of an elliptic curve $E_{/\mathbb{Q}}$ having complex multiplication by $\mathcal{O}$ proceeds similarly to what happened before, with the only exception of the order $\mathcal{O} = \mathbb{Z}[\sqrt{-3}]$. Indeed if 
\[
    \mathcal{O} \in \{ \mathbb{Z}[3 \zeta_3], \mathbb{Z}[2 i], \mathbb{Z}[\sqrt{-2}], \mathbb{Z}[\sqrt{-7}] \}
\]
where $\zeta_3 := (-1+\sqrt{-3})/2$ and $i := \sqrt{-1}$, then all the elliptic curves $E_0$ appearing in \cref{tab:twist_minimal_CM_elliptic_curves} with complex multiplication by $\mathcal{O}$ share the property that $\lvert \mathfrak{f}_{E_0} \rvert$ is a power of the unique rational prime $p \in \mathbb{N}$ which ramifies in the quadratic extension $\mathbb{Q} \subseteq K$. Hence \cref{thm:disjointness_intro} provides a decomposition
\[
    \Gal(K((E_0)_{\text{tors}})/K) \cong \prod_{q} \Gal(K(E_0[q^{\infty}])/K)
\]
where the product runs over all rational primes $q \in \mathbb{N}$, because in this case the finite set $S_{E_0} \subseteq \mathbb{N}$ appearing in \cref{thm:disjointness_intro} consists of the single prime $p$.
This shows that the division fields of $E_0$ are linearly disjoint over $K$. Moreover, \cref{prop:ramification} implies that $\Gal(K(E_0[q^m])/K) \cong (\mathcal{O}/q^m \mathcal{O})^{\times}$ for every rational prime $q \neq p$ and every $m \in \mathbb{N}$. 
On the other hand, \cref{prop:Deuring_formula} shows that $\mathfrak{f}_{\varphi_0} = \mathfrak{p}^k$ is a power of the unique prime ideal $\mathfrak{p} \subseteq \mathcal{O}_K$ lying over $p$, with $k \leq 2$ if $\mathcal{O} \not\in \{\mathbb{Z}[2i], \mathbb{Z}[\sqrt{-2}]\}$ and $k \leq 6$ otherwise. 
Hence \cref{thm:coates_wiles} and \cref{thm:galois_ray_class_fields} give $\Gal(K(E_0[p^m])/K) \cong (\mathcal{O}/p^m)^{\times}/\{\pm 1\}$ for every $m \in \mathbb{N}$ such that $m \geq 1$ if $\mathcal{O} \not\in \{\mathbb{Z}[2i], \mathbb{Z}[\sqrt{-2}]\}$ and $m \geq 3$ otherwise.

Let now $E_{/\mathbb{Q}}$ be any elliptic curve with complex multiplication by $\mathcal{O}$. Since $j(E) = j(E_0) \not\in \{ 0, 1728 \}$ we know that $E = E_0^{(\Delta)}$ for some fundamental discriminant $\Delta \in \mathbb{Z}$. 
If $\mathcal{O} = \mathbb{Z}[3 \zeta_3]$ or $\mathcal{O} = \mathbb{Z}[\sqrt{-7}]$ we can assume that $p \nmid \Delta$ because $\sqrt{-p} \in K$. 
Hence \cref{thm:disjointness_intro} shows that
\[
    \Gal(K(E_{\text{tors}})/K) \cong \left( \prod_{q \not\in S} \Gal(K(E[q^{\infty}])/K) \right) \times \Gal(K(E[S^{\infty}])/K)
\]
with the product running over the rational primes $q \in \mathbb{N}$ such that $q \not\in S$, where in this case the finite set $S = S_E \subseteq \mathbb{N}$ appearing in \cref{thm:disjointness_intro} consists uniquely of the primes dividing
$\lvert \mathfrak{f}_E \rvert = (p \, \Delta)^2$.
Exactly as before \labelcref{it:twist_q_not_p} and \labelcref{it:twist_2_not_p} show that $\Gal(K(E[\ell^m])/K) \cong (\mathcal{O}/\ell^m \mathcal{O})^{\times}$ for every prime $\ell \in \mathbb{N}$ and every $m \in \mathbb{N}$. 
Moreover, \cref{prop:proposition_twist} shows that $K(E[p^m]) = H_{p^m,\mathcal{O}}(\sqrt{\Delta})$ and 
$K(E[p^m]) \cap K(E[\Delta]) = K(\sqrt{\Delta})$ for every $m \in \Z_{\geq 1}$. Hence the family of division fields $\{K(E[q^{\infty}])\}_{q \in S}$ is entangled over $K$, and for every collection of integers 
$\{a_q\}_{q \in S} \subseteq \mathbb{Z}_{\geq 1}$ with $a_2 \geq 3$ we get
\[
    \Gal(L/K) \cong \frac{\prod_{q \in S} \left( \mathcal{O}/q^{a_q}\mathcal{O} \right)^{\times}}{\{\pm 1\}}    
\]
where $L$ is the compositum of all the division fields $K(E[q^{a_q}])$ for $q \in S$.

Studying the entanglement in the family of division fields of $E$ becomes slightly more complicated if $\mathcal{O} \in \{ \mathbb{Z}[2 i], \mathbb{Z}[\sqrt{-2}] \}$. 
First of all, note that there exists a unique $\Delta_2 \in \{ 1, -4, -8, 8 \}$ such that $\Delta = \Delta_2 \, \Delta'$ where $\Delta' \in \mathbb{Z}$ is an odd fundamental discriminant. 
We can now write $E = E_1^{(\Delta')}$ where $E_1 := E_0^{(\Delta_2)}$. One can check that if $\mathcal{O} = \mathbb{Z}[\sqrt{-2}]$ then $E_1$ is isomorphic to one of the four elliptic curves with complex multiplication by $\mathbb{Z}[\sqrt{-2}]$ appearing in \cref{tab:twist_minimal_CM_elliptic_curves}. 
On the other hand, if $\mathcal{O} = \mathbb{Z}[2 i]$ then $E_1$ can be either one of the two elliptic curves
\begin{equation} \label{eq:pristine_curves_2i}
    \begin{aligned}
    y^2 &= x^3 - 44 x - 112 \\[-7pt]
    y^2 &= x^3 - 44 x + 112
\end{aligned}
\end{equation}
or one of the two elliptic curves with complex multiplication by $\mathbb{Z}[2 i]$ appearing in \cref{tab:twist_minimal_CM_elliptic_curves}. In each case it is not difficult to see that $\lvert \mathfrak{f}_{E_1} \rvert \in \mathbb{N}$ is a power of $2$, which shows that the division fields of $E_1$ behave similarly to the division fields of $E_0$.
More precisely, \cref{thm:disjointness_intro} gives
\[
    \Gal(K((E_1)_{\text{tors}})/K) \cong \prod_{q} \Gal(K(E_1[q^{\infty}])/K)
\]
where the product runs over all the rational primes $q \in \mathbb{N}$.
This shows that the division fields of $E_1$ are linearly disjoint over $K$.
Moreover, \cref{prop:ramification} shows that  $\Gal(K(E_1[q^m])/K) \cong (\mathcal{O}/q^m \mathcal{O})^{\times}$ for every rational prime $q \geq 3$ and every $m \in \mathbb{N}$, and a combination of \cref{prop:Deuring_formula} and \cref{thm:coates_wiles} gives $\Gal(K(E_1[2^m])/K) \cong (\mathcal{O}/2^m \mathcal{O})^{\times}/\{ \pm 1 \}$ for every $m \in \mathbb{N}$ such that $m \geq 3$. 
This concludes the analysis of the division fields of $E = E_1$ if $\Delta' = 1$. On the other hand, if $\Delta' \neq 1$ then $\lvert \mathfrak{f}_E \rvert =  \lvert \mathfrak{f}_{E_1} \rvert \, (\Delta')^2$ where $\lvert \mathfrak{f}_{E_1} \rvert$ is a power of $2$. Hence \cref{thm:disjointness_intro} shows that
\[
    \Gal(K(E_{\text{tors}})/K) \cong \left( \prod_{q \not\in S} \Gal(K(E[q^{\infty}])/K) \right) \times \Gal(K(E[S^{\infty}])/K)
\]
with the product running over the rational primes $q \in \mathbb{N}$ such that $q \not\in S$ where $S = S_E$ denotes the finite set appearing in \cref{thm:disjointness_intro}, which in this case consists of the primes dividing $2 \cdot \Delta'$. 
Similarly to what happened before, \labelcref{it:twist_q_not_p} and \labelcref{it:twist_2} show that $\operatorname{Gal}(K(E[\ell^m])/K) \cong (\mathcal{O}/\ell^m \mathcal{O})^{\times}$ for every prime $\ell \in \mathbb{N}$ and every $m \in \mathbb{N}$. Moreover, \cref{prop:proposition_twist} gives $K(E[2^m]) = H_{2^m,\mathcal{O}}(\sqrt{\Delta'})$ and 
$K(E[2^m]) \cap K(E[\Delta']) = K(\sqrt{\Delta'})$
for every $m \geq 3$.
Hence the family of division fields $\{K(E[q^{\infty}])\}_{q \in S}$ is entangled over $K$, and for all 
$\{a_q\}_{q \in S} \subseteq \mathbb{Z}_{\geq 1}$ with $a_2 \geq 3$ we get
\[
    \Gal(L/K) \cong \frac{\prod_{q \in S} \left( \mathcal{O}/q^{a_q}\mathcal{O} \right)^{\times}}{\{\pm 1\}}    
\]
where $L$ is the compositum of all the division fields $K(E[q^{a_q}])$ for $q \in S$.

We are left with the analysis of the entanglement between the division fields of an elliptic curve $E$ defined over $\mathbb{Q}$ which has complex multiplication by $\mathcal{O} = \mathbb{Z}[\sqrt{-3}]$. As usual $E = E_0^{(\Delta)}$ for some fundamental discriminant $\Delta \in \mathbb{Z}$, where $E_0$ is one of the two elliptic curves with complex multiplication by $\mathbb{Z}[\sqrt{-3}]$ appearing in \cref{tab:twist_minimal_CM_elliptic_curves}. In contrast to what we have seen before, here $\lvert \mathfrak{f}_{E_0} \rvert = 2^2 \, 3^2$ is not a prime power. This forces us to study separately the division fields $K(E_0[2^\infty])$ and $K(E_0[3^\infty])$. 
First of all, one can compute that for any of the two possibilities for $E_0$, given by the Weierstrass equations $y^2 = x^{3} - 15 x + 22$ and $ y^2 = x^{3} - 135 x - 594$,
the representation $\rho_{E_0,3}$ is not surjective, \textit{i.e.} $K(E_0[3]) = H_{3,\mathcal{O}} = K(\sqrt[3]{2})$. This clearly shows that $\rho_{E_0,3^n}$ is not surjective for every $n \in \mathbb{Z}_{\geq 1}$. Moreover, $\rho_{E_0,2^n}$ is surjective for every $n \in \mathbb{Z}_{\geq 1}$.
Indeed, \cref{thm:galois_ray_class_fields} and \cref{prop:hilbert_12} imply that
\[ 
    \left\lvert \left( \frac{\mathcal{O}}{2^n \mathcal{O}} \right)^{\times} \right\rvert = \frac{[H_{2^n \, 3,\mathcal{O}} \colon K]}{[H_{3,\mathcal{O}} \colon K]} = \frac{[H_{2^n \, 3,\mathcal{O}} \colon K]}{[K(E_0[3]) \colon K]} \leq \frac{[K(E_0[2^n \, 3]) \colon K]}{[K(E_0[3]) \colon K]} \leq [K(E_0[2^n]) \colon K]
\]
hence \cref{lem:free_O-module} shows that every inequality is actually an equality, and $\rho_{E_0,2^n}$ is surjective.
This gives that $K(E_0[2^n]) \cap K(E_0[3^m]) = K$ for every $n, m \in \mathbb{Z}_{\geq 1}$. These considerations together with \cref{thm:disjointness_intro} and \cref{prop:ramification} give a decomposition
\[
    \Gal(K((E_0)_{\text{tors}})/K) \cong \prod_{q} \Gal(K(E_0[q^{\infty}])/K)
\]
where the product runs over all rational primes $q \in \mathbb{N}$. Moreover, for every $m \in \mathbb{N}$ we get
\[
    \Gal(K(E_0[q^m])/K) \cong \begin{cases} (\mathcal{O}/q^m \mathcal{O})^{\times}, \ \text{if} \ q \neq 3 \\
    (\mathcal{O}/3^m \mathcal{O})^{\times}/\{\pm 1\}, \ \text{if} \ q = 3 \end{cases}    
\]
and the family of division fields $\{K(E[q^{\infty}])\}_q$ is linearly disjoint over $K$.

Let us go back to the division fields of $E = E_0^{(\Delta)}$, where we can assume that $3 \nmid \Delta$ because $\sqrt{-3} \in K$. Write now $\Delta = \Delta_2 \, \Delta'$ as above, where $\Delta_2 \in \{1,-4,-8,8\}$ and $\Delta' \in \mathbb{Z}$ an odd fundamental discriminant, and let $E_1 := E_0^{(\Delta_2)}$. 
Then \labelcref{it:twist_2} implies that $\rho_{E_1,2^n}$ is surjective for every $n \geq 3$. Moreover, $\rho_{E_1,3^n}$ is surjective for every $n \geq 1$, which follows from \cref{prop:proposition_twist} after observing that $K(E_0[3]) \cap K(\sqrt{\Delta_2}) = K$ because $[K(E_0[3]) \colon K] = 3$. These considerations, together with \cref{thm:disjointness_intro}, show that
\[
    \Gal(K((E_1)_{\text{tors}})/K) \cong \left( \prod_{q \not\in S} \Gal(K(E_1[q^{\infty}])/K) \right) \times \Gal(K(E_1[S^{\infty}])/K)
\]
with the product running over the rational primes $q \in \mathbb{N}$ such that $q \not\in S$ where $S = \{2,3\}$ and $K(E_1[S^{\infty}])$ denotes the compositum of the division fields $K(E_1[2^{\infty}])$ and $K(E_1[3^{\infty}])$.
Moreover, \labelcref{it:twist_q_not_p}, \labelcref{it:twist_p} and the previous considerations show that $\Gal(K(E_1[\ell^m])/K) \cong (\mathcal{O}/\ell^m \mathcal{O})^{\times}$ for every prime $\ell \in \mathbb{N}$ and every $m \in \mathbb{N}$. 
Now, \cref{prop:proposition_twist} shows that $K(E_1[3^m]) = H_{3^m,\mathcal{O}}(\sqrt{\Delta_2})$ and $K(E_1[3^m]) \cap K(E_1[\Delta_2]) = K(\sqrt{\Delta_2})$ for every $m \in \Z_{\geq 1}$. 
Hence $K(E_1[2^{\infty}])$ and $K(E_1[3^{\infty}])$ are entangled over $K$, and for every pair of integers $a, b \in \mathbb{Z}_{\geq 1}$ with $a \geq 3$ we have that
\[
    \operatorname{Gal}(L/K) \cong \frac{(\mathcal{O}/2^a \mathcal{O})^{\times} \times (\mathcal{O}/3^b \mathcal{O})^{\times}}{\{\pm 1\}}
\]
where $L$ denotes the compositum of $K(E_1[2^a])$ and $K(E_1[3^b])$.

To conclude our analysis of the division fields of $E = E_0^{(\Delta)}$ we can observe that $E = E_1^{(\Delta')}$ and that $\operatorname{gcd}(\Delta',\mathfrak{f}_{E_1}) = \operatorname{gcd}(\Delta',6) = 1$. Hence \cref{thm:disjointness_intro} gives the decomposition
\[
    \Gal(K(E_{\text{tors}})/K) \cong \left( \prod_{q \not\in S} \Gal(K(E[q^{\infty}])/K) \right) \times \Gal(K(E[S^{\infty}])/K)
\]
with the product running over the rational primes $q \in \mathbb{N}$ such that $q \not\in S$ where $S \subseteq \mathbb{N}$ denotes the finite set of primes dividing $6 \Delta'$.
Now, \labelcref{it:twist_q_not_p} and \labelcref{it:twist_p} show that $\Gal(K(E[\ell^m])/K) \cong (\mathcal{O}/\ell^m)^{\times}$ for all rational primes $\ell \in \mathbb{Z}$ and all $m \in \mathbb{N}$. Moreover, \cref{prop:proposition_twist} shows that
$K(E[3^m]) \cap K(E[\Delta]) = K(\sqrt{\Delta})$
and $K(E[3^m]) = H_{3^m,\mathcal{O}}(\sqrt{\Delta})$ for every $m \in \Z_{\geq 1}$. Hence the family $\{ K(E[q^{\infty}]) \}_{q \in S}$ is entangled over $K$, and for every collection of integers 
$\{a_q\}_{q \in S} \subseteq \mathbb{Z}_{\geq 1}$ such that $a_2 \geq 3$ we get
\[
    \Gal(L/K) \cong \frac{\prod_{q \in S} \left( \mathcal{O}/q^{a_q}\mathcal{O} \right)^{\times}}{\{\pm 1\}}    
\]
where $L$ is the compositum of all the division fields $K(E[q^{a_q}])$ for $q \in S$.

The following theorem summarises the previous discussion. Recall that, for every rational prime $q \in \mathbb{N}$, we denote by $K(E[q^{\infty}])$ the compositum of all the division fields $\{ K(E[q^n]) \}_{n \in \mathbb{N}}$ associated to an elliptic curve $E$, and for every finite set of primes $S \subseteq \N$ we denote by $K(E[S^{\infty}])$ the compositum of all the fields $\{K(E[q^{\infty}])\}_{q \in S}$.
\begin{theorem} \label{thm:classification_entanglement}
    Let $\mathcal{O}$ be an order inside an imaginary quadratic field $K$ such that $\operatorname{Pic}(\mathcal{O}) = 1$ and $\Delta_{\mathcal{O}} < -4$.
    We introduce the following notation:
    \[
        n = n(\mathcal{O}) := \begin{cases} 4, \ \text{if} \ \mathcal{O} \in \{ \mathbb{Z}[2 i], \mathbb{Z}[\sqrt{-2}] \} \\
        2, \ \text{otherwise}
        \end{cases} \ \text{and} \quad
        \begin{aligned}
            p &\in \mathbb{N} \ \text{the unique prime ramifying in} \ \mathbb{Q} \subseteq K, \\
        \mathfrak{p} &\subseteq \mathcal{O}_K \ \text{the unique prime lying above} \ p .
        \end{aligned}
    \]
    Label all the elliptic curves defined over $\mathbb{Q}$ which have complex multiplication by $\mathcal{O}$ as $\{ A_r \}_{r \in \mathbb{Z}_{\geq 1}}$ in such a way that $\lvert \mathfrak{f}_{A_r} \rvert \leq \lvert \mathfrak{f}_{A_{r + 1}} \rvert$ for every $r \in \mathbb{Z}_{\geq 1}$. Then $\lvert \mathfrak{f}_{A_n} \rvert < \lvert \mathfrak{f}_{A_{n + 1}} \rvert$ and the properties of the division fields associated to the elliptic curve $A_r$ depend on $r$ as follows:
    
    \vspace{0.3\baselineskip}
    \noindent
    \begin{minipage}[t]{0.08\textwidth}
        \fbox{$r \leq n$}
    \end{minipage}
    \begin{minipage}[t]{0.91\textwidth}
    \begin{description}[leftmargin=*]
            \item[Disjointness]
            the family $\{ K(A_r[q^{\infty}])\}_q$, where $q \in \mathbb{N}$ runs over all the rational primes, is linearly disjoint over $K$;
            \item[Maximality] $\Gal(K(A_r[q^m])/K) \cong (\mathcal{O}/q^m \mathcal{O})^{\times}$ for every prime $q \neq p$ and every $m \in \mathbb{N}$;
            \item[Minimality] $\Gal(K(A_r[p^m])/K) \cong (\mathcal{O}/p^{m} \mathcal{O})^{\times}/\{\pm 1\}$ for every $m \geq n - 1$;
        \end{description}
    \end{minipage}
    
    \vspace{0.3\baselineskip}
    \noindent
    \begin{minipage}[t]{0.08\textwidth}
        \fbox{$r > n$}
    \end{minipage}
    \begin{minipage}[t]{0.91\textwidth}
    \begin{description}[leftmargin=*]
            \item[Twist] there exists a unique $r_0 \leq n$ and a unique fundamental discriminant $\Delta_r \in \mathbb{Z}$ coprime with $p$ such that $A_r = A_{r_0}^{(\Delta_r)}$;
            \item[Disjointess] there is a decomposition
            \[
            \Gal(K((A_r)_{\text{tors}})/K) \cong \left( \prod_{q \not\in S_r} \Gal(K(A_r[q^{\infty}])/K) \right) \times \Gal(K(A_r[S^{\infty}])/K)
            \] 
            where $S_r \subseteq \mathbb{N}$ denotes the finite set of primes dividing $p \cdot \Delta_r$ and the product runs over the rational primes $q \in \mathbb{N}$ such that $q \not\in S_r$.
            This shows that the family 
            \[
                \{ \, K(A_r[S_r^{\infty}]) \, \} \cup
                \{ \, K(A_r[q^{\infty}]) \, \}_{q \not\in S_r}
            \]
            is linearly disjoint over $K$;
            \item[Entanglement]
            for every $m \in \mathbb{N}$ such that $m \geq n - 1$ we have that
            \[
            K(A_r[p^m]) = H_{p^m,\mathcal{O}}(\sqrt{\Delta_r})
            \qquad \text{and} \qquad 
            K(A_r[p^m]) \cap K(A_r[\Delta_r]) =  K(\sqrt{\Delta_r})
            \]
            which shows that the family $\{K(A_r[q^{\infty}])\}_{q \in S_r}$ is entangled over $K$;
            \item[Maximality]
            $\Gal(K(A_r[q^m])/K) \cong (\mathcal{O}/q^m \mathcal{O})^{\times}$ for every prime $q \in \mathbb{N}$ and every $m \in \mathbb{N}$;
            \item[Minimality] for every collection of integers $\{a_q\}_{q \in S_r} \subseteq \mathbb{Z}_{\geq 1}$ with $a_2 \geq 3$ we get
            \[
                \Gal(L/K) \cong \frac{\prod_{q \in S_r} \left( \mathcal{O}/q^{a_q}\mathcal{O} \right)^{\times}}{\{\pm 1\}}
            \]
            where $L$ is the compositum of all the division fields $K(A_r[q^{a_q}])$ for $q \in S_r$.
        \end{description}
    \end{minipage}
\end{theorem}
\begin{remark} \label{rmk:minimality_of_S_over_Q}
    We claim that \cref{thm:classification_entanglement} implies that the isomorphism \eqref{eq:iso_entanglement} appearing in \cref{thm:disjointness_intro} does not hold in general if the set $S$ does not contain all the primes dividing the integer $B_E := \fgotic_\mathcal{O} \, \Delta_F \, N_{F/\Q}(\fgotic_{E}) \in \mathbb{Z}$.
    To see this, fix an imaginary quadratic order $\mathcal{O}$ having trivial class group $\operatorname{Pic}(\mathcal{O}) = \{1\}$, conductor $\mathfrak{f}_\mathcal{O} \neq 2$ and discriminant $\Delta_\mathcal{O} < -4$. Let $n = n(\mathcal{O}) \in \{2,4\}$ be as in \cref{thm:classification_entanglement}.
    Then, if we take $E = A_r$ for any $r > n$, \cref{thm:classification_entanglement} shows that \eqref{eq:iso_entanglement} does not hold for any set $S$ which does not contain the set $S_r$ appearing in \cref{thm:classification_entanglement}.
    Since this set $S_r$ coincides with the set of primes dividing the integer $B_E = B_{A_r}$, this proves our claim.
\end{remark}
\begin{remark} \label{rmk:0_1728_excluded}
    We exclude the two orders $\Z[i]$ and $\Z[\zeta_3]$ in the statement of \cref{thm:classification_entanglement} because elliptic curves having complex multiplication by these orders admit quartic (respectively sextic) twists (as explained in \cite[Chapter~X, Proposition~5.4]{si09}). To study these we would need a generalisation of \cref{prop:proposition_twist}, which will be subject of future investigations.
\end{remark}

\section*{Acknowledgements}

We would like to thank Ian Kiming, Fabien Pazuki and Peter Stevenhagen, for their precious guidance and constant encouragement during the preparation of this work.
We thank the anonymous referee for their helpful comments and suggestions.

\vspace{\baselineskip}
\noindent
\framebox[\textwidth]{
\begin{tabular*}{0.96\textwidth}{@{\extracolsep{\fill} }cp{0.84\textwidth}}
 % The EU emblem
\raisebox{-0.7\height}{%
    \begin{tikzpicture}[y=0.80pt, x=0.8pt, yscale=-1, inner sep=0pt, outer sep=0pt, 
    scale=0.12]
    \definecolor{c003399}{RGB}{0,51,153}
    \definecolor{cffcc00}{RGB}{255,204,0}
    \begin{scope}[shift={(0,-872.36218)}]
      \path[shift={(0,872.36218)},fill=c003399,nonzero rule] (0.0000,0.0000) rectangle (270.0000,180.0000);
      \foreach \myshift in 
           {(0,812.36218), (0,932.36218), 
    		(60.0,872.36218), (-60.0,872.36218), 
    		(30.0,820.36218), (-30.0,820.36218),
    		(30.0,924.36218), (-30.0,924.36218),
    		(-52.0,842.36218), (52.0,842.36218), 
    		(52.0,902.36218), (-52.0,902.36218)}
        \path[shift=\myshift,fill=cffcc00,nonzero rule] (135.0000,80.0000) -- (137.2453,86.9096) -- (144.5106,86.9098) -- (138.6330,91.1804) -- (140.8778,98.0902) -- (135.0000,93.8200) -- (129.1222,98.0902) -- (131.3670,91.1804) -- (125.4894,86.9098) -- (132.7547,86.9096) -- cycle;
    \end{scope}
    %\draw[very thin,dashed] (current bounding box.south west) rectangle               (current bounding box.north east);
    \end{tikzpicture}%
}
&
This project has received funding from the European Union Horizon 2020 research and
innovation programme under the Marie Sk{\l}odowska-Curie grant agreement No 801199.
\end{tabular*}
}

\vspace{\baselineskip}
\noindent
\framebox[\textwidth]{
\begin{tabular*}{0.96\textwidth}{@{\extracolsep{\fill} }cp{0.84\textwidth}}
 % The EU emblem
 \raisebox{-0.7\height}{
\includegraphics[width=0.07\textwidth]{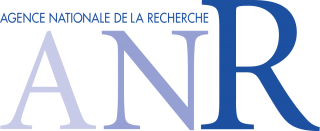}}
&
Ce travail a été réalisé au sein du LABEX MILYON (ANR-10-LABX-0070) de l'Université de Lyon, dans le cadre du programme
``Investissements d'Avenir'' (ANR-11-IDEX-0007) géré par l'Agence Nationale de la Recherche (ANR).
\end{tabular*}
}

\printbibliography

@Book{si09,
 Author = {J. H. {Silverman}},
 Title = {{The arithmetic of elliptic curves. 2nd ed.}},
 FJournal = {{Graduate Texts in Mathematics}},
 Journal = {{Grad. Texts Math.}},
 ISSN = {0072-5285},
 Volume = {106},
 Edition = {2nd ed.},
 ISBN = {978-0-387-09493-9},
 Pages = {xx + 513},
 Year = {2009},
 Publisher = {New York, NY: Springer},
 MSC2010 = {11-01 11G05 14-01 11G07 11G10 11G20 11G40 14H52 14G05 11Gxx 14H25 11Y16 20J06},
 Zbl = {1194.11005}
}

@Article{st08,
 Author = {M. {Streng}},
 Title = {{Divisibility sequences for elliptic curves with complex multiplication.}},
 FJournal = {{Algebra \& Number Theory}},
 Journal = {{Algebra \& Number Theory}},
 ISSN = {1937-0652},
 Volume = {2},
 Number = {2},
 Pages = {183--208},
 Year = {2008},
 Publisher = {Mathematical Sciences Publishers (MSP), Berkeley, CA},
 MSC2010 = {11G05 11G15 14H52 14K22},
 Zbl = {1158.14029},
 DOI = {10.2140/ant.2008.2.183}
}

@misc{silverman-errata,
  title        = {Errata and Corrections to \textit{The Arithmetic of Elliptic Curves}, 2nd Edition},
  author = {{Silverman}, Joseph H.},
  url = {http://www.math.brown.edu/~jhs/AEC/AECErrata.pdf},
  year = {2015}
}

@Book{si94,
 Author = {J. H. {Silverman}},
 Title = {{Advanced topics in the arithmetic of elliptic curves.}},
 FJournal = {{Graduate Texts in Mathematics}},
 Journal = {{Grad. Texts Math.}},
 ISSN = {0072-5285},
 Volume = {151},
 ISBN = {0-387-94325-0},
 Pages = {xiii + 525},
 Year = {1994},
 Publisher = {New York, NY: Springer-Verlag},
 MSC2010 = {14H52 14G40 11-02 11Gxx 14-02 11G07 11G05 14G10},
 Zbl = {0911.14015}
}

@book{Neukirch_1999, place={Berlin, Heidelberg}, series={Grundlehren der mathematischen Wissenschaften}, title={Algebraic Number Theory}, volume={322}, ISBN={978-3-642-08473-7}, url={http://link.springer.com/10.1007/978-3-662-03983-0}, DOI={10.1007/978-3-662-03983-0}, publisher={Springer Berlin Heidelberg}, author={Neukirch, Jürgen}, year={1999}, collection={Grundlehren der mathematischen Wissenschaften} }

@article{Serre_1971, title={Propriétés galoisiennes des points d’ordre fini des courbes elliptiques}, volume={15}, ISSN={1432-1297}, url={https://doi.org/10.1007/BF01405086}, DOI={10.1007/BF01405086}, number={4}, journal={Inventiones mathematic{\ae}}, author={Serre, Jean-Pierre}, year={1971}, pages={259–331} }

@misc{Campagna_Stevenhagen,
    title={Cyclic reduction of Elliptic Curves},
    author={Francesco Campagna and Peter Stevenhagen},
    url={https://arxiv.org/abs/2001.00028},
    note = {arXiv:2001.00028},
    year="2020"
}

@article {Lombardo_2017,
    AUTHOR = {Lombardo, Davide},
     TITLE = {Galois representations attached to abelian varieties of {CM}
              type},
    JOURNAL = {Bulletin de la Soci\'{e}t\'{e} Math\'{e}matique de France},
    VOLUME = {145},
      YEAR = {2017},
    NUMBER = {3},
     PAGES = {469--501},
      ISSN = {0037-9484},
   MRCLASS = {11G10 (11F80 11G05 14K15 14K22)},
  MRNUMBER = {3766118},
MRREVIEWER = {Remke Kloosterman},
       URL = {https://smf.emath.fr/publications/representations-galoisiennes-associees-aux-varietes-abeliennes-de-type-cm}
}

@misc{Lozano-Robledo_2019, title={Galois representations attached to elliptic curves with complex multiplication}, url={http://arxiv.org/abs/1809.02584}, note={arXiv:1809.02584}, author={Lozano-Robledo, Álvaro}, year={2019} }

@article{Bourdon_Clark_Stankewicz_2017, title={Torsion points on CM elliptic curves over real number fields}, volume={369}, ISSN={0002-9947, 1088-6850}, DOI={10.1090/tran/6905}, number={12}, journal={Transactions of the American Mathematical Society}, author={Bourdon, Abbey and Clark, Pete and Stankewicz, James}, year={2017}, pages={8457–8496} }

@article{Bourdon_Clark_2020, title={Torsion points and Galois representations on CM elliptic curves}, volume={305}, ISSN={1945-5844}, DOI={10.2140/pjm.2020.305.43}, number={1}, journal={Pacific Journal of Mathematics}, publisher={Mathematical Sciences Publishers}, author={Bourdon, Abbey and Clark, Pete L.}, year={2020}, pages={43–88} }

@misc{Lozano_Robledo_prep, author = {Lozano-Robledo, Álvaro}, title = {Applications of a classification of Galois representations attached to elliptic curves with complex multiplication}, note = {\textit{In preparation}}, url = {https://alozano.clas.uconn.edu/research-articles/}}

@misc{Campagna_Stevenhagen_CM, author = {Campagna, Francesco and Stevenhagen, Peter}, title = {Cyclic reduction of CM elliptic curves}, note = {\textit{In preparation}}, url = {https://sites.google.com/view/francesco-campagna/research}}

@misc{Pengo,
    title={Mahler's measure and elliptic curves with potential complex multiplication},
    author={Riccardo Pengo},
    year={2020},
    note={arXiv:2005.04159},
    url={https://arxiv.org/abs/2005.04159}
}

@article {Coates_2013,
    AUTHOR = {Coates, John},
     TITLE = {Lectures on the {B}irch-{S}winnerton-{D}yer conjecture},
    JOURNAL = {Notices of the International Congress of Chinese Mathematicians},
    VOLUME = {1},
      YEAR = {2013},
    NUMBER = {2},
     PAGES = {29--46},
      ISSN = {2326-4810},
   MRCLASS = {11G15 (14G10 14G25)},
  MRNUMBER = {3310602},
       DOI = {10.4310/ICCM.2013.v1.n2.a5}
}

@misc{Smith_2018, title={Ramification in the Division Fields of Elliptic Curves and an Application to Sporadic Points on Modular Curves}, url={http://arxiv.org/abs/1810.04809}, note={arXiv:1810.04809}, author={Smith, Hanson}, year={2018}}

@article{Coates_Wiles_1977, title={On the conjecture of Birch and Swinnerton-Dyer}, volume={39}, ISSN={1432-1297}, DOI={10.1007/BF01402975}, number={3}, journal={Inventiones mathematic{\ae}}, author={Coates, J. and Wiles, A.}, year={1977}, pages={223–251} }

@article{Lozano-Robledo_2016, title={Ramification in the division fields of elliptic curves with potential supersingular reduction}, volume={2}, ISSN={2363-9555}, DOI={10.1007/s40993-016-0040-z}, number={1}, journal={Research in Number Theory}, author={Lozano-Robledo, Álvaro}, year={2016}, pages={8} }

@article{Lozano-Robledo_2018, title={Uniform boundedness in terms of ramification}, volume={4}, ISSN={2363-9555}, DOI={10.1007/s40993-018-0095-0}, number={1}, journal={Research in Number Theory}, author={Lozano-Robledo, Álvaro}, year={2018}, pages={6} }

@article{Lozano-Robledo_2013-formal, title={Formal groups of elliptic curves with potential good supersingular reduction}, volume={261}, ISSN={0030-8730}, DOI={10.2140/pjm.2013.261.145}, number={1}, journal={Pacific Journal of Mathematics}, author={Lozano-Robledo, Álvaro}, year={2013}, pages={145–164} }

@article{Arthaud_1978, title={On Birch and Swinnerton-Dyer’s conjecture for elliptic curves with complex multiplication. I}, volume={37}, number={2}, journal={Compositio Mathematica}, author={Arthaud, Nicole}, year={1978}, pages={209–232},url={http://www.numdam.org/item/?id=CM_1978__37_2_209_0} }

@book{Cox_2013, edition={Second}, series={Pure and Applied Mathematics (Hoboken)}, title={Primes of the form $x^2 + ny^2$}, ISBN={978-1-118-39018-4}, doi={10.1002/9781118400722}, publisher={John Wiley \& Sons, Inc., Hoboken, NJ}, author={Cox, David A.}, year={2013}, collection={Pure and Applied Mathematics (Hoboken)} }

@book{Shimura_1998, series={Princeton Mathematical Series}, title={Abelian varieties with complex multiplication and modular functions}, volume={46}, ISBN={978-0-691-01656-6}, url={https://mathscinet.ams.org/mathscinet-getitem?mr=1492449}, DOI={10.1515/9781400883943}, publisher={Princeton University Press, Princeton, NJ}, author={Shimura, Goro}, year={1998}, collection={Princeton Mathematical Series} }

@inproceedings{Stevenhagen_2001,
author = "Stevenhagen, Peter",
booktitle = "Class Field Theory – Its Centenary and Prospect",
doi = "10.2969/aspm/03010161",
pages = "161--176",
publisher = "Mathematical Society of Japan",
title = "Hilbert’s 12th Problem, Complex Multiplication and Shimura Reciprocity",
url = "https://doi.org/10.2969/aspm/03010161",
year = "2001"
}

@book{Lang_1987, edition={Second}, series={Graduate Texts in Mathematics}, title={Elliptic functions}, volume={112}, ISBN={978-0-387-96508-6}, doi={10.1007/978-1-4612-4752-4}, publisher={Springer-Verlag, New York}, author={Lang, Serge}, year={1987}, collection={Graduate Texts in Mathematics} }

@article{Murty_1983, title={On Artin’s conjecture}, volume={16}, ISSN={0022-314X}, DOI={10.1016/0022-314X(83)90039-2}, number={2}, journal={Journal of Number Theory}, author={Murty, M. Ram}, year={1983}, pages={147–168} }

@article{Milne_1972, title={On the arithmetic of abelian varieties}, volume={17}, ISSN={1432-1297}, DOI={10.1007/BF01425446}, number={3}, journal={Inventiones mathematicae}, author={Milne, J. S.}, year={1972}, pages={177–190} }

@article{Lv_Deng_2015, title={On orders in number fields: Picard groups, ring class fields and applications}, volume={58}, ISSN={1674-7283, 1869-1862}, DOI={10.1007/s11425-015-4979-3}, number={8}, journal={Science China Mathematics}, author={Lv, Chang and Deng, YingPu}, year={2015}, pages={1627–1638} }

@phdthesis{KUHMAN_1978, place={United States -- California}, title={On the Conjecture of Birch and Swinnerton-Dyer for Elliptic Curves with Complex Multiplication.}, url={https://search.proquest.com/docview/302939165/citation/3EFFC3D8F7F24554PQ/1}, school={Stanford University}, author={Kuhman, Nicole Claudine}, year={1978} }

@book{Shimura_1994, series={Publications of the Mathematical Society of Japan}, title={Introduction to the arithmetic theory of automorphic functions}, volume={11}, ISBN={978-0-691-08092-5}, url={https://press.princeton.edu/books/paperback/9780691080925/introduction-to-arithmetic-theory-of-automorphic-functions}, publisher={Princeton University Press, Princeton, NJ}, author={Shimura, Goro}, year={1994}, collection={Publications of the Mathematical Society of Japan} }

@book{Husemoller_2004, edition={Second}, series={Graduate Texts in Mathematics}, title={Elliptic curves}, volume={111}, ISBN={978-0-387-95490-5}, doi={10.1007/978-1-4757-5119-2}, publisher={Springer-Verlag, New York}, author={Husemöller, Dale}, year={2004}, collection={Graduate Texts in Mathematics} }

@book{Artin_Tate_1968, title={Class field theory}, url={https://bookstore.ams.org/chel-366-h/}, publisher={W. A. Benjamin, Inc., New York-Amsterdam}, author={Artin, E. and Tate, J.}, year={1968} }

@book{Washington_2008, edition={Second edition}, series={Discrete Mathematics and its Applications (Boca Raton)}, title={Elliptic curves}, ISBN={978-1-4200-7146-7}, url={https://www.routledge.com/Elliptic-Curves-Number-Theory-and-Cryptography-Second-Edition/Washington/p/book/9781420071467}, publisher={Chapman \& Hall/CRC, Boca Raton, FL}, author={Washington, Lawrence C.}, year={2008}, collection={Discrete Mathematics and its Applications (Boca Raton)} }

@misc{Yi_Lv_2018,
    title={On Ring Class Fields of Number Rings},
    author={Hairong Yi and Chang Lv},
    year={2018},
    note={arXiv:1810.04810},
    url={https://arxiv.org/abs/1810.04810}
}

@misc{Chen_2020,
author = "Chen, Justin",
note = "In: \textit{Journal of Commutative Algebra} (to appear)",
title = "Surjections of unit groups and semi-inverses",
url = "https://projecteuclid.org:443/euclid.jca/1545015624"
}

@book{Eisenbud_1995, series={Graduate Texts in Mathematics}, title={Commutative algebra}, volume={150}, ISBN={978-0-387-94268-1}, DOI={10.1007/978-1-4612-5350-1}, publisher={Springer-Verlag, New York}, author={Eisenbud, David}, year={1995}, collection={Graduate Texts in Mathematics} }

@book{Bourbaki_Commutative_Algebra, title={Commutative algebra. Chapters 1–7}, ISBN={978-3-540-19371-5}, url={https://www.springer.com/gp/book/9783540642398}, publisher={Springer-Verlag, Berlin}, author={Bourbaki, Nicolas}, year={1989}, collection={Elements of Mathematics (Berlin)} }

@article{Sohngen_1935, title={Zur komplexen Multiplikation}, volume={111}, ISSN={1432-1807}, DOI={10.1007/BF01472223}, number={1}, journal={Mathematische Annalen}, author={Söhngen, Heinz}, year={1935}, pages={302–328} }

@article {Serre_Tate_1968,
    AUTHOR = {Serre, Jean-Pierre and Tate, John},
     TITLE = {Good reduction of abelian varieties},
  JOURNAL = {Annals of Mathematics. Second Series},
    VOLUME = {88},
      YEAR = {1968},
     PAGES = {492--517},
      ISSN = {0003-486X},
   MRCLASS = {14.51},
  MRNUMBER = {0236190},
MRREVIEWER = {M. J. Greenberg},
       DOI = {10.2307/1970722}
}

@article{Ulmer_2016, title={Conductors of $\ell$-adic representations}, volume={144}, ISSN={0002-9939, 1088-6826}, DOI={10.1090/proc/12880}, abstractNote={We give a new formula for the Artin conductor of an -adic representation of the Weil group of a local field of residue characteristic .}, number={6}, journal={Proceedings of the American Mathematical Society}, author={Ulmer, Douglas}, year={2016}, pages={2291–2299} }

@book {Schertz_2010,
    AUTHOR = {Schertz, Reinhard},
     TITLE = {Complex multiplication},
    SERIES = {New Mathematical Monographs},
    VOLUME = {15},
 PUBLISHER = {Cambridge University Press, Cambridge},
      YEAR = {2010},
     PAGES = {xiv+361},
      ISBN = {978-0-521-76668-5},
   MRCLASS = {11G15 (11G05 14G50 14H52)},
  MRNUMBER = {2641876},
MRREVIEWER = {Eugenio Jes\'{u}s G\'{o}mez Ayala},
       DOI = {10.1017/CBO9780511776892},
       URL = {https://doi-org.ep.fjernadgang.kb.dk/10.1017/CBO9780511776892},
}

@book{Gras_2003, series={Springer Monographs in Mathematics}, title={Class field theory}, ISBN={978-3-540-44133-5}, doi={10.1007/978-3-662-11323-3}, publisher={Springer-Verlag, Berlin}, author={Gras, Georges}, year={2003}, collection={Springer Monographs in Mathematics} }

@misc{Campagna_Pengo_II, author = {Campagna, Francesco and Pengo, Riccardo}, title = {On the index of Galois representations attached to elliptic curves with complex multiplication}, note = {\textit{In preparation}}, url = {https://sites.google.com/view/francesco-campagna/research}}

@article{Shimura_1971, title={On the Zeta-Function of an Abelian Variety with Complex Multiplication}, volume={94}, ISSN={0003-486X}, DOI={10.2307/1970768}, number={3}, journal={Annals of Mathematics}, author={Shimura, Goro}, year={1971}, pages={504–533} }

@article{Robert_1983, title={Sur le corps de définition de certaines courbes elliptiques à multiplications complexes}, ISSN={0989-5558}, journal={Séminaire de Théorie des Nombres de Bordeaux}, publisher={Société Arithmétique de Bordeaux}, author={Robert, Gilles}, year={1983}, pages={1–18}, url={http://www.jstor.org/stable/44165480} }

@misc{Gurney_2019, title={Frobenius lifts and elliptic curves with complex multiplication}, url={http://arxiv.org/abs/1910.14358}, note={In: \textit{{Algebra \& Number Theory} (to appear)}}, author={Gurney, Lance} }

\end{document}